%% file: modify.tex
\documentclass[a4paper]{amsart}
\usepackage[margin={120pt, 120pt}, centering]{geometry}
\usepackage[utf8]{inputenc}
\usepackage[T1]{fontenc}
\usepackage{textcomp}
\usepackage{amsmath, amssymb}
\usepackage{amsthm}
\usepackage{mathrsfs}
\usepackage{esint}
\usepackage{float}
\usepackage{import}
\usepackage{xifthen}
\usepackage{pdfpages}
\usepackage{transparent}
\usepackage{hyperref}
\usepackage{subcaption}
\usepackage{caption}
\usepackage[most]{tcolorbox}
\usepackage[shortlabels]{enumitem}

\newtheorem{problem}{Problem}[section]
\newtheorem{theorem}[problem]{Theorem}
\newtheorem{proposition}[problem]{Proposition}
\newtheorem{corollary}[problem]{Corollary}
\newtheorem{example}[problem]{Example}

\newtheorem{lemma}[problem]{Lemma}
\newtheorem{remark}[problem]{Remark}
\numberwithin{equation}{section}

\theoremstyle{definition}
\newtheorem{definition}[problem]{Definition}
\newtcbox{\redbox}{colback=red!5!white,colframe=red!75!black, tcbox raise base}

\begin{document}
\input{fpage2.tex}

\section{Introduction}%
\label{sec:introduction}

As a one-dimensional singular version of surfaces, the networks always attract a lot of researchers. In the study of minimal surfaces and mean curvature flows, especially in the case with singularities, networks are always considered first since they are simple enough for analysis but still could give some useful results. Probably one of the famous stories is a group of undergraduate students \cite{foisy1993standard} gave the proof of the 2-dimensional analog of the double bubble conjecture. Later on, F. Morgan and W. Wichiramala \cite{morgan2002standard} extended their result and showed the unique stable double bubble in $\mathbb{R}^2 $ is the standard double bubble.
There is also extensive research on planar clusters.
See for example \cite{cox1994shortest, morgan1998wulff, heppes2005planar, wichiramala2004proof}. In particular, one can also find the nice and elementary introduction of planar clusters in \cite{morgan2016geometric,morgan2018space}.

Similarly, since the mean curvature flows of surfaces with junctions are hard, there is also a lot of work regarding the one-dimensional mean curvature flows for networks, namely network flows. See for instance \cite{schulze2020local,mantegazza2004motion,bronsard1993three,ilmanen2014short,tonegawa2016blow} for some results of mean curvature flow with triple junctions and network flows.

Besides that, the Morse index is an important concept for minimal surfaces.
In general, it is pretty hard to calculate the index of a general minimal surface in general ambient spaces, especially the minimal surfaces generated by Almgren-Pitts min-max theory, see for instance \cite{urbano1990minimal,fraser2007index,marques2021morse,li2016existence,guang2021min} (including the free boundary minimal surfaces case). 

After the work \cite{wang2022curvature}, we are interested in other aspects of minimal multiple junction surfaces, especially the Morse index of those surfaces. Unfortunately, we might need to do a lot of analysis of elliptic operators on the multiple junction surfaces, especially the regularity. So before that, we can try to consider calculating the index of one-dimensional multiple junction surfaces to see if we can find a good method to compute the index.
In this paper, we will focus on the Morse index and nullity of triple junction networks, especially when they are minimally embedded in $\mathbb{S}^2$. 

Before the computation, we need the definitions of stationary networks on $\mathbb{S}^2$. The precise definitions and related notation will be described in Section \ref{sec:prelimiarly}. Here we give a quick introduction so we can state the main theorem.
At first, we will need to calculate its first and second variations.
This will lead to the function spaces and stability operators we are interested in on this network. 
Indeed, similar things have been done in my previous work \cite{wang2022curvature}.
So for the functions defined on the network, we still need a compatibility condition (see Subsection \ref{sub:some_useful_function_spaces} for detailed function spaces) to make sure it can arise from the variation on $\mathbb{S}^2$. 
The stability operator just has form
\[
	Lu=\Delta u+u
\]
by the second variation of networks.
So this operator is closely related to the Laplacian operator on the network. 
This is one difference with the case of the minimal surfaces since there are usually more terms related to the second fundamental form of surfaces in the stability operators.

After defining the Jacobi operator $L$ on geodesic networks, we can talk about its eigenvalues and eigenfunctions. Standard methods and results about the spectrum of elliptic operators from PDEs can directly be applied to the operator $L$ here, especially the min-max characterization of eigenvalues. Until now, we can define the (Morse) index and nullity of the network in the sphere. 

Now we can state our first main theorem of index and nullity of embedded stationary triple junction networks in $\mathbb{S}^2$ (see Theorem \ref{thm_index_and_nullity_of_triple_junction_networks} for more precise statement).
\begin{theorem}
	[Index and nullity of stationary triple junction networks]
	The Morse index of all embedded closed stationary triple junction networks in $\mathbb{S}^2$ is $F-1$ where $F$ is the number of regions on the sphere cut by this network. 
	The corresponding eigenfunctions are all locally constant (explained in Theorem \ref{thm_index_and_nullity_of_triple_junction_networks}) with eigenvalue $-1$.

	The nullity of all embedded closed stationary triple junction networks in $\mathbb{S}^2$ is 3 and the corresponding eigenfunctions are generated by the rotations on $\mathbb{S}^2$ with eigenvalue $0$.
	\label{thm_index_and_nullity_of_triple_junction_networks_short}
\end{theorem}

In other words, this theorem says that there is a gap between the first two eigenvalues (without multiplicity) of the Laplacian operator. Recall that the outstanding open problem from S.T.Yau \cite{yau1982seminar} says the first non-trivial eigenvalue of the Laplacian on a closed embedded minimal hypersurface in $\mathbb{S}^n$ is $n$. This conjecture is far from solved (see for instance \cite{tang2013isoparametric,zhu2017minimal}).
Our theorem says that Yau's conjecture is true for one-dimensional stationary triple junction networks under our scenarios of the above theorem. So it will give partial results related to Yau's conjecture.

	The proof of Theorem \ref{thm_index_and_nullity_of_triple_junction_networks_short} requires the classification of embedded closed stationary triple junction networks in $\mathbb{S}^2$.
	If we do not assume the networks to have junction order at most three at each junction point, then there are infinitely many stationary embedded networks in $\mathbb{S}^2$, although they are not area-minimizing locally in the sense of sets. 

% Yau's conjecture.
%%So we have Yau's conjecture for the case of the network.
%%\begin{conjecture}
%%	For every embedded stationary network in $\mathbb{S}^2$, the first non-trivial eigenvalue is 1.
%%\end{conjecture}

Usually, it is relatively hard to direct compute the networks' index and nullity just based on the definitions. Inspired by the work of H. Tran \cite{tran2020index}, we can consider dividing the big network into several small parts and see if we can calculate the index and nullity through the index and nullity of a small network and a suitable Dirichlet-to-Neumann map.
It turns out this method works for the networks. So this part is the central part of this paper. It will provide a vital tool in the calculation, although we need a bit more space to give the definition of the Dirichlet-to-Neumann map for a network with boundary. 
The eigenvalues of the Dirichlet-to-Neumann map are known as the Steklov eigenvalues. 
There is also quite a lot of research related to the Steklov eigenvalues, and they have a relationship with minimal surfaces, too.
For example, Fraser and Schoen \cite{fraser2012minimal,fraser2016sharp} gave the relation between extremal Steklov eigenvalues and free boundary minimal surfaces in a ball.

Now suppose we cut a big network $\mathcal{N}$ into small pieces $\left\{ \mathcal{N}_i \right\}_{i=1}^l$ along a finite point set $\overline{\mathcal{P}}$. Then we can define a new Dirichlet-to-Neumann map $\overline{T}$ on set $\overline{\mathcal{P}}$ using the Dirichlet-to-Neumann map constructed on small networks. The precise definition is given in Section \ref{sec:index_theorem_for_networks}. Using this Dirichlet-to-Neumann map, we have the following index and nullity theorem to calculate the index and nullity for a big network (see Theorem \ref{thm_index_theorem_for_network} and Theorem \ref{thm_nullity_theorem_for_network} for details).
\begin{theorem}
	[Index and nullity theorem]
	Suppose $\mathcal{N}$ is a network with partition $\left\{ \mathcal{N}_i \right\}_{i=1}^l$.
	Then the index and nullity of $\mathcal{N}$ can be computed as follows,
	\begin{align*}
		\mathrm{Ind}(\mathcal{N})={} & \sum_{i =1}^{l}\mathrm{Ind}(\mathcal{N}_i)+\mathrm{Ind}(\overline{T})+\mathrm{dim}(F_1), \\
		\mathrm{Nul}(\mathcal{N})={} & \mathrm{Nul}(\overline{T})+\mathrm{dim}(F_2)+\sum_{i =1}^{l}\mathrm{dim}(\mathcal{I}_0(\tilde{Q}_i)).
	\end{align*}
	Here, the linear spaces $F_1,F_2$, and $\mathcal{I}_0(\tilde{Q}_i)$ are defined in Section \ref{sec:index_theorem_for_networks}.
	%In particular, the following identity holds,
	%\[
	%	\mathrm{Ind}(\mathcal{N})+\mathrm{Nul}(\mathcal{N})=\sum_{i =1}^{l}(\mathrm{Ind}(\mathcal{N}_i)+\mathrm{Nul}(\mathcal{N}_i))+\mathrm{Ind}(\overline{T})+\mathrm{Nul}(\overline{T}).
	%\]
	\label{thm_index_and_nullity_theorem_short}
\end{theorem}

Although the key thought comes from the work \cite{tran2020index}, there are still some differences from the cases of the free boundary surfaces. The first thing is, the properties of elliptic operators on the network are quite different from the usual elliptic operators on surfaces. So we need to pay more attention to how to define the Dirichlet-to-Neumann map for each small network. Then we can try to construct a global Dirichlet-to-Neumann map $\overline{T}$ from the Dirichlet-to-Neumann on the small networks along the junctions. In the definition of $\overline{T}$, we need to make sure $\overline{T}$ can agree with the usual Dirichlet-to-Neumann on subnetworks and also require every function in the image of $\overline{T}$ will still lie the function space we care about. So we will define quite a lot of spaces to make sure $\overline{T}$ is well-defined.

As a corollary of Theorem \ref{thm_index_and_nullity_theorem_short} (see Corollary \ref{cor:index_estimate}), we know the number of eigenvalues not greater than 0 (count with multiplicity) of a big network is just the sum of the number of eigenvalues not greater than 0 of the small networks and of the Dirichlet-to-Neumann $\overline{T}$.
This corollary will help us to calculate the index and nullity of complex networks on $\mathbb{S}^2$.

\begin{remark}
	This theorem can be extended to higher dimensions for surface clusters like the minimal multiple junction surfaces mentioned in \cite{wang2022curvature}. See \cite{thesis} for details. %This would help us to calculate the Morse index for some special minimal triple junction surfaces.
	%We will investigate this direction in future work.
\end{remark}

In Section \ref{sec:prelimiarly}, we will give the basic definitions of networks in $\mathbb{S}^2$ and their indices and nullities. Then we will talk about the eigenvalues and eigenfunctions from the aspect of elliptic partial differential equations in Section \ref{sec:spectrum_on_networks_and_dirichlet_to_neumann_map}, including the definition of Dirichlet-to-Neumann map on networks with boundary. In Section \ref{sec:index_theorem_for_networks}, we will give the new Dirichlet-to-Neumann map for the partition of networks and prove the main Theorem \ref{thm_index_and_nullity_theorem_short}.
In the last section, we will calculate the index and nullity of embedded closed stationary triple junction networks to finish the proof of Theorem \ref{thm_index_and_nullity_of_triple_junction_networks_short}.

\section{Preliminaries}%
\label{sec:prelimiarly}

\subsection{Curves and networks in $\mathbb S^2$}%
\label{sub:definition_of_networks}

We say $\gamma:I\rightarrow \mathbb{S}^2$ is a (smooth) \textit{curve} if $\gamma$ is a smooth map from an interval $I$ to $\mathbb{S}^2$. Sometimes, we call a subset of $\mathbb{S}^2$ a curve if this subset is an image of a curve $\gamma$.

We make the following definitions and assumptions for the curves we are interested in.
\begin{itemize}
	\item A curve $\gamma$ is called \textit{regular} if $\gamma'(t)\neq 0$ for every $t \in I$. When we talk about a curve in this paper, we will always assume it is regular.
	\item We define $\xi(t) =\frac{\gamma'(t)}{\left|\gamma'(t)\right|}$ to be the \textit{unit tangent vector field} along regular curve $\gamma$.
\item We write $\nu(t)$ to be the \textit{unit normal vector field} along a curve $\gamma$. 
\item Although $I$ can be unbounded, we will only consider the case when $I$ is a bounded closed interval. Moreover, we can assume $I=[0,1]$ and $\gamma$ has constant speed after reparametrization.
\item We write $\tau$ as the \textit{unit outer normal} of $\gamma$. That means $\tau$ is only defined at $\gamma(0)$ and $\gamma(1)$ with $\tau(\gamma(0))=-\xi(0), \tau(\gamma(1))=\xi(1)$.
\item We will assume all curves in this paper do not have self-intersection.
	That is $\gamma(t_1)\neq \gamma(t_2)$ for any $0\le t_1<t_2\le 1$.
\item Given a regular curve $\gamma(t):I\rightarrow \mathbb{S}^2$, we choose $s$ to be an arc length parameterization of $\gamma$. So $ds:=|\gamma'(t)|dt$ is the line element of $\gamma$.
	When we have a function $f$ defined along curve $\gamma$, the integration of $f$ along curve $\gamma$ is defined as
	\[
		\int_{ \gamma} f ds:=\int_{0}^{1} f(\gamma(t))ds(t)=\int_{0}^{1} f(\gamma(t))|\gamma'(t)|dt.
	\]
	\item The \textit{Length} of $\gamma$, written as $L(\gamma)$, is defined as
	\[
		L(\gamma):=\int_0^1 \left|\gamma'(t
		)\right|dt=\int_{ \gamma} ds.
	\]
	Since $\gamma$ has constant speed, we have $|\gamma'(t)|=L(\gamma)$ for $t \in [0,1]$.

\item Given a function $f$ defined along a regular curve $\gamma$, we define the gradient of $f$ along $\gamma$ as
	\[
		\nabla ^{\gamma} f = \frac{d f}{d s} \frac{d \gamma}{d s}= \frac{d f(\gamma(t))}{dt} \frac{\gamma'(t)}{L^2(\gamma)}.
	\]
	
	The Laplacian of $f$ along $\gamma$ is defined as
	\[
		\Delta_{\gamma}f:=\frac{d^2f}{ds^2}.
	\]

\end{itemize}

\begin{definition}
	A \textit{network} in $\mathbb{S}^2$ with interior endpoints set $\mathcal{P}$ and boundary endpoints set $\mathcal{Q}$ is defined as a collection of finite curves $\{ \gamma^1,\cdots , \gamma^n\}$ in $\mathbb{S}^2$ such that
\begin{itemize}
		\item $\mathcal{P},\mathcal{Q}\subset \mathbb{S}^2$ are two finite sets of points such that $\mathcal{P} \cap \mathcal{Q}
			=\emptyset $.
		\item $\mathcal{P} \cup \mathcal{Q}$ is the set of all possible (distinct) endpoints of $\gamma^i$ for $i=1,\cdots ,n$. That is $\mathcal{P} \cup  \mathcal{Q}= \bigcup_{i=1}^{n}\{ \gamma^i(0), \gamma^i(1)\}$.
	\end{itemize}
\end{definition}

We say a network $\mathcal{N}$ is \textit{embedded} if $\gamma^i$ can only intersect each other at their endpoints. 
We will always assume $\mathcal{N}$ is embedded in $\mathbb{S}^2$ unless we say it is immersed.

For any $P \in \mathcal{P}$, we write $\{ \gamma^{j} _{P}\} _{j=1,\cdots ,j_P}$ as the collection of $\gamma^{i}$ having $P$ as its endpoint. We write $\tau^{j}_{P}$ as the exterior unit tangent vector of $\gamma^{j}_{P}$ for $j=1,\cdots ,j_P$. For example, if $\gamma^{j}_{P}(0)=P$, then $\tau^j_{P}=-\frac{(\gamma^j_{P})'(0)}{\left|(\gamma^j_{P})'(0)\right|}$.

Similarly, for any $Q \in \mathcal{Q}$, we write $\{ \gamma^k_Q \}_{k=1,\cdots , k_Q}$ as the collection of $\gamma^i$ having $Q$ as its endpoint and use $\tau^k_Q$ to denote the exterior unit tangent vector of $\gamma_Q^k$.
In most cases of this paper, $k_Q$ will be one, and we write $\gamma_Q= \gamma^1_Q$ for short.

To avoid ambiguity, we suppose $\tau^j_{P}$ $(\tau^k_Q)$ point in different directions for different values of $j$ ($k$).

We say a network $\mathcal{N}$ is \textit{closed} if $Q=\emptyset $.

\begin{remark}
	Of course, we can define the networks in the general manifold or even define the network intrinsically by identifying their endpoints. 
	We can see most of the arguments can go through even if for the intrinsic definitions. 
\end{remark}

\begin{definition}
	We say $\mathcal{N}$ is a \textit{stationary network} in $\mathbb{S}^2$ if the following hold.
	\begin{itemize}
		\item Each $\gamma^{i}$ is an arc of a great circle.
		\item For any $P$ in $\mathcal{P}$, we have 
			\[
				\sum_{j =1}^{j_P}
				\tau_{P}^j=0.
			\]
	\end{itemize}
\end{definition}

\begin{remark}
	If $\mathcal{N}$ is stationary, then we know all of $j_P\ge 2$ by the second condition in the definition.
\end{remark}

\begin{remark}
	We can also consider the network with density $\theta _i$ at each $\gamma^i$ like \cite{wang2022curvature}.
By imposing appropriate densities, we can see that the one skeleton of the spherical truncated icosahedron (soccer ball) can be stationary. So it is possible to calculate the index and nullity for this kind of network.
\end{remark}

By a simple geometric observation, we know for the stationary network $\mathcal{N}$, the following two special cases hold. Let $P \in \mathcal{P}$, then 
\begin{itemize}
	\item If $P$ is a double endpoint, i.e. $j_P=2$, then $\tau^1_{P}=-\tau^2_{P}$ as a vector.
	\item If $P$ is a triple endpoint, i.e. $j_P=3$, then $\tau^1_{P}+\tau^2_{P}+ \tau^3_{P}=0$ as a vector.
\end{itemize}

We can define the total length of $\mathcal{N}$, written as $L(\mathcal{N})$, as the sum of all $\gamma^{i}$. That is
\[
	L(\mathcal{N}):= \sum_{i =1}^{n}L(\gamma^{i}).
\]

Like the minimal surfaces, we have the first variation formula for the length of $\mathcal{N}$. Suppose $\mathcal{N}_t=\{ \gamma^i_t \}_{i=1}^n$ is a smooth variation of $\mathcal{N}$ fixing $\mathcal{Q}$, then
\[
	\left.\frac{d}{dt}\right|_{t=0}
	L(\mathcal{N}_t)=\sum_{i =1}^{n}
	\int_{\gamma^{i}}
	\kappa^{i}\nu^{i}\cdot X ds^{i}+ \sum_{P \in \mathcal{P}}\sum_{j =1}^{j_P}
	X\cdot \tau_{P}^{j},
\]
where $\kappa^i$ is the geodesic curvature of $\gamma^{i}$ in $\mathbb{S}^2$ with respect to $\nu^i$, $\nu^{i}=\nu^{\gamma^i}$, the unit normal vector field along $\gamma^{i}$, $X$, a variational vector field associated with variation of $\mathcal{N}$.

From the first variation formula, we can see that $\mathcal{N}$ is a stationary network in $\mathbb{S}^2$ if and only if it is a critical point of length for any suitable variation $\mathcal{N}_t$.

Similarly, we can compute the second variation formula. Suppose $\mathcal{N}$ is a stationary network. If we write $\phi^{i}=X\cdot \nu^{i}$, then the second derivative of length can be computed by 
\[
	\left.\frac{d^{2}}{dt^{2}}\right|_{t=0}
	L(\mathcal{N}_t)=\sum_{i= 1}^{n}
	\int_{\gamma^{i}} \left|\nabla^{\gamma^i}
	\phi^{i}\right|^2-\left|\phi^{i}\right|^2
	ds^i,
\]
where $\nabla^{\gamma^{i}}$ is the gradient on $\gamma^{i}$. 

So we can say $\mathcal{N}$ is \textit{stable} if $\left.\frac{d^{2}}{dt^{2}}\right|_{t=0}L(\mathcal{N}_t)\ge 0$ for any variation $\mathcal{N}_t$ fixing $\mathcal{Q}$.

\subsection{Some useful function spaces}%
\label{sub:some_useful_function_spaces}

Based on the properties of $\phi^{i}$, we define the following function spaces on $\mathcal{N}$. 

\begin{definition}
	When we say $\phi$ is a \textit{function} on $\mathcal{N}$, we mean $\phi$ is a tuple $(\phi^1, \cdots ,\phi^n)$ such that each $\phi^i$ is a function defined on $\gamma^i$.

	We say a function $\phi$ defined on $\mathcal{N}$ is in the function space $C^k(\mathcal{N})$ ($k \in \mathbb{N} \cup \{ \infty \}$) if each $\phi^i$ is of class $C^k$.
\end{definition}

Similarly, we have the Sobolev space on $\mathcal{N}$ defined as follows.

\begin{definition}
	We say a function $\phi$ defined on $\mathcal{N}$ is in the function space $W^{k,p}(\mathcal{N})$ if each $\phi^i$ is of class $W^{k,p}$ on $\gamma^{i}$.
\end{definition}

Now suppose we have a $C^k$ (tangential) vector field on $\mathbb{S}^2$, then we can define a $C^k(\mathcal{N})$ function by defining $\phi$ by taking $\phi^i= X\cdot \nu^i$. 
We will write this function $\phi= X \cdot\nu$ for short.
Besides, for any $C^k$ function $f$ on $\mathbb{S}^2$, we know the function $\phi$ defined by $\phi^i:=f|_{\gamma^i}$ is a $C^k(\mathcal{N})$ function. 
So we know $C^k(\mathcal{N})$ function space is quite large. Since we are interested in the functions of the form $\phi=X\cdot\nu$, let us define a subspace of $C^k(\mathcal{N})$ in the following way.
\begin{definition}
	We say a function $\phi \in C^k_1(\mathcal{N})$ if for any $P \in \mathcal{P}$, there exists a vector $X \in T_{P}\mathbb{S} ^2$ such that $\phi^j_P= X \cdot \nu ^{\gamma^{j}_{P}}$ at $P$. 
	Here we use $\phi^j_P:=\phi^i$ such that $\gamma^i=\gamma^j_P$.
\end{definition}

\begin{remark}
	The main reason why we are interested in the space $C_1^k(\mathcal{N})$ is, many functions in this space come from variations induced by $C^k$ vector fields $X$ on the sphere, and only the normal components of such vector fields along the curves $\gamma^i$ appear in the second variation formula.
	Indeed, there are some other important function spaces on a minimal triple junction hypersurface. See \cite{thesis} for more details.
\end{remark}

Similarly, we have a version like Sobolev Spaces.
\begin{definition}
	We say a function $\phi \in W^{k,p}_1(\mathcal{N})$, $k\ge 1$ if for any $P\in \mathcal{P}$, there exists a vector $X \in T_{P}\mathbb{S}^2$ such that $\phi_P^j= X \cdot \nu ^{\gamma^{j}_{P}}$ at $P$. Here the restriction of $\phi^j_P$ at $P$ should be understood in the trace sense.
\end{definition}

We can easily check $W^{k,p}_1(\mathcal{N})$ is a closed subspace of $W^{k,p}(\gamma^1)\times \cdots \times W^{k,p}(\gamma^n)$ with product norm, so $W^{k,p}_1(\mathcal{N})$ is a Banach space.

Here we introduce other notation to present these function spaces.
For any $P \in \mathcal{P}$, we define $V(P):= \mathbb{R} ^{j_P}$ as a vector space with $\mathrm{dim}= j_P$. Note that when given a function $\phi \in C^{k}(\mathcal{N})$, the restriction of $\phi$ on $P$ is just a tuple $(\gamma^1_P|_P,\cdots , \gamma^{j_P}_P|_P)$, which can be viewed as an element in $V(P)$.
So $V(P)$ is just like a function space on $P$, which comes from the restriction of functions on $\mathcal{N}$.
Now we define $V_1(P):=\{ (v_1,\cdots ,v_{j_P}) \in V(P): v_j=X\cdot \nu^j_P,\quad  \forall 1\le j\le j_P \text{ for some }X \in T_P\mathbb{S}^2\}$. $V_1(P)$ is a subspace of $V(P)$ with dimension at most $2$.

Then we can define $V(\mathring{\mathcal{P}}):= \bigoplus_{P \in \mathring{\mathcal{P}}}V(P), V_1(\mathring{\mathcal{P}}):=\bigoplus_{P \in \mathring{\mathcal{P}}}V_1(P)$ for any $\mathring{\mathcal{P}} \subset \mathcal{P}$.
Then we know a function $\phi \in C^{k}(\mathcal{N})$ $( \phi \in W^{k,p}(\mathcal{N}) )$ is in the space $C^{k}_1(\mathcal{N})$ $(\phi \in W^{k,p}_1(\mathcal{N}))$ if and only if $\phi|_{\mathcal{P}} \in V_1(\mathcal{P})$. 

Besides, we introduce another subspace of $V(\mathcal{P})$ that we will use later on.
We define $V_2(P):= V_1^\bot (P)$ for any $P \in \mathcal{P}$ and write $V_2(\mathring{\mathcal{P}}):=\bigoplus _{P \in \mathring{\mathcal{P}}}V_2(P)$ for any $\mathring{\mathcal{P}}\subset \mathcal{P}$ where the inner product on $V(\mathcal{P} _1)$ is the standard Euclidean inner product. 

Another helpful function space is the space of trace zero functions. 
We define $W^{k,p}_0(\mathcal{N})$ as the trace zero function space. That is, $W^{k,p}_0(\mathcal{N}):= \{ \phi \in W^{k,p}(\mathcal{N}): \phi = 0 \text{ on }\mathcal{Q}\}$.
We will write $W^{k,p}_{0,1}(\mathcal{N}):= W^{k,p}_0(\mathcal{N})\cap W^{k,p}_1(\mathcal{N})$ for short.

Similarly, we write $C^k_0(\mathcal{N})$ as the space of functions in $C^k(\mathcal{N})$ with zero boundary values. We write $C^{k}_{0,1}(\mathcal{N}):= C^{k}_0(\mathcal{N})\cap C^k_1(\mathcal{N})$.

Note that in the above, when we write $\phi=0$ on $\mathcal{Q}$, we actually mean $\phi^{k_i}_Q=0$ at $Q$ for $i=1,\cdots ,k_Q$ for all $Q \in \mathcal{Q}$. This means when we restrict $\phi$ on $\mathcal{Q}$, this restriction is just like a vector with dimension $\sum_{Q \in \mathcal{Q} }^{} k_Q$. For convenience, we write $V(\mathcal{Q}):= \bigoplus _{Q \in \mathcal{Q}}\mathbb{R} ^{k_Q}$ as the function space defined on $\mathcal{Q}$.

From these definitions, we know the function $X\cdot \nu$ is a $C^k_1(\mathcal{N})$ function if $X$ is a $C^{k}$ vector field on $\mathbb{S}^2$.

\begin{remark}
	In general, given $\phi \in C_1^k( \mathcal{N})$, one may not find a $C^k$ vector field $X$ on $\mathbb{S}^2$ such that $\phi^i=X \cdot \nu^{i}$. 

	In the previous definitions of $C^k_1(\mathcal{N})$ and $W^{k,p}_1(\mathcal{N})$, for $\mathcal{N}$ being a stationary network, if $P$ is a double endpoint or a triple endpoint, and the choice of $\nu^j$ on $\gamma^j$ satisfies $\sum_{j =1}^{j_P} \nu^{\gamma^j_{P}}=0$, then the condition for $\phi$ can be simplified by
	\[
		\sum_{j =1}^{j_P}
		\phi^j_P=0, \quad j_P=2 \text{ or }j_P=3.
	\]
\end{remark}

Following the index form for minimal surfaces, we can define the index form for $W^{1,2}_{0,1 }(\mathcal{N})$ functions. Let $\mathcal{N}$ be a stationary network in $\mathbb{S}^2$. From the second variation formula, we can define the bilinear form, called the index form, by the following.

\begin{definition}
	
	\label{def_indexForm}
For any $\phi,\psi \in W^{1,2}_{0,1}( \mathcal{N})$, we define the \textit{index form} by,
\[
	Q(\phi,\psi):= \sum_{i =1}^{n}\int_{\gamma^{i}} 
	\nabla^{\gamma^i}\phi^i \cdot \nabla^{\gamma^i}\psi^i
	-\phi^i \psi^ids^{i}.
\]

\end{definition}

To simplify our formulas, we use the following notation. For any $\phi \in W^{k,p}_1(\mathcal{N})$ (or $C_1^k(\mathcal{N})$),
\begin{align*}
	\int_{\mathcal{N}} \phi:={} & 
	\sum_{i =1}^{n}\int_{\gamma^i} 
	\phi^i ds^i,\\
	\int_{\mathcal{P}} \phi:={} & 
	\sum_{P \in \mathcal{P}}\sum_{j =1}^{j_P}
	\left.\phi^j_P\right|_{P}.
\end{align*}
The integral of $\phi$ on $\mathcal{P}$ is just the summation of all $\phi^{i}$ taking values at endpoints which are in $\mathcal{P}$.

So the index form can be written as
\[
	Q(\phi,\psi)= \int_{\mathcal{N}} 
	\nabla \phi \cdot\nabla \psi- \phi \psi=
	-\int_{\mathcal{N}} \phi\left( \Delta \psi+\psi\right) + \int_{\mathcal{P}} 
	\phi\frac{\partial }{\partial \tau}\psi,
\]
where we have assumed $\phi, \psi \in C^k_{0,1}(\mathcal{N})$ for $k\ge 2$ for the last identity.
Besides, with the above notation, $\int_{\mathcal{P}} \phi \frac{\partial }{\partial \tau}\psi$ should be understood as
\[
	\sum_{P \in \mathcal{P}
	}\sum_{j =1}^{j_P}
\left.\phi^j_P\right|_{P}
\frac{\partial }{\partial \tau^{j}_{P}}
\psi^j_P.
\]

Associated with the index form $Q$, we can define the Jacobi operator for stationary networks in $\mathbb{S}^2$ as follows.
\begin{definition}
	For any $\phi \in C^{k}_1(\mathcal{N})$ for $k\ge 2$, we can define the \textit{Jacobi operator} $L \phi$ on network $\mathcal{N}$ by 
\[
	(L \phi)^i:=\Delta_{\gamma^i}\phi^i
	+\phi^i.\quad (\text{Or }
	L \phi=\Delta \phi+\phi \text{ for short.})
\]
\label{def_JacobiOp}
\end{definition}

Then, we can rewrite $Q$ as
\begin{equation}
	Q(\phi,\psi)=-\int_{ \mathcal{N}} \phi L \psi +\int_{ \mathcal{P}} \phi \frac{\partial }{\partial \tau}\psi.
	\label{eq:bilinearForm}
\end{equation}

\begin{remark}
	The Laplacian in the index form is just the usual second variation of a function with arc-length parameterization. One can try to define the index form for some particular clusters of minimal surfaces like minimal multiple junction surfaces defined in \cite{wang2022curvature}.
\end{remark}

Now we can define the Morse index and nullity of $Q$ for a stationary network $\mathcal{N}$.

\begin{definition}
	The Morse index of a stationary network $\mathcal{N}$, written as $\mathrm{Ind}(\mathcal{N})$, is defined to be the maximal dimension of a subspace of $ W^{1,2}_{0,1}(\mathcal{N})$ which the quadratic form $Q$ is negative definite.
%We write this maximal subspace as $\mathcal{J}_0^-(\mathcal{N})$.

	Similarly, the nullity of $\mathcal{N}$, written as $\mathrm{Nul}(\mathcal{N})$, is defined as the maximal dimension of a subspace of $W^{1,2}_{0,1}(\mathcal{N})$ which the quadratic form $Q$ vanishes.
%We write this maximal subspace as $\mathcal{J}_0^0(\mathcal{N})$.
\end{definition}

Considering the case of the usual symmetric elliptic operator, we can define the spectrum of an operator $L$.
\begin{definition}
	We say $\lambda \in \mathbb{R} $ is a weak eigenvalue of $L$ if there is a nonzero function $\phi \in W^{1,2}_{0,1}(\mathcal{N})$ such that 
	\[
		Q(\phi,\psi)=\lambda \int_{ 
			\mathcal{N}
		} \phi\psi
	\]
	for all $\psi \in W^{1,2}_{0,1} (\mathcal{N})$. The function $\phi$ is an eigenfunction of $L$ of eigenvalue $\lambda$.
\end{definition}

In general, we need to develop the regularity theorems to show the eigenfunction $\phi \in C^\infty_{0,1}
(\mathcal{N})$. But for the 1-dimensional case, we can easily show $\phi \in C^\infty_{0,1}(\mathcal{N})$. Indeed, if $\phi$ is an eigenfunction, then on each curve $\gamma^i$, $\phi^i$ is H\"oder continuous by Morrey's inequality. So at least we know $\phi \in C^0_{0,1} (\mathcal{N} )$. Note that $\phi^i$ satisfies $-\Delta u-u=\lambda u$ weakly on $\gamma^i$ with Dirichlet condition $u=\phi^i$ at $\gamma^i(0)$ and $\gamma^i(1)$, so $\phi^i$ is a smooth function up to the boundary by the usual regularity theorem for the Dirichlet problems. Hence $\phi \in C^\infty_{0,1}(\mathcal{N})$. Now we need to introduce another condition on function spaces to describe the properties of eigenfunctions.
We say a function $\phi \in C^k_1 (\mathcal{N})$, $k\ge 1$, (or $\phi \in W^{k,p}_1(\mathcal{N}), k\ge 2$) satisfies the compatibility condition at $P$ if the following equations hold 
\[
	\sum_{j=1 }^{j_P}X \cdot \nu^{\gamma^j_{P}}
	\frac{\partial }{\partial \tau^j _P}\phi^j_{P}=0 \text{ for all }
	X \in T_P\mathbb{S}^2
\text{ at } P.
\]

We say a function $\phi \in C^k_1 (\mathcal{N})$, $k\ge 1$, (or $\phi \in W^{k,p}_1(\mathcal{N}), k\ge 2 $) satisfies the \textit{compatibility condition} if $\phi$ satisfies the compatibility condition at $P$ for every $P \in \mathcal{P}$.

Note that $X \cdot \nu ^{\gamma^j_P} \in V_1(P)$, then $\phi$ satisfies the compatibility condition at $P$ if and only if $(\frac{\partial }{\partial \tau_P^j} \phi^j_P)_{j=1}^{j_P}\in V_2(P)$. So $\phi$ satisfies the compatibility condition on $\mathcal{P}$ if and only if $\frac{\partial }{\partial \tau} \phi \in V_2(\mathcal{P})$ in short notation. In some special cases, this compatibility condition will become much simple, which will help us to understand what this condition says. If at $P$, we have $j_P=2$ or $3$(double or triple junction) and $\sum_{j =1}^{j_P}\nu ^{\gamma^j_{P}}=0$, then the compatibility condition at $P$ is equivalent to the condition 
\[
	\frac{\partial }{\partial \tau
	^j_{P}}\phi^j _{P}= \frac{\partial }{\partial \tau
	^k_{P}}\phi^k _{P},\quad \forall 1\le j,k\le j_P \text{ at } P.
\]

\begin{remark}
	The reason why we introduce the compatibility condition is following.
	If $\phi$ is a weak eigenfunction of $L$ with eigenvalue $\lambda$, then by \eqref{eq:bilinearForm}, know $\phi$ should satisfy
	\[
		L\phi=0 \text{ on }\mathcal{N}, \text{ and } \psi \frac{\partial \phi}{\partial \tau}=0, \text{ for any } \phi \in V_1(\mathcal{P}).
	\]
	The second condition is exactly the compatibility condition of $\phi$ on $\mathcal{P}$.
\end{remark}

With these conditions, we can describe another definition of eigenvalues and eigenfunctions.
\begin{definition}
	We say $\lambda \in \mathbb{R} $ is an eigenvalue of $L$ if there is a function $\phi \in C^\infty_{0,1} (\mathcal{N})$ which satisfies the compatibility condition on $\mathcal{P}$ such that 
	\[
		-L \phi=\lambda \phi.
	\]
	The function $\phi$ is called the eigenfunction of $L$ with eigenvalue $\lambda$.
\end{definition}

Using the integration by parts and regularity results above, we know $\lambda\ (\phi)$ is an eigenvalue (an eigenfunction) if and only if it is a weak eigenvalue (a weak eigenfunction).

\begin{remark}
	From the definitions in this subsection, we know the function space $C^\infty_{0,1}(\mathcal{N})$ only relies on the choice of $V_1(P)$ for each $P \in \mathcal{P}$. So if we want to define the function space $C^\infty_{0,1}(\mathcal{N})$ for intrinsic networks, we only need to choose a suitable function space $V_1(P)$ for each $P$. 
	This will allow us to define the general elliptic operator acting on space $W^\infty_{0,1} (\mathcal{N})$, which does not rely on the embedding.

	Besides, we can consider the general operator $Lu= \Delta u+ du$ where $d \in C^\infty(\mathcal{N})$. 
	%So we will always use operator $L$ instead of $\Delta+1$ to indicate they can be generalized. 
\end{remark}

\subsection{Examples of eigenfunctions}%
\label{sub:examples_of_eigenfunctions}

Here we give an example of a stationary network. From this example, we will illustrate what the space $C^k_1(\mathcal{N})$ looks like and which conditions that eigenfunctions should satisfy.

Let $\mathcal{N}$ be the one-skeleton of a regular spherical tetrahedron like the one shown in Figure \ref{fig:tetrahedron}.

\begin{figure}[ht]
	\centering
\begin{subfigure}[b]{.4\linewidth}
		\centering 
		\includegraphics[width=0.8\textwidth]{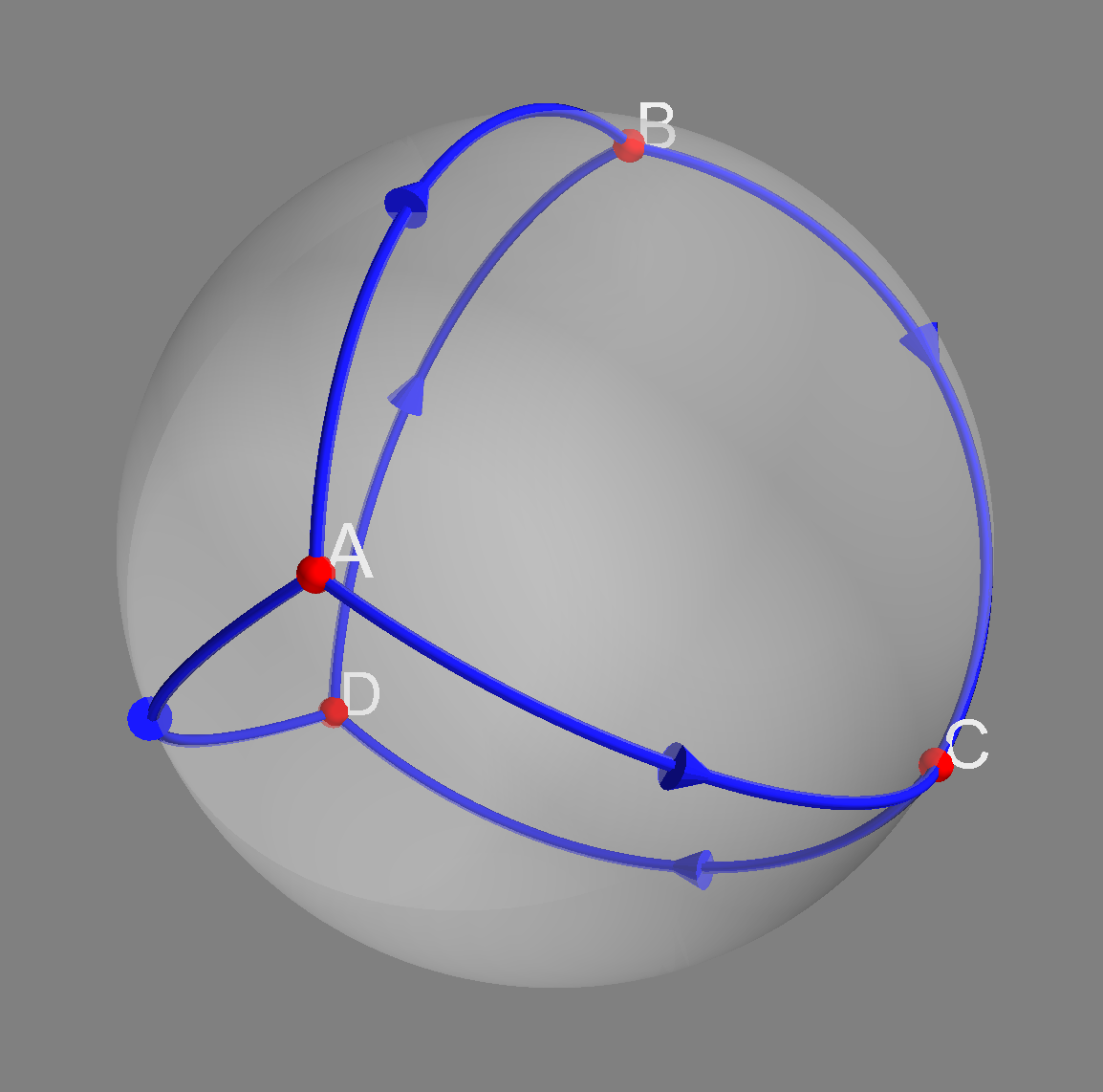}
		\caption{$\mathcal{N}$ on the sphere}
		\label{fig:figures/T2.png}
	\end{subfigure}
	\begin{subfigure}[b]{.5\linewidth}
    \centering 
	\begingroup
	\fontsize{7pt}{12pt}
	\def\svgwidth{0.8\columnwidth}
	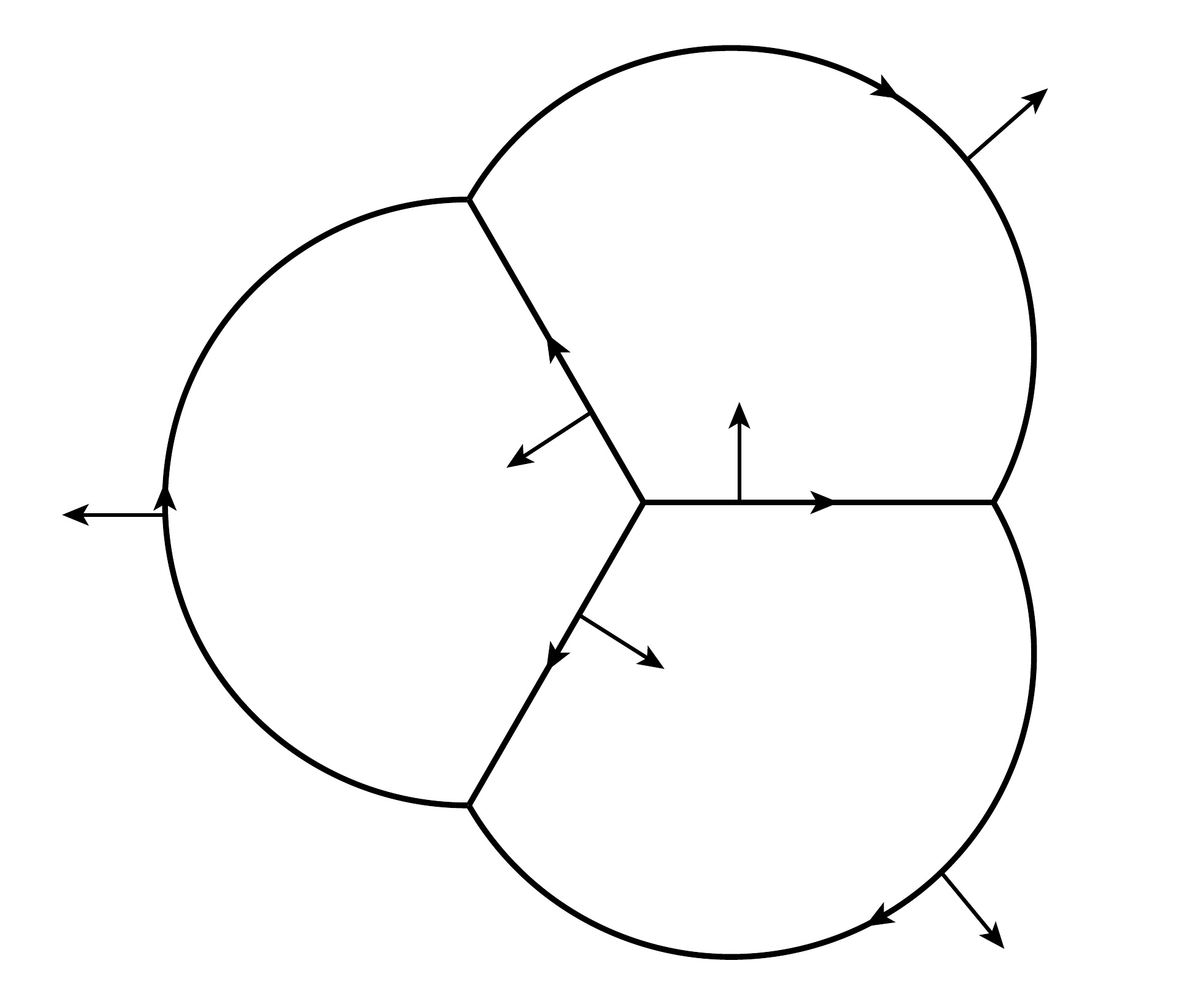
	\endgroup

    \caption{Stereographic projection of $\mathcal{N}$}
	\label{fig:orientation_tegrahedron}
	\end{subfigure}

	\caption{A regular spherical tetrahedron}
	\label{fig:tetrahedron}
\end{figure}

$\mathcal{N}$ has six arcs with orientation shown in Figure \ref{fig:orientation_tegrahedron}. So for a function $\phi \in C^k(\mathcal{N})$, $\phi$ belongs to $C_{1}^k(\mathcal{N})$ if and only if the following holds
\begin{equation}
	\begin{cases}
		\phi^1(0)+\phi^2(0)+\phi^3(0)=0, &  \\
		\phi^1(1)-\phi^5(0)+\phi^6(1)=0, & \\
		\phi^2(1)-\phi^6(0)+\phi^4(1)=0,\\
		\phi^3(1)-\phi^4(0)+\phi^5(1)=0.
	\end{cases}
	\label{eq:boundary}
\end{equation}

Let $l=\arccos (-\frac{1}{3})$ be the length of arcs in $\mathcal{N}$. Then the function $\phi$ becomes an eigenfunction with eigenvalue $\lambda$ if and only if the following holds, 
\begin{equation}
	\begin{cases}
		-\frac{1}{l^2}(\phi^i)''- \phi^i=\lambda \phi^i, & \text{ for }
		i=1,\cdots ,6\\
		(\phi^1)'(0)=(\phi^2)'(0)= (\phi^3)'(0), &\\
		(\phi^1)'(1)=(\phi^5)'(0)= (\phi^6)'(1), & \\
		(\phi^2)'(1)=(\phi^6)'(0)= (\phi^4)'(1),\\
		(\phi^3)'(1)=(\phi^4)'(0)= (\phi^5)'(1).
	\end{cases}
	\label{eq:derivative}
\end{equation}

Here, we have assumed the length of $(\gamma^i)'$ is constant for each $i$.  By minimizing the Rayleigh quotient, we know the lowest eigenvalue should be $-1$, and the corresponding eigenfunction should be constant on each $\gamma^i$.  Suppose $\phi$ is a such function.  This means we suppose $\phi^i=c^i$, which is a constant on $\gamma^i$.  The condition (\ref{eq:derivative}) holds automatically with $\lambda=-1$, and the condition (\ref{eq:boundary}) will become \[ \begin{cases}
		c^1+c^2+c^3=0, &  \\
		c^1-c^5+c^6=0, & \\
		c^2-c^6+c^4=0,\\
		c^3-c^4+c^5=0.
	\end{cases}
\]

This is a system of linear equations, and we can easily find the set of solutions is a subspace of dimension three. 

This example shows, unlike the usual symmetric elliptic operators on surfaces, the principal eigenvalue of $L$ on $\mathcal{N}$ may not be simple. In this case, the dimension of the eigenfunction set with eigenvalue $-1$ is 3.

Another eigenvalue of $L$ we are interested in is $\lambda=0$. The corresponding eigenfunction of it is generated by the rotational on spheres.
For example, $\phi$ defined by
\[
	\begin{cases}
	\phi^1(t)=\phi^2(t)=\phi^3(t)= \sin(lt), &  \\
	\phi^4(t)=\phi^5(t)=\phi^6(t)= \frac{\sin l}{2 \sin \frac{l}{2}}
	\sin \left( \frac{l}{2}-lt \right),
	\end{cases}
\]
will have eigenvalue $0$.

\section{The Spectrum on networks and the Dirichlet-to-Neumann map}%
\label{sec:spectrum_on_networks_and_dirichlet_to_neumann_map}

\subsection{Eigenvalues on networks}%
\label{sub:eigenvalues_on_networks}

Here we will describe the eigenvalues of the operator $L$ where $L=\Delta+d$ for $d \in C^\infty(\mathcal{N})$ from the point of view of symmetric elliptic operators. 
We can define the bilinear form $\mathcal{L}$ on $W^{1,2}_1(\mathcal{N})$ associated with $L$ by 
\begin{equation}
	\mathcal{L}(u,v):=\int_{ \mathcal{N}} \nabla u\cdot \nabla v
	-duv,\quad u,v \in W^{1,2}_1(\mathcal{N}).
	\label{eq:defLBilinearForm}
\end{equation}

In particular, if $d\equiv 1$, then $\mathcal{L}=Q$.

\begin{remark}
	If we also assume $v \in C^k_1(\mathcal{N})$ or $W^{k,2}_1(\mathcal{N})$ for $k\ge 2$ and $v$ satisfies compatibility condition on $\mathcal{P}$, then by integration by parts, we have
	\begin{equation}
		\int_{\mathcal{N}} \nabla u \cdot \nabla v = \int_{\mathcal{Q}} u \frac{\partial v}{\partial \tau}-\int_{ \mathcal{N}} u \Delta v,
		\label{eq:IBPcomp}
	\end{equation}
	since $u\in V_1(\mathcal{P}), \frac{\partial v}{\partial \tau} \in V_2(\mathcal{P})$.
	
	%Indeed, the definition of $\mathcal{L}(u,v)$ makes sense for any $u,v \in W^{1,2}(\mathcal{N})$.
	%Suppose $u,v$ are sufficient regular.
	%After integration by parts, we can rewrite $\mathcal{L}(u,v)$ as
	%\begin{equation}
	%	\mathcal{L}(u,v)=-\int_{ \mathcal{N}} vLu+\int_{ \mathcal{P}\cup \mathcal{Q}} v \frac{\partial u}{\partial \tau}.
	%	\label{eq:defLBilinearIBP}
	%\end{equation}

	In particular, we can rewrite \eqref{eq:defLBilinearForm} as
	\begin{equation}
		\mathcal{L}(u,v)=-\int_{ \mathcal{N}} vLu+\int_{ \mathcal{Q}} v \frac{\partial u}{\partial \tau}. 
		\label{eq:defLBilinearComp}
	\end{equation}
\end{remark}

To make our theorems more general, we will assume all concepts defined before, including eigenvalues, eigenfunctions, indices, and nullities, are all associated with the operator $L$.

We are interested in the following Dirichlet problem 
\begin{equation}
	\begin{cases}
	Lu=f, & \text{ on } \mathcal{N},\\
	u=g, & \text{ on }\mathcal{Q},\\
	u \in V_1(\mathcal{P}), \frac{\partial u}{\partial \tau}
	\in V_2(\mathcal{P}),
 & \text{ on }\mathcal{P},
	\end{cases}
	\label{eq:Dirichlet_problem}
\end{equation}
with $f \in L^2(\mathcal{N}):= \bigoplus_{i=1}^n L^2(\gamma^i)$, $g=(g_Q)_{Q \in \mathcal{Q}}\in V(\mathcal{Q})$.

This problem is equivalent to the following problem
\begin{equation}
	\begin{cases}
	Lw=\hat{f}, & \text{ on }\mathcal{N}, \\
	w=0, & \text{ on }\mathcal{Q},\\
	w \in V_1(\mathcal{P}), \frac{\partial w}{\partial \tau}
	\in V_2(\mathcal{P})
		 & \text{ on }\mathcal{P},
	\end{cases}
	\label{eq:Dirichlet_problem2}
\end{equation}
if we set $w=u-\hat{g}$ where $\hat{g}$ is any smooth extension of $g$ in $C^\infty_1(\mathcal{N})$ satisfying the compatibility condition on $\mathcal{P}$ and $\hat{f}=f-L\hat{g}$.

We can define the weak solutions to the problem (\ref{eq:Dirichlet_problem}) following the theory of PDEs. 

We say $u \in W^{1,2}_{1}(\mathcal{N})$ weakly solves (\ref{eq:Dirichlet_problem}) if the following holds,
\begin{align}
	\begin{cases}
		\mathcal{L}(u,v)= -\int_{ \mathcal{N}} f v, &\quad  \forall v \in 
		W^{1,2}_{0,1}(\mathcal{N}),\\
	w=g, &\quad \text{on }\mathcal{Q} \text{ in the trace sense.}
	\end{cases}
	\label{eq:weak_sol}
\end{align}

If $f$ and $u$ are sufficient regular, then by integration by parts, we know $u$ solves (\ref{eq:Dirichlet_problem}) if and only if $u$ solves (\ref{eq:Dirichlet_problem}) weakly.

Similarly, we say $w \in W^{1,2}_{1,0}(\mathcal{N})$ weakly solves the problem (\ref{eq:Dirichlet_problem2}) if $\mathcal{L}(w,v)=-\int_{ \mathcal{N}} \hat{f}v$ for all $v \in W^{1,2}_{0,1}(\mathcal{N})$.

To show the solvability of the Dirichlet problem (\ref{eq:Dirichlet_problem2}), let us consider $L_\sigma w= Lw-\sigma w$.
The corresponding bilinear form is written as
\[
	\mathcal{L}_\sigma(w,v):= \int_{ \mathcal{N}} \nabla w\cdot \nabla v +(\sigma-d)wv.
\]

This will define an inner product on the Hilbert space $W^{1,2}_{0,1}(\mathcal{N}) $ when $\sigma$ large enough such that $\sigma-d>2$. Hence by the Riesz representation theorem, we know for any $\hat{f} \in L^2(\mathcal{N})$, there is a unique $w \in W^{1,2}_{0,1} (\mathcal{N})$ such that
\begin{equation}
	\mathcal{L}_\sigma(w,v)=
	-\int_{ \mathcal{N}} \hat{f}v, \quad \forall v \in W^{1,2}_{0,1}
	(\mathcal{N}).
	\label{eq:Lsigma}
\end{equation}

So we can define an operator
\[
	L^{-1}_\sigma:L^2(\mathcal{N})\rightarrow W^{1,2}_{0,1}(\mathcal{N}),\quad L^{-1}_\sigma(\hat{f})=w.
\]

If we take $v=w$ in \eqref{eq:Lsigma} and use Cauchy-Schwarz inequality, we can get a standard elliptic estimate for $w$ as
\[
	\|w\|_{W^{1,2}(\mathcal{N})}\le C\|\hat{f}\|_{L^2(\mathcal{N})}.
\]

Hence, the operator $L_\sigma^{-1}$ is a compact operator on $L^2(\mathcal{N})$ by Sobolev Embedding Theorem. 
Following the standard steps in the spectral theory, we know $L$ has discrete eigenvalues $-1\le \lambda_1\le \lambda_2\le \cdots $ forming a sequence that goes to infinity.

Moreover, the eigenvalues of $L$ can be characterized by the min-max principle by the standard steps in the theory of elliptic PDEs in the following way,
\begin{align}
	\lambda_k= \min _{V_k \subset W^{1,2}_{0,1}(\mathcal{N})}
	\max _{u \in V_k}
	\frac{\mathcal{L}(u,u)}{\|u\|_{L^2(\mathcal{N})}
	^2},
	\label{eq:min_max}
\end{align}
where $V_k$ is any $k$-dimensional subspace of $W^{1,2}_{0,1}(\mathcal{N})$.

\begin{definition}
	\label{def_Eigenspace}
	We define $\mathcal{J}_\lambda= \mathcal{J}_\lambda(L)$ to be the eigenspace of Jacobi operator $L$ with eigenvalue $\lambda$, and we define the following spaces,
\begin{align*}
	\mathcal{J}_0^-=\mathcal{J}_0^-(\mathcal{N}):={} & 
	\bigoplus_{\lambda<0}^{}
	\mathcal{J}_\lambda(L),\\
	\mathcal{J}_0^0=\mathcal{J}_0^0(\mathcal{N}):={} & 
	\mathcal{J}_0(L).
\end{align*}

\end{definition}

Note that by the min-max principle (\ref{eq:min_max}), we have
\[
	\mathrm{Ind}(\mathcal{N})=
	\mathrm{dim}(\mathcal{J}_0^-),\quad 
	\mathrm{Nul}(\mathcal{N})=\mathrm{dim}
	(\mathcal{J}_0^0).
\]

\subsection{The Dirichlet-to-Neumann map}%
\label{sub:dirichlet_to_neumann_map}
\begin{definition}
	\label{def_DeriNullSpace}
	For any $\tilde{\mathcal{Q}} \subset \mathcal{Q}$, the subset of $\mathcal{Q}$, we define 
\[ D_\tau \mathcal{J}_0^0 (\tilde{\mathcal{Q} } ):= \{ v \in V(\tilde{\mathcal{Q}} ): v=\frac{\partial }{\partial \tau}u \text{ on } \tilde{\mathcal{Q}} \text{ for some }u \in \mathcal{J} _0^0\}.\] 
If $\tilde{\mathcal{Q}} = \mathcal{Q}$, then we write $D_\tau\mathcal{J}_0^0 := D_\tau\mathcal{J}_0^0 (\mathcal{Q})$ for short.
\end{definition}

Then by the Fredholm alternative, we have the following result.

\begin{proposition}
	The Dirichlet problem (\ref{eq:Dirichlet_problem}) has a (weak) solution if and only if 
	\[
		\int_{ \mathcal{N}} fv+ \int_{ \mathcal{Q}} g \frac{\partial }{\partial \tau}v=0, \quad \forall v \in \mathcal{J}_0^0
		(\mathcal{N}).
	\]
	
	In addition, this solution is unique up to an addition of $v \in \mathcal{J}_0 ^0$.

	Moreover, for the case $f\equiv 0$, the problem (\ref{eq:Dirichlet_problem}) has a solution if and only if $g \in (D_\tau \mathcal{J}_0^0)^\bot$. 

	Hence for any $\tilde{\mathcal{Q}} \subset \mathcal{Q}$, we can define a map
\begin{align*}
		T:(D_\tau \mathcal{J}_0^0
		(\tilde{\mathcal{Q}} ))^\bot \rightarrow {} & (D_\tau \mathcal{J}
		_0^0(\tilde{\mathcal{Q}})) ^\bot  \\
		g\rightarrow {} & 
		\frac{\partial }{\partial \tau}u
	\end{align*}
	where $u$ is a solution of (\ref{eq:Dirichlet_problem}) with $f\equiv 0$ and giving boundary condition $g$ taking the value zero outside of $\tilde{\mathcal{Q}} $ and $(D_\tau \mathcal{J}_0^0 (\tilde{\mathcal{Q}} ))^\bot $ is the orthogonal complement of $D_\tau \mathcal{J}_0^0 (\tilde{\mathcal{Q} } )$ in the space $V(\tilde{\mathcal{Q}} )$.
	\label{prop:D_N_map}
\end{proposition}

Based on Proposition \ref{prop:D_N_map}, we can make the following definition.
\begin{definition}
	\label{def_DNmap}
	For any $\tilde{\mathcal{Q}}\subset \mathcal{Q}$, we define the \textit{Dirichlet-to-Neumann map} $T$ of $\mathcal{N}$ with respect to $\tilde{\mathcal{Q}}$ to be the map defined in Proposition \ref{prop:D_N_map}.
\end{definition}
\begin{remark}
	When we talk about the Dirichlet-to-Neumann map $T$ of a network $\mathcal{N}$, we always assume this is a map defined on boundary $\mathcal{Q}$ or a subset of $\mathcal{Q}$.
\end{remark}

\begin{proof}[Proof of Proposition \ref{prop:D_N_map}]
	Let $\hat{g}$ be any smooth extension of $g$ in $C_1^\infty(\mathcal{N})$ with compatibility condition and $\hat{f}=f-L\hat{g}$.
	Applying the Fredholm alternative for the operator $L_\sigma^{-1}$ for some $\sigma$ large enough, we know the problem (\ref{eq:Dirichlet_problem2}) has a weak solution if and only if
	\begin{equation}
		\int_{ \mathcal{N}} \hat{f} v=0, \quad \forall v \in \mathcal{J}_0^0.
		\label{eq:pfSolvableCondition}
	\end{equation}

	This solution is unique up to an addition of $v \in \mathcal{J}_0^0 $.
	Note that if $v \in \mathcal{J}_0^0$, using \eqref{eq:defLBilinearComp}, we know the condition \eqref{eq:pfSolvableCondition} is equivalent to %using $Lv=0$ on $\mathcal{N}$, $v=0$ on $\mathcal{Q}$, we know the above condition equivalent to
	\begin{align*}
		0={} & \int_{ \mathcal{N}} fv-vL\hat{g}=\int_{ \mathcal{N}} fv+\mathcal{L}(v,\hat{g})-\int_{ \mathcal{Q}} v \frac{\partial \hat{g}}{\partial \tau}\\
		={} & \int_{ \mathcal{N}}(fv-\hat{g}Lv)+\int_{\mathcal{Q}} (\hat{g} \frac{\partial v}{\partial \tau}-v \frac{\partial \hat{g}}{\partial \tau})\\
		={}& \int_{\mathcal{N}}^{} fv+\int_{ \mathcal{Q}} g \frac{\partial v}{\partial \tau}.
	\end{align*}
	Here, for the last equality, we have used $Lv=0$ on $\mathcal{N}$ and $v=0$ on $\mathcal{Q}$.
	This is precisely the condition of the solvability of the problem (\ref{eq:Dirichlet_problem}). 

	For the case $f\equiv 0$, (\ref{eq:Dirichlet_problem}) has a solution if and only if 
	\[
		\int_{ \mathcal{Q}} g \frac{\partial v}{\partial \tau}=0
		\Longleftrightarrow g \in (D_\tau \mathcal{J}_0^0 )^\bot.
	\]

	This solution is smooth since $f\equiv 0$ is smooth. So it is not just a weak solution. 

	Now let $g \in (D_\tau \mathcal{J}_0^0 (\tilde{\mathcal{Q}} )) ^\bot $. We extend $g$ to an element $\hat{g}$ in $V(\mathcal{Q})$ by taking zeros outside of $\tilde{\mathcal{Q}}$. Since $\hat{g} \in (D_\tau\mathcal{J}_0^0)^\bot $, we can find a solution of (\ref{eq:Dirichlet_problem}) with $f\equiv 0$ with given boundary condition $\hat{g}$.
	Let $u$ be one such solution.
	In general, we do not have $\left.\frac{\partial u}{\partial \tau}\right|_{\tilde{\mathcal{Q}}}\bot D_\tau\mathcal{J}_0^0(\tilde{\mathcal{Q}})$. But by the properties of orthogonal projections, we know there is a unique $h \in D_\tau\mathcal{J}_0^0(\tilde{\mathcal{Q}})$ such that
	\[
		\left.\frac{\partial u}{\partial \tau}\right|_{\tilde{\mathcal{Q}}}+h\bot D_\tau\mathcal{J}_0^0(\tilde{\mathcal{Q}}).
	\]

	By the definition of $D_\tau\mathcal{J}_0^0(\tilde{\mathcal{Q}})$, there is a function $v \in \mathcal{J}_0^0$ such that $h=\left.\frac{\partial v}{\partial \tau}\right|_{\tilde{\mathcal{Q}}}$. We will denote $u_g$ as $u+v$ with $u,v$ defined above. In summary, $u_g \in C^{\infty}_1(\mathcal{N})$ will solve the following problem by the construction of $u_g$,
	\begin{align}
	\begin{cases}
	Lu_g=0, & \text{ on }\mathcal{N}, \\
	u_g=g, &  \text{ on }\tilde{\mathcal{Q}},\\
	u_g=0, & \text{ on }\mathcal{Q}\backslash \tilde{\mathcal{Q}},\\
	\frac{\partial u_g}{\partial \tau} \in V_2(\mathcal{P}),\quad &
	\text{ on }\mathcal{P}\ (\text{compatibility conditions}),\\
	\frac{\partial u_g}{\partial \tau} \in (D_\tau\mathcal{J}_0^0(\tilde{\mathcal{Q}}))^\bot ,\quad &\text{ on }\tilde{\mathcal{Q}}.
	\end{cases}
	\label{eq:Lextension}
	\end{align}
	
	Hence, we can define 
	\[
		T(g):=\left.\frac{\partial u_g}{\partial \tau}\right|_{\tilde{\mathcal{Q}}}.
	\]

	To show $T$ is well-defined, we need to check if there is another solution $v_g \in C^{\infty}_1(\mathcal{N})$ to the problem (\ref{eq:Lextension}), then we have $\left.\frac{\partial u_g}{\partial \tau}\right|_{\tilde{\mathcal{Q}}}=\left.\frac{\partial v_g}{\partial \tau}\right|_{\tilde{\mathcal{Q}}}$. Indeed, if we set $w=u_g-v_g$, then $w \in C^{\infty}_1(\mathcal{N})$ solves the following problem
	\begin{align}
		\begin{cases}
		Lw=0, & \text{ on }\mathcal{N}, \\
		w=0, & \text{ on }\mathcal{Q},\\
		\frac{\partial u}{\partial \tau} \in V_2(\mathcal{P}),&\text{ on }\mathcal{P},\\
		\frac{\partial w}{\partial \tau} \in (D_\tau\mathcal{J}_0^0(\tilde{\mathcal{Q}}))^\bot ,\quad &\text{ on }\tilde{\mathcal{Q}}.
		\end{cases}
		\label{eq:pf_unq}
	\end{align}

	Note that first three conditions for $w$ in problem (\ref{eq:pf_unq}) imply $w \in \mathcal{J}_0^0$. Hence $\left.\frac{\partial w}{\partial \tau}\right|_{\tilde{\mathcal{Q}}} \in D_\tau\mathcal{J}_0^0(\tilde{\mathcal{Q}})$ by the definition of $D_\tau\mathcal{J}_0^0(\tilde{\mathcal{Q}})$. Combining with the fourth condition in problem (\ref{eq:pf_unq}), we know $\left.\frac{\partial w}{\partial \tau}\right|_{\tilde{\mathcal{Q}}}=0$. So the $T$ is indeed well-defined.  
\end{proof}

In the proof of Proposition \ref{prop:D_N_map}, we have shown for any $g \in (D_\tau\mathcal{J}_0^0 (\tilde{\mathcal{Q}}))^\bot $, there is an extension $u_g \in C^{\infty}_1(\mathcal{N})$ of $g$ solving the problem (\ref{eq:Lextension}).
But the choice of $u_g$ might not be unique.
We need to define a new function space to give a canonical way of the definition of the extension $u_g$.

\begin{definition}
	\label{def_NullSpace00}
	We define the function space $\mathcal{I}_0(\mathcal{\tilde{Q}})$ by
	\[
		\mathcal{I}_0(\tilde{\mathcal{Q}} ):=\{ v \in \mathcal{J}_0^0 , \frac{\partial v }{\partial \tau}=0 \text{ on } \tilde{\mathcal{Q}}\}.
	\]
\end{definition}

Based on Proposition \ref{prop:D_N_map}, we can define $u_g$, the \textit{$L$-extension} of $g$.

\begin{proposition}
	\label{prop_Lext}
	For any $g \in (D_\tau \mathcal{J}_0^0(\mathcal{\tilde{Q}}))^\bot $, there exists a unique $u_g$, called the \textit{$L$-extension} of $g$, such that $u_g$ solves Problem \eqref{eq:Lextension} and
	$u_g \bot \mathcal{I}_0(\mathcal{\tilde{Q}})$.
	Here, $u_g \bot \mathcal{I}_0(\mathcal{\tilde{Q}})$ means $u_g$ is orthogonal to $\mathcal{I}_0(\mathcal{\tilde{Q}})$ with respect to the standard inner product defined on space $L^2(\mathcal{N})$.
\end{proposition}

\begin{proof}
	Indeed, the proof of Proposition \ref{prop:D_N_map} already showed Problem \eqref{eq:Lextension} has a solution, and such solution is unique up to an addition of the function which solves Problem \eqref{eq:pf_unq}.
	But the solution set to the Problem \eqref{eq:pf_unq} is exactly $\mathcal{I}_0(\tilde{\mathcal{Q}})$.
	Hence, by adding a suitable element in $\mathcal{I}_0(\mathcal{\tilde{Q}})$, we can find a unique $u_g$ solving Problem \eqref{eq:Lextension} and $u_g \bot \mathcal{I}_0(\mathcal{\tilde{Q}})$.
\end{proof}
%Moreover, $u_g$ is unique up to an addition of functions in the space $\mathcal{I}_0(\tilde{\mathcal{Q}} ):=\{ v \in \mathcal{J}_0^0 , \frac{\partial v }{\partial \tau}=0 \text{ on } \tilde{\mathcal{Q}}\}$ by the Fredholm alternative. (Indeed, the solution set of problem (\ref{eq:pf_unq}) is just $\mathcal{I}_0(\tilde{\mathcal{Q}})$.)
%Hence, we can always choose $u_g$ such that $u_g$ is orthogonal to $\mathcal{I}_0(\tilde{Q})$ (write it as $u_g \bot \mathcal{I}_0(\tilde{Q})$) with respect to the standard inner product $(u,v):= \int_{ \mathcal{N}} uv$ defined on space $L^2(\mathcal{N})$.
%We call $u_g$ is an \textit{$L$-extension} of $g$ with respect to $\tilde{\mathcal{Q}}$ if $u$ solves problem (\ref{eq:Lextension}) and $u \bot \mathcal{I}_0 (\tilde{\mathcal{Q}})$.
It is easy to verify the $L$-extension is a linear map from $(D_\tau\mathcal{J}_0^0(\tilde{\mathcal{Q}}))^\bot $ to $C^{\infty}_1(\mathcal{N})$.
Hence, we know $T:(D_\tau\mathcal{J}_0^0 (\tilde{\mathcal{Q}} )) ^\bot \rightarrow (D_\tau\mathcal{J}_0^0 (\tilde{\mathcal{Q}} ))^\bot$ is a linear map. Moreover, $T$ is indeed a self-adjoint linear transformation on the finite-dimensional space $(D_\tau \mathcal{J}_0^0 (\tilde{\mathcal{Q}} ))^\bot$ after a simple calculation, so $T$ has finite real eigenvalues, which we write them as $\sigma_1\le \sigma_2\le \cdots \le \sigma _{\mathrm{dim}(D_\tau \mathcal{J}_0^0  (\tilde{\mathcal{Q}} ))^\bot }$.

At last, we mention a result that the eigenvalues of $\mathcal{N}$ only rely on the image of $\mathcal{N}$ in $\mathbb{S}^2$ in some sense. 

Let us see an example first. Given a curve $\gamma:[0,1]\rightarrow \mathbb{S}^2$, we can think this curve as a network of course with $\mathcal{P}=\emptyset $. On the other hand, we can cut this curve into two parts in the sense that we choose $\gamma_1=\gamma|_{[0,a]}, \gamma_2=\gamma|_{[a,1]}$ for $0<a<1$. Now if we consider new network $\mathcal{N}_1:=(\gamma_1,\gamma_2)$ with $\mathcal{P}_1=\{ \gamma(a) \}$, then the network $\mathcal{N}_1$ still represents the curve $\gamma$ as they have the same image. 
We call $\gamma_1,\gamma_2$ as the refinement of $\gamma$.
Now we say a \textit{one-step refinement} of network $\mathcal{N}$ is the new network $\mathcal{N}_1$ that all of $\gamma^i$ coming from $\mathcal{N}$ but only one $\gamma^i$ is replaced by $\gamma_1^i,\gamma_2^i$, the refinement of $\gamma^i$. Note that $\mathcal{P}_1$ will have one more element than $\mathcal{P}$.
We say two networks $\mathcal{N}_1, \mathcal{N}_2$ are one-step congruent to each other if $\mathcal{N}_1$ is the one-step refinement of $\mathcal{N}_2$ or vice versa. We say $\mathcal{N}_1, \mathcal{N}_2$ are congruent to each other if there is a finite sequence of one-step refinements starting at $\mathcal{N}_1$ and ending at $\mathcal{N}_2$.

Clearly, if two networks are congruent, then they will have the same image in $\mathbb{S}^2$. Moreover, for two embedded networks, they are congruent if and only if they have the same image.

We have the following proposition.
\begin{proposition}
	\label{prop_congruent}
	If $\mathcal{N}_1,\mathcal{N}_2$ are congruent to each other, then they have the same eigenvalues (counting multiplicity).
	
\end{proposition}
\begin{proof}
	By the definition of congruent, we only consider the one-step refinement case. But for this case, this is easy to see since the eigenfunctions on $\mathcal{N}$ can be cut like cutting the curve (restriction to subcurves). Conversely, we can connect two functions reversely by the compatibility conditions. This will give a weak solution after the connection. Hence by the results of regularity, we know it should be a classical solution, too.
\end{proof}

Moreover, we have the following result.
\begin{proposition}
	The eigenvalues of the network do not rely on the choice of the orientation on each curve $\gamma^i$ in $\mathcal{N}$. That is, if we choose different normal vector fields $\hat{\nu}^j$ for $\gamma^j$, they still give the same eigenvalues.
	
\end{proposition}
The proof for this proposition is quite simple since we can give a direct one-to-one map between their admissible function spaces when giving different normal vector fields. For any eigenfunction $\phi$ of $L$ with normal vector field $\nu$ on $\mathcal{N}$, if we have another unit normal vector field $\hat{\nu}$ on $\mathcal{N}$, we construct $\hat{\phi}$ by choosing $\hat{\phi}_i=\phi_i$ if $\hat{\nu}_i=\nu_i$ and $\hat{\phi}_i=-\phi_i$ if $\hat{\nu}_i=-\nu_i$. It is easy to verify $\hat{\phi}$ is an eigenfunction with the same eigenvalue with unit normal vector field $\hat{\nu}_i$.

These two propositions show that the eigenvalues of an embedded network are entirely determined by its image.

\section{Index Theorem for networks}%
\label{sec:index_theorem_for_networks}

This section will show that one can compute the index and nullity of a large network by dividing it into small networks.

To be more precise, let $\mathcal{N}$ be a network in $\mathbb{S}^2$. Let $\mathcal{N}_i,i=1,\cdots ,l$ be networks in $\mathbb{S}^2$ with interior endpoint set $\mathcal{P}_i$ and boundary endpoint set $\mathcal{Q}_i$.
We say $\{ \mathcal{N}_i \}_{i=1,\cdots , l}$ is a \textit{partition} of $\mathcal{N}$ if the following holds.
\begin{itemize}
	\item $\mathcal{N}=\bigsqcup_{i=1}^{l} \mathcal{N}_i$. $\mathcal{N}$ is the disjoint union of $\mathcal{N}_i$.
		
	\item $\mathcal{P}$ can be divided into two disjoint sets $\mathring{ \mathcal{P}} $ and $\overline{\mathcal{P} }$. $\mathcal{Q}_i$ can be divided into two disjoint sets $ \tilde{\mathcal{Q}}_i$ and $\overline{\mathcal{Q}}_i$ for each $i=1,\cdots ,l$. These divisions make the following holds.
		\begin{itemize}
			\item $\mathring {\mathcal{P}}= \bigsqcup_ {i=1}^{l}\mathcal{P}_i$.
			\item $\overline{\mathcal{P}}= \bigcup_{i=1}^{l} \tilde{\mathcal{Q}}_i$.
			\item $\mathcal{Q}=\bigcup_{i=1}^{l }\overline{\mathcal{Q}}_i$.
		\end{itemize}
\end{itemize}

Given a network $\mathcal{N}$, we will cut it into several small networks along some interior endpoint set $\overline{\mathcal{P}}$ in order to compute its index and nullity.

\subsection{Dirichlet-to-Neumann map with respect to a partition}%
\label{sub:dirichlet_to_neumann_on_mathcalp_}

Here we want to define a Dirichlet-to-Neumann map with respect to the partition $\mathcal{N}=\cup _{i=1}^l \mathcal{N}_l$ using the Dirichlet-to-Neumann map from the subnetworks.

Recall that the Dirichlet-to-Neumann map on $\mathcal{N}_i$ with respect to $\tilde{\mathcal{Q}}_i$ is defined as the map
\begin{equation}
	T:(D_\tau \mathcal{J}_0^0 (\tilde{\mathcal{Q}}_i))
	^\bot \rightarrow (D_\tau \mathcal{J}_0^0 (\tilde{\mathcal{Q}}_i))
	^\bot, g\rightarrow \frac{\partial u_g}{\partial \tau},
	\label{eq:DNmapNoProjection}
\end{equation}
where $u_g$ is the $L$-extension of $g$ with respect to $\tilde{\mathcal{Q}}_i$.

Note that $V(\overline{\mathcal{P}})\cong \bigoplus_{i=1}^{l}V(\tilde{\mathcal{Q}}_i)$ and $(D_\tau\mathcal{J}_0^0(\tilde{\mathcal{Q}}_i)) ^\bot  \subset V(\tilde{\mathcal{Q}}_i)$, we can define $\hat{T}$ on $\bigoplus_{i=1}^{l} (D_\tau\mathcal{J}_0^0(\tilde{\mathcal{Q}}_i))^\bot $ using $T$ on each $\mathcal{N}_i$ by 
\begin{align}
	\hat{T}: \bigoplus_{i=1}^{l}(D_\tau\mathcal{J}_0^0
	(\tilde{\mathcal{Q}}_i))^\bot {}&\rightarrow \bigoplus_{i=1}^{l}(D_\tau\mathcal{J}_0^0
	(\tilde{\mathcal{Q}}_i))^\bot \nonumber \\
	(g_1,\cdots ,g_l){}& \rightarrow (T(g_1),\cdots ,T(g_l))
	\label{eq:That}
\end{align}

But still, we prefer to define a map related to the space $V_1(\overline{\mathcal{P}})$.
\begin{definition}
	\label{def_V1bar}
	We define a subspace of $V_1(\mathcal{\overline{P}})$ by
	\[\overline{V}_1(\overline{\mathcal{P}} ):= \bigoplus_{i=1}^{l}(D_\tau\mathcal{J}_0^0 (\tilde{\mathcal{Q}}_i))^\bot \cap V_1(\overline{\mathcal{P}}).\]
\end{definition}
	We write $\text{Proj}_1:V(\overline{\mathcal{P}} )\rightarrow V_1(\overline{\mathcal{P}})$ as the orthogonal projection to the space $V_1(\overline{\mathcal{P}})$.
Restricting it to $\bigoplus_{i=1}^{l} (D_\tau\mathcal{J}_0^0(\tilde{\mathcal{Q}}_i))^\bot$, we know $\text{Proj}_1: \bigoplus_{i=1}^{l}(D_\tau\mathcal{J}_0^0 (\tilde{\mathcal{Q}}_i))^\bot \rightarrow \overline{V}_1(\overline{\mathcal{P}})$ is an orthogonal projection.

\begin{definition}
	\label{def_DNpartition}
	We define a \textit{Dirichlet-to-Neumann map $\overline{T}: \overline{V}_1(\overline{\mathcal{P}}) \rightarrow \overline{V}_1 (\overline{\mathcal{P}})$ with respect to the partition} $\mathcal{N}=\cup _{i=1}^l \mathcal{N}_l$ by
\begin{equation}
	\overline{T}(g):= \text{Proj}_1(\hat{T}(g)).
	\label{eq:Tbar}
\end{equation}
\end{definition}

\begin{remark}
	Note that if we write a map as $\overline{T}$, we always mean the map defined in Definition \ref{def_DNpartition}.
	Instead, if we have a map $T$, the one without the bar, then we mean this is a Dirichlet-to-Neumann map defined in Definition \ref{def_DNmap}.
\end{remark}

We can find that $\overline{T}$ is still a self-adjoint linear transformation on the space $\overline{V}_1(\overline{\mathcal{P}})$.
So it has finite real eigenvalues. 
We use $E_\sigma$ to denote the subspace of $\overline{V}_1 (\overline{\mathcal{P}})$ with eigenvalue $\sigma$ with respect to the operator $\overline{T}$.

We define the following subspaces of $\overline{V}_1(\overline{\mathcal{P}})$ by 
\begin{itemize}
	\item $E^-(\overline{T}):= \bigoplus_{\sigma<0}^{} E_\sigma$. $\overline{T}$ is negative definite on this space.
	\item $E^0(\overline{T}):= E_0$, $\overline{T}$ is zero on this space.
\end{itemize}

So the index and nullity of $\overline{T}$ can be written as follows,
\[ \mathrm{Ind}(\overline{T}):= \mathrm{dim} E^-(\overline{T}),\quad\mathrm{Nul}(\overline{T}):= \mathrm{dim}E^0(\overline{T}).\]

%We call the operator $\overline{T}:\overline{V}_1(\overline{\mathcal{P}})\rightarrow \overline{V}_1(\overline{\mathcal{P}})$ to be the \textit{Dirichlet-to-Neumann map} related to network $\mathcal{N}$ with its partition $\{ \mathcal{N}_i \}$.

Note that eigenvalues of $\overline{T}$ can also be characterized by the min-max principle as
\[
	\sigma_k(\overline{T})= \min _{V_k \subset \overline{V}_1 (\overline{\mathcal{P}})} \max _{g \in V_k} \frac{(\overline{T}(g),g)}{(g,g)},
\] 
where $(\cdot,\cdot)$ is the standard Euclidean inner product on $\overline{V}_1 (\overline{\mathcal{P}})$ and $V_k$ a $k$-dimensional subspace of $\overline{V}_1(\overline{\mathcal{P}})$. This is because $\overline{T}$ is a self-adjoint operator on a finite-dimensional space.

\begin{remark}
Note that for any $v \in C^\infty_{0,1}(\mathcal{N})$, we have
\begin{align}
	(\overline{T}(g),v|_{\overline{\mathcal{P}}})={}& \int_{ \overline{\mathcal{P}}}
	\text{Proj}_1(\frac{\partial }{\partial \tau}u_g)v= \int_{ \overline{\mathcal{P}}} 
	v \frac{\partial u_g}{\partial \tau}=\int_{ \mathcal{N}} 
	\nabla u_g \cdot \nabla v+v\Delta u_g\nonumber \\
	={}&\int_{ \mathcal{N}} 
	\nabla u_g\cdot \nabla v - du_g v=\mathcal{L}(u_g,v),
	\label{eq:D_N_map_two_forms}
\end{align}
here, we still use $u_g$ to denote the \textit{$L$-extension} of $g$ in the sense of $u_g=(u_{g|_{\tilde{\mathcal{Q}}_i}})_{i=1, \cdots ,l}$ where $u_{g|_{\tilde{\mathcal{Q}}_i}}$ is the $L$-extension of $g|_{\tilde{\mathcal{Q}}_i}$ on $\mathcal{N}_i$ with respect to the set $\tilde{\mathcal{Q}}_i$, and we use \eqref{eq:IBPcomp} with $\mathcal{N}_i$ in place of $\mathcal{N}$, $\tilde{\mathcal{Q}}_i$ in place of $\mathcal{Q}$ for the third equality.
\end{remark}

\begin{remark}
	There is another way to describe $\sigma_k(\overline{T})$ as
\[
	\sigma_k(\overline{T})=\min _{V_k \subset \overline{V}_1(\overline{\mathcal{P}})}
	\max _{g \in V_k}
	\frac{\mathcal{L}(u_g,u_g)}{(g,g)}.
\]
\end{remark}

Another several important spaces are those related to the nullity of the subnetworks $\mathcal{N}_i$.
Let us define the following spaces.
\begin{definition}
	\label{def_Fj}
	We define
\begin{align*}
	F_j:={}&\bigoplus_{i=1}^{l}
	D_\tau\mathcal{J}_0^0(\tilde{\mathcal{Q}}_i
	)\cap V_j(\overline{\mathcal{P}})
	\quad \text{ for }j=1,2.\\
	\overline{V}_2(\overline{\mathcal{P}}):={}&
	\bigoplus_{i=1}^{l}(D_\tau\mathcal{J}_0^0(\tilde{\mathcal{Q}}_i))
	^\bot  \cap V_2(\overline{\mathcal{P}}).
\end{align*}
\end{definition}
Recall that $\overline{V}_1(\mathcal{\mathcal{P}})$ is already defined in Definition \ref{def_V1bar}.

Clearly, the direct sum of $F_1,F_2$ is just $\bigoplus_{i=1}^{l}D_\tau\mathcal{J}_0^0 (\tilde{\mathcal{Q}}_i)$.
%To better illustrate the relations of these complex spaces, let's consider the space $\overline{V}_2(\overline{\mathcal{P}})$ defined by
%\[
%	\overline{V}_2(\overline{\mathcal{P}}):=
%	\bigoplus_{i=1}^{l}(D_\tau\mathcal{J}_0^0(\tilde{\mathcal{Q}}_i))
%	^\bot  \cap V_2(\overline{\mathcal{P}}).
%\]
%
%Then the $V(\overline{\mathcal{P}})$ can be decomposed as 
%$V(\overline{\mathcal{P}})=\overline{V}_1(\overline{\mathcal{P}})
%\oplus \overline{V}_2(\overline{\mathcal{P}})\oplus F_1\oplus F_2$.
Here, we list some relations between $F_j,\overline{V}_j(\overline{\mathcal{P}})$, $V_j(\overline{\mathcal{P}})$, and $D_\tau\mathcal{J}_0^0(\tilde{\mathcal{Q}}_i)$.
\begin{align}
	V_1(\overline{\mathcal{P}})={} & \overline{V}_1(\overline{\mathcal{P}})\oplus F_1, \label{eq:V1decomp}\\
	V_2(\overline{\mathcal{P}})={} & \overline{V}_2(\overline{\mathcal{P}})\oplus F_2,\\
	\bigoplus_{i=1}^{l}( D_\tau\mathcal{J}_0^0(\tilde{\mathcal{Q}}_i) ) ^\bot ={}& \overline{V}_1(\overline{\mathcal{P}})\oplus 
	\overline{V}_2(\overline{\mathcal{P}}),\\
	\bigoplus_{i=1}^{l}D_\tau\mathcal{J}_0^0(\tilde{\mathcal{Q}}_i)={}&  F_1\oplus F_2.
\end{align}

Here is a quick explanation of the meaning of $F_1$ and $F_2$. For any $g \in 
\bigoplus_{i=1}^{l}D_\tau\mathcal{J}_0^0(\tilde{\mathcal{Q}}_i)$, we know there is $u \in C_1^\infty(\mathcal{N})$ such that $u|_{\mathcal{N}_i} \in \mathcal{J}_0^0(\mathcal{N}_i)$ and $\frac{\partial u}{\partial \tau}=g$ on $\overline{\mathcal{P}}$. Note that actually $Lu=0$ on $\mathcal{N}$ all the time. But since $u$ does not satisfy the compatibility condition on $\overline{\mathcal{P}}$ in general, $u$ is not an eigenfunction on $\mathcal{N}$. We can divide the set $\bigoplus_{i=1}^{l}D_\tau\mathcal{J}_0^0(\tilde{\mathcal{Q}}_i)$ into $F_1$ and $F_2$. If $g \in F_2$, then the corresponding $u$ will satisfy the compatibility condition on $\overline{\mathcal{P}}$ since $\frac{\partial u}{\partial \tau}=g \in V_2(\overline{\mathcal{P}})$. 
This time, $u$ will become an eigenfunction for $L$ on $\mathcal{N}$ with eigenvalue 0. So $F_2$ will be related to the nullity of the whole network $\mathcal{N}$. Informally, this space is related to the functions such that they have eigenvalues 0 when restricted on each small network, and they happen to be the null functions on the big network. The orthogonal complement of $F_2$ is $F_1$. So we know the meaning of $F_1$ is $F_1$ will be related to the functions such that they are in the null space when restricted on each small network, but they will be orthogonal to the null space of the whole networks informally. 
%This is also the statement that appeared in Theorem \ref{thm_index_and_nullity_theorem_short}.

Now let us state the main index theorem in a formal way.
\begin{theorem}
	[Index Theorem for networks]
	Suppose $\mathcal{N}$ is a network with partition $\{ \mathcal{N}_i \}_{i=1,\cdots ,l}$.
	Then the index of $\mathcal{N}$ can be computed as follows,
	\[
		\mathrm{Ind}(\mathcal{N})= \sum_{i =1}^{l}\mathrm{Ind}
		(\mathcal{N}_i)+ \mathrm{Ind}(\overline{T})+ \mathrm{dim}(F_1).
	\]
	\label{thm_index_theorem_for_network}
\end{theorem}
\begin{remark}
	This theorem holds for arbitrary networks, not only triple junction ones.
\end{remark}

The proof of this theorem relies on the following lemma.

\begin{lemma}
	Let $g \in F_1$ and suppose 
	\[
	W_g:=\{ v \in C^\infty_{0,1}(\mathcal{N} ):v|_{\overline{\mathcal{P}}}=cg \text{ for some }c \in \mathbb{R}, v|_{\mathcal{N}_i}\bot_\mathcal{L} \mathcal{J}_0^- (\mathcal{N}_i) \text{ for each }i\}.
	\]
	Here we use $\bot_\mathcal{L}$ to denote the orthogonal with respect to the bilinear form $\mathcal{L}$.
Then the following holds.
\begin{itemize}
	\item If $g\equiv 0$ on $ \overline{\mathcal{P}}$, then $\mathcal{L}$ is non-negative definite on the space $W_g$.
	\item If $g\neq 0$ on $ \overline{\mathcal{P}}$, then $\mathcal{L}$ has index 1 on the space $W_g$.
\end{itemize}
\label{lem:1_dim_extension}
\end{lemma}
\begin{proof}
	For the first case, for any $v \in W_g$, we know $v|_{\mathcal{Q}_i} \equiv 0$. Hence $v|_{\mathcal{N}_i} \in C^\infty _{0,1}(\mathcal{N}_i)$.
	Since $v|_{\mathcal{N}_i}\bot \mathcal{J}_0^-(\mathcal{N}_i)$, we know
	\[
		\mathcal{L}(v|_{\mathcal{N}_i}, v|_{\mathcal{N}_i})\ge 0.
	\]

	Hence
	\[
		\mathcal{L}(v,v)=\sum_{i =1}^{l}
		\mathcal{L}(v|_{\mathcal{N}_i}, v|_{\mathcal{N}_i})\ge 0.
	\]

	For the second case, since $g|_{\tilde{\mathcal{Q}}_i} \in D_\tau\mathcal{J}_0^0 (\tilde{\mathcal{Q}}_i)$, we can find some $w_i \in \mathcal{J}_0^0 (\mathcal{N}_i)$ such that $\frac{\partial w_i}{\partial \tau}|_{\tilde{\mathcal{Q}}_i }=g|_{\tilde{\mathcal{Q}}_i}$. Then we set $w=\bigoplus_{i=1}^{l}w_i$. Clearly $w|_{\mathcal{N}_i} \bot _{\mathcal{L}}
	\mathcal{J}_0^-(\mathcal{N}_i)$ and $w \in C^{\infty}_{0,1}(\mathcal{N})$.
	Let $\hat{g} \in C^\infty_{0,1} (\mathcal{N})$ be the smooth extension of $g$ such that $\hat{g}|_{\mathcal{N}_i}\bot_{\mathcal{L}}\mathcal{J}_0^-(\mathcal{N}_i)$ for each $i$.
	Such $\hat{g}$ exists since $\mathcal{J}_0^-(\mathcal{N}_i)$ is a finite dimensional space for each $i$.
	Then we know $\hat{g}+cw \in W_g$ for any $c \in \mathbb{R} $.

Note that
\[
	\mathcal{L}(w,w)=\sum_{i =1}^{l}
	\mathcal{L}(w_i,w_i)=0.
\]

So
\begin{align*}
	\mathcal{L}(\hat{g}+cw,\hat{g}+cw)
	={}&\mathcal{L}(\hat{g},\hat{g})+2c\mathcal{L}(\hat{g},w)\\
	={}& \mathcal{L}(\hat{g},\hat{g})
	+2c\left( \int_{ \overline{\mathcal{P}}} 
	g \frac{\partial w}{\partial \tau}- \int_{ \mathcal{N}} \hat{g}
Lw\right) \\
	={}& \mathcal{L}(\hat{g},\hat{g})+2c\int_{ 
	\overline{\mathcal{P}}}g^2,
\end{align*}
where we have used \eqref{eq:defLBilinearComp} for second equality.
Since $\int_{ \overline{\mathcal{P}}} 
g^2>0$, we can always choose $c \in \mathbb{R} $ such that $\mathcal{L}(\hat{g}+cw,\hat{g}+cw)<0$. Hence $\mathcal{L}$ has at least index 1 on the space $W_g$.

On the other hand, suppose we have two elements $v_1,v_2 \in W_g$ such that $\mathcal{L}(\alpha v_1+\beta v_2, \alpha v_1+\beta v_2)<0$ for any $\alpha,\beta\neq 0$ and $v_1,v_2$ are linearly independent. By the definition of $W_g$, we know there is $c \in \mathbb{R} $ such that $v_1=cv_2$ when restricting on $\overline{\mathcal{P}}$.
Hence $\left( v_1-cv_2 \right) |_{\overline{\mathcal{P}}}=0$. 
By the result in the first case, we know $\mathcal{L}(v_1-cv_2,v_1-cv_2)\ge 0$, this is a contradiction with $\mathcal{L}$ having index at least two on $W_g$.
So we have shown $\mathcal{L}$ has the exact index one on space $W_g$.
\end{proof}

\begin{proof}[Proof of Theorem \ref{thm_index_theorem_for_network}]
	We divide the proof into two parts.

	\noindent\textbf{First part.}
	We prove the following result,
	\[\mathrm{Ind}(\mathcal{N})\ge \sum_{i =1}^{l}\mathrm{Ind}
		(\mathcal{N}_i)+ \mathrm{Ind}(\overline{T})+ \mathrm{dim}(F_1).\]
	
	Let us do it step by step.

	\noindent\textbf{Step 1.} $\mathcal{L}$ is negative definite on each $\mathcal{J}_0^-(\mathcal{N}_i)$ (taking zero extension outside of $\mathcal{N}_i$). This is obvious since $\mathcal{L}$ has the same values on the subnetwork and the whole network after zero extension.
	Note that $\mathcal{J}_0^-(\mathcal{N}_i) \subset C^{\infty}_{0,1}(\mathcal{N}_i)$ and $\mathcal{J}_0^-(\mathcal{N}_i) \bot _{\mathcal{L}}\mathcal{J}_0^-(\mathcal{N}_j)$ for any $i\neq j$. (Here the functions in $\mathcal{J}_0^-(\mathcal{N}_i)$ might not the satisfy compatibility condition on $\overline{\mathcal{P}}$ after zero extension on the whole network $\mathcal{N}$.)

	\noindent\textbf{Step 2.}
	If $(\overline{T}(g),g)<0$, from identity (\ref{eq:D_N_map_two_forms}), we know $\mathcal{L}(u_g,u_g)= (\overline{T}(g),g)<0$ if $u_g$ is the $L$-extension of $g$. 
	So we can define
	\[
		E^-_L(\overline{T}):= \{ u_g \in C^\infty_{0,1}
	(\mathcal{N}):g \in E^-(\overline{T}) \}
	\]
	to be the $L$-extension space of $E^-(\overline{T})$. Note that $L$-extension is a linear map, so $E^-_L(\overline{T})$ is a finite-dimensional space. It is easy to see $\mathrm{dim}(E^-_L(\overline{T} ))=\mathrm{dim}(E^-(\overline{T}))$ and $\mathcal{L}$ is negative definite on the space of $E^-_L(\overline{T})$ by using the identity (\ref{eq:D_N_map_two_forms}).\\

	\noindent\textbf{Step 3.} We will construct a new space from $F_1$ that $\mathcal{L}$ is negative definite on.

	Let $N=\mathrm{dim}(F_1)$ and we write $\{ g^1,\cdots ,g^N \}$ as the orthogonal basis of $F_1$ in the $L^2$ norm. Then we can choose $\hat{g}^j$ as the extension of $g^j$ such that $\hat{g}^j|_{\mathcal{N} _i} \bot \mathcal{J}_0^- (\mathcal{N}_i)$ for each $i$ and $\mathcal{L}(\hat{g}^i,\hat{g}^j)=0$ for $i\neq j$. This is possible since each $\mathcal{J}_0^-(\mathcal{N}_i)$ is finite-dimensional and we can choose $\hat{g}^i$ inductively.

	Then by the proof of Lemma \ref{lem:1_dim_extension}, we can find $w^{j}, c^j$ with $g^j=\frac{\partial w^j }{\partial \tau}$ on $\overline{\mathcal{P}}$ and $w^j|_{\mathcal{N}_i}\in \mathcal{J}_0^0(\mathcal{N}_i)$ for each $i$ such that if we set $u^j=\hat{g}^{j}+c^j w^j$, then $\mathcal{L}(u^j,u^j)<0$.

	Moreover, we have
\begin{align*}
		\mathcal{L}(u^j,u^k)={} & 
		\mathcal{L}(\hat{g}^j,\hat{g}^k)+ c^j\mathcal{L}(\hat{g}^k,w^j)+ c^k\mathcal{L}(\hat{g}^j,w^k)+ c^jc^k\mathcal{L}(w^j,w^k)\\
		={} & 0+c^j\int_{ \overline{\mathcal{P}}
		}\hat{g}^k \frac{\partial w^j
	}{\partial \tau}+ c^k \int_{ \overline{\mathcal{P}}} 
	\hat{g}^j \frac{\partial w^k
	}{\partial \tau}+0\\
	={}& (c^j+c^k)
	\int_{ \overline{\mathcal{P}}} 
	\hat{g}^j\hat{g}^k=0
	\end{align*}
	for any $j\neq k$.
	
	So $\mathcal{L}$ is negative definite on the space $\text{span} \left( u^1,\cdots ,u^N \right) $.

	We use $\hat{F}_1$ to denote the space $\text{span}(u^1,\cdots ,u^N)$.
	\\

	\noindent\textbf{Step 4.} The spaces $\mathcal{J}_0^-(\mathcal{N}_i)$, $E^-_L(\overline{T})$, and $\hat{F}_1$ are orthogonal to each other with respect the form $\mathcal{L}$. Here we will view $\mathcal{J}_0^-(\mathcal{N}_i)$ as a subspace of $C^{\infty}_{0,1}(\mathcal{N})$ by zero extension outside of $\mathcal{N}_i$.

	Firstly, we know $\mathcal{J}_0^-(\mathcal{N}_i)\bot_{\mathcal{L}}\mathcal{J}_0^-(\mathcal{N}_j)$ for any $i\neq j$ by the properties of zero extension.

	Secondly, it is clear that $\hat{F}_1\bot_\mathcal{L} \mathcal{J}_0^- (\mathcal{N}_i)$ by the choice of $\hat{F}_1$. 

	Thirdly, for any $v \in \mathcal{J}_0^- (\mathcal{N}_i)$ and $u_g \in E^-_L(\overline{T})$, we have
\begin{align*}
		\mathcal{L}(v,u_g)={} & \int_{ \overline{\mathcal{P}}} 
		v \frac{\partial u_g}{\partial \tau}- \int_{ \mathcal{N}} v Lu_g=0.
	\end{align*}

	So $\mathcal{J}_0^-(\mathcal{N}_i) \bot_\mathcal{L} E^-_L(\overline{T})$.

	At last, for any $u_g \in E^-_L(\overline{T})$, we have
	\[
		\mathcal{L}(u_g,u^j)= \int_{ \overline{\mathcal{P}}} 
		u^j \frac{\partial u_g
		}{\partial \tau}- \int_{ \mathcal{N}} u^jLu_g= 0\quad \text{ for }j=1,\cdots ,N,
	\]
	since $\frac{\partial u_g}{\partial \tau} |_{\overline{\mathcal{P}}} \in D_\tau\mathcal{J}_0^0 (\overline{\mathcal{P}})^\bot $ and $u^j|_{\overline{\mathcal{P}}} \in D_\tau\mathcal{J}_0^0 (\overline{\mathcal{P}})$.
	Hence $\hat{F}_1\bot _{\mathcal{L}}E^-_L(\overline{T})$.

	With these results, we know that $\mathcal{L}$ is actually negative definite on the space $W:=\bigoplus_{i=1}^{l} \mathcal{J}_0^-(\mathcal{N}_i)\oplus E^-_L(\overline{T})\oplus \hat{F}_1$.

	So $\mathcal{L}$ has index at least $\sum_{i =1}^{l}\mathrm{Ind} (\mathcal{N}_i)+ \mathrm{Ind}(\overline{T})+ \mathrm{dim}(F_1)$ on $\mathcal{N}$.
	\\

	\noindent\textbf{Second part.}
	We prove the following result,
	\[\mathrm{Ind}(\mathcal{N})\le \sum_{i =1}^{l}\mathrm{Ind}
		(\mathcal{N}_i)+ \mathrm{Ind}(\overline{T})+ \mathrm{dim}(F_1).\]

	Note that $\mathcal{L}$ is negative definite on $W$, then we know $-\mathcal{L}(\cdot,\cdot)$ is an inner product on $W$.
	Hence, the projection $\text{Proj}_{\mathcal{L}}:C_{0,1}^{\infty}(\mathcal{N})\rightarrow W$ with respect to the bilinear form $\mathcal{L}$ is well-defined.

	Let $\overline{W}$ be the maximal subspace of $C^\infty_{0,1}(\mathcal{N})$ such that $\mathcal{L}$ is negative definite on it.

	We will show that the projection $\text{Proj}_\mathcal{L}: \overline{W}\rightarrow W$ is a bijection.
	We note that $\text{Proj}_{\mathcal{L}}$ is onto by the definition of $W$ and $\overline{W}$. 
	Now, let us show $\text{Proj}_\mathcal{L}$ is one-to-one.
	If not, choose a non-zero $u \in \text{Ker} (\text{Proj}_\mathcal{L})$. This means $u\bot_{\mathcal{L}}W$. Let us consider the restriction $g:=u|_{\overline{\mathcal{P}}}$ on $\overline{\mathcal{P}}$. Since $g \in V_1(\overline{\mathcal{P}})$, we know $g$ can be decomposed to $g_1+g_2 \in \overline{V}_1(\overline{\mathcal{P}})\oplus F_1$. Let $u_1$ to be the $L$-extension of $g_1$, that is $u_1=u_{g_1}$. Now we write $u_2=u-u_1$.

	At first, let us check $u_1,u_2\bot_\mathcal{L} W$.
	We only need to check $u_1\bot _\mathcal{L}W$ since $u \bot_\mathcal{L}W$.
	\begin{itemize}
		\item $u_1\bot_\mathcal{L} \mathcal{J}_0^-(\mathcal{N}_i)$ for each $i$ is clear by the definition of $L$-extension.
		\item $u_1\bot_\mathcal{L} E^-_L(\overline{T})$ is done by the following arguments. 
			
			Using $u\bot_{\mathcal{L}}W$, for any $h \in E_\sigma$, let $u_h$ be the $L$-extension of $h$, then we have
			\[
				0=\mathcal{L}(u,u_h)=(u|_{\overline{\mathcal{P}}},T(h))=\sigma \int_{ \overline{\mathcal{P}}} uh= \sigma(g,h).
			\]
			
			So $g\bot E^-(\overline{T})$ by the arbitrary choice of $\sigma<0$. 
			Note that $g_2\bot E^-(\overline{T})$ automatically since $g_2 \in D_\tau\mathcal{J}_0^0(\overline{\mathcal{P}})$. So $g_1 \bot  E^-(\overline{T})$.
			By the properties of $L$-extension, we know $u_1\bot_\mathcal{L} \hat{E}_L^- (\overline{T})$ (cf., from the identity (\ref{eq:D_N_map_two_forms})).
		\item $u_1\bot_\mathcal{L} \hat{F}_1$ since $u_1|_{\overline{\mathcal{P}}} \in (D_\tau\mathcal{J}_0^0(\overline{\mathcal{P}}))^\bot $ and $L$-extension properties.
	\end{itemize}

	So we will find $u_2\bot _{\mathcal{L}}W$.

	Again, we can see $\mathcal{L}(u_1,u_2)=0$ by the $L$-extension properties of $u_1$.

	Now by the variational characterization of the negative eigenspace of $\overline{T}$, we know $\mathcal{L}(u_1,u_1)\ge 0$ since $u_1 \bot_\mathcal{L} \hat{E}^-(\overline{T})$.

	If $u_2| _{\overline{\mathcal{P}}}\equiv 0$, then by the first case of Lemma \ref{lem:1_dim_extension}, we know $\mathcal{L}(u_2,u_2)\ge 0$. On the other hand, if $u_2|_{\overline{\mathcal{P}}}\neq 0$ on $\overline{\mathcal{P}}$, since $g_2 \in F_1$, we can find an element $w \in \hat{F}_1$ such that $w=u_2$ on $\overline{\mathcal{P}}$.
	So, we can get
\begin{align*}
	0\le {} & \mathcal{L}(u_2-w,u_2-w)\quad (\text{the first case of Lemma \ref{lem:1_dim_extension}})\\
	={} &\mathcal{L}(u_2,u_2)+\mathcal{L}(w,w)\quad (u_2\bot _{\mathcal{L}}W)\\
	<{}&\mathcal{L}(u_2,u_2)\quad (\mathcal{L}(w,w)<0).
	\end{align*}

	Now use $u_1\bot_\mathcal{L}u_2$, we have $\mathcal{L}(u,u)=\mathcal{L}(u_1,u_1)+\mathcal{L}(u_2,u_2)\ge 0$, which contradicts the fact that $\mathcal{L}$ is negative definite on $\overline{W}$.

	So we finished the proof.
\end{proof}

Similarly, we have the following theorem.
\begin{theorem}
	[Nullity Theorem for networks]
	Suppose $\mathcal{N}$ is a network with partition $\{ \mathcal{N}_i \}_{i=1,\cdots ,l}$.

	The nullity of $\mathcal{N}$ can be computed by 
	\[
		\mathrm{Nul}(\mathcal{N})= \mathrm{Nul}(\overline{T})+ \mathrm{dim}(F_2)+ \sum_{i =1}^{l}\mathrm{dim}(\mathcal{I}_0
		(\tilde{\mathcal{Q}}_i)).
	\]
	\label{thm_nullity_theorem_for_network}
\end{theorem}
\begin{proof}
	The proof for this theorem is similar to the above proof, and it is a bit easier.\\

	\noindent\textbf{First part.} We show
	\[
		\mathrm{Nul}(\mathcal{N})\ge \mathrm{Nul}(\overline{T})+\mathrm{dim}
		(F_2)+\sum_{i =1}^{l}\mathrm{dim}(\mathcal{I}_0
		(\tilde{\mathcal{Q}}_i)).
	\]
	
	Firstly, we define 
	\[\hat{E}^0(\overline{T}):= \{ u_g \in C^\infty_{0,1}(\mathcal{N}): g \in E^0(\overline{T}), u_g \text{ is the }L \text{-extension of }g\}.\]
	Using formula (\ref{eq:D_N_map_two_forms}) and $(\overline{T}(g),v|_{\overline{\mathcal{P}}})=0$ for any $g \in E^0(\overline{T})$, $v \in C^\infty_{0,1}(\mathcal{N})$, we know $\mathcal{L}$ indeed vanishes on the space $\hat{E}^0(\overline{T})$.

	Secondly, for any element $g \in F_2$, we can choose an element $v_g \in C^\infty_{0,1} (\mathcal{N})$ such that $v_g|_{\overline{\mathcal{P}}} \equiv 0$, $g=\frac{\partial v_g}{\partial \tau} |_{\overline{\mathcal{P}}}$ and $v_g|_{\mathcal{N}_i}\in \mathcal{J}_0^0 (\mathcal{N}_i)$. $v_g$ is unique up to an addition of element in the set $\bigoplus_{i=1}^{l}\mathcal{I}_0(\tilde{\mathcal{Q}}_i)$ by the Fredholm alternative.
	So we always choose a unique $v_g$ such that $v_g\bot \bigoplus_{i=1}^{l}\mathcal{I}_0(\tilde{\mathcal{Q}}_i)$.
	Recall that $v_g\bot \bigoplus_{i=1}^{l}\mathcal{I}_0(\tilde{\mathcal{Q}}_i)$ means $v_g$ is orthogonal to $\bigoplus_{i=1}^{l}\mathcal{I}_0(\tilde{\mathcal{Q}}_i)$ with respect to the standard inner product in $L^2(\mathcal{N})$.
	Again, since $g \in V_2(\overline{\mathcal{P}})$, we know $v_g$ will satisfy 
	\[
		\mathcal{L}(v_g,v)=\int_{ \overline{\mathcal{P}}} 
		v \frac{\partial v_g}{\partial \tau}=\int_{ \overline{\mathcal{P}}} vg=0,
	\]
	for any $v \in C^\infty_{0,1}(\mathcal{N})$. So if we write $\hat{F}_2:=\{ v_g: g \in F_2\}$, we know $\mathcal{L}$ vanishes on the space $\hat{F}_2$.

	At last, for any element $u \in \mathcal{I}_0 (\tilde{\mathcal{Q}} _i)$, we have
	\[
		\mathcal{L}(u,v)=\int_{ \overline{\mathcal{P}}} 
		v \frac{\partial u}{\partial \tau}- \int_{ \mathcal{N}} vLu=0,
	\]
	for any $v \in C^\infty_{0,1}(\mathcal{N})$.
	Hence $\mathcal{L}$ vanishes on the space $\mathcal{I}_0(\tilde{\mathcal{\mathcal{L}}}_i)$.

	We can easily see that $\hat{E}^0, \hat{F}_2, \mathcal{I}_0(\tilde{\mathcal{Q}}_i )$ are linearly independent. 
	So we know $\mathrm{Nul}(\mathcal{N})\ge \mathrm{Nul}(\overline{T}) +\mathrm{dim}(F_2)+\sum_{i =1}^{l} \mathrm{dim}(\mathcal{I}_0(\tilde{\mathcal{Q}}_i))$.

	We write $W=\hat{E}^0\oplus \hat{F}_2\oplus \bigoplus_{i=1}^{l}\mathcal{I}_0(\tilde{\mathcal{Q}}_i)$.\\

	\noindent\textbf{Second part.} We show
	\[
		\mathrm{Nul}(\mathcal{N})\le \mathrm{Nul}(\overline{T})+\mathrm{dim}
		(F_2)+\sum_{i =1}^{l}\mathrm{dim}(\mathcal{I}_0
		(\tilde{\mathcal{Q}}_i)).
	\]

	Let $\overline{W}$ be the null space of $\mathcal{L}$. In the first part, we have seen $W \subset \overline{W}$.
	Now let us choose $w \in \overline{W}\backslash W$. Then $w \in C^{\infty}_{0,1}(\mathcal{N})$ will solve the following problem,
	\begin{equation}
		\begin{cases}
		Lw=0, & \text{ on }\mathcal{N}, \\
		w=0, & \text{ on }\mathcal{Q},\\
		\frac{\partial w}{\partial \tau} \in V_2(\mathcal{P}),\quad 
			 &\text{ on }\mathcal{P}.
		\end{cases}
		\label{eq:pfNullExtension}
	\end{equation}
	
	Restricting on $\mathcal{N}_i$, $w$ is still a solution to a Dirichlet problem with boundary condition $w|_{\tilde{\mathcal{Q}}_i}$. By the existence of solutions of problem (\ref{eq:Dirichlet_problem}), we know $w|_{\tilde{\mathcal{Q}}_i} \in (D_\tau\mathcal{J}_0^0(\tilde{\mathcal{Q}}_i))^\bot $. Combining $w|_{\overline{\mathcal{P}}}\in V_1(\overline{\mathcal{P}})$, we have $w|_{\overline{\mathcal{P}}} \in \overline{V}_1(\overline{\mathcal{P}})$. So $\overline{T}(w|_{\overline{\mathcal{P}}})$ is well-defined.

	Up to an addition of an element in $\hat{E}^0$, we can suppose $w|_{\overline{\mathcal{P}}}\bot E^0(\overline{T})$.
	Again up to an addition of an element in $\hat{F}_2$, we can suppose
	\[
		\frac{\partial w}{\partial \tau}|_{\overline{\mathcal{P}}}
		\in \overline{V}_2(\overline{\mathcal{P}}),
	\]
	since $\left.\frac{\partial w}{\partial \tau} \right|_{\overline{\mathcal{P}}}\in V_2(\overline{\mathcal{P}})$ and $V_2(\overline{\mathcal{P}})=\overline{V}_2(\overline{\mathcal{P}})\oplus F_2$. This step does not change the values of $\left.w\right|_{\overline{\mathcal{P}}}.$
	Moreover, up to an addition of elements in each $\mathcal{I}_0(\tilde{\mathcal{Q}}_i)$, we suppose $w\bot \mathcal{I}_0(\tilde{\mathcal{Q}}_i)$.
	This step does not change the values of $\left.w\right|_{\overline{\mathcal{P}}}$, too.
	In summary, after these operations on $w$, we can check $w$ satisfies the following conditions, 
	\begin{itemize}
		\item $w|_{\tilde{\mathcal{Q}}_i} \in (D_\tau\mathcal{J}_0^0(\tilde{\mathcal{Q}}_i))^\bot $. (Since $w |_{\overline{\mathcal{P}}}\in \overline{V}_1(\overline{\mathcal{P}})$.)
		\item $\left.\frac{\partial w}{\partial \tau}\right|_{\tilde{\mathcal{Q}}_i}\in (D_\tau\mathcal{J}_0^0(\tilde{\mathcal{Q}}_i))^\bot $. (Since $\left.\frac{\partial w}{\partial \tau}\right|_{\overline{\mathcal{P}}} \in \overline{V}_2(\overline{\mathcal{P}})$.)
				\item $w|_{\mathcal{N}_i}\bot \mathcal{I}_0(\tilde{\mathcal{Q}}_i)$. 
	\end{itemize}
	So by the definition of $L$-extension, $w$ is an $L$-extension of $w|_{\overline{\mathcal{P}}}$.

	On the other hand, $\text{Proj}_1(\frac{\partial w}{\partial \tau})=0$ since $\left.\frac{\partial w}{\partial \tau}\right|_{\overline{\mathcal{P}}}\in V_2(\overline{\mathcal{P}})$ by \eqref{eq:pfNullExtension}.
	By the definition of $\overline{T}$, we know $\overline{T}(w|_{\overline{\mathcal{P}}})=0$.
	Hence $w|_{\overline{\mathcal{P}}}\in E^0(\overline{T})$. The only possible way is $w|_{\overline{\mathcal{P}}}\equiv 0$ on $\overline{\mathcal{P}}$. But right now the unique $L$-extension for $0$ is 0 itself, so $w\equiv 0$ on $\mathcal{N}$. This means for any $w \in \overline{W}\backslash W$, $w$ will be a linear combination of elements in $W$, so $\overline{W}=W$. So the null space of $\mathcal{L}$ is indeed $W$ itself.

	In conclusion, we have $\mathrm{Nul}(\mathcal{N})= \mathrm{Nul}(\overline{T})+\mathrm{dim} (F_2)+\sum_{i =1}^{l}\mathrm{dim}(\mathcal{I}_0 (\tilde{\mathcal{Q}}_i))$.
	
	Moreover, we have
	\[
		\mathcal{J}_0^0(\mathcal{N})
		\cong W\cong \hat{E}^0\oplus \hat{F}_2\oplus \bigoplus_{i=1}^{l}\mathcal{I}_0(\tilde{\mathcal{Q}}_i).
	\]
\end{proof}

Note that we have $\mathrm{dim}(D_\tau\mathcal{J}_0^0 (\tilde{\mathcal{Q}}_i))+ \mathrm{dim}(\mathcal{I}_0 (\tilde{\mathcal{Q}}_i))= \mathrm{dim}(\mathcal{J}_0^0 (\mathcal{N}_i))=\mathrm{Nul}(\mathcal{N}_i)$, as a corollary, we have the following result.
\begin{corollary}
	Suppose $\mathcal{N}$ is a network with the partition $\{ \mathcal{N}_i \}_{i=1,\cdots ,l}$.
	Then we have
	\[
		\mathrm{Ind}(\mathcal{N})+ \mathrm{Nul}(\mathcal{N})= \sum_{i =1}^{l}\left( \mathrm{Ind}
		(\mathcal{N}_i)+\mathrm{Nul}
	(\mathcal{N}_i)\right) + \mathrm{Ind}(\overline{T})+ \mathrm{Nul}(\overline{T}).
	\]
	\label{cor:index_estimate}
\end{corollary}

\section{Index and nullity of the stationary triple junction networks on spheres}%
\label{sec:index_and_nullity_of_the_equiangular_networks_on_sphere}

We say a network $\mathcal{N}$ is a triple junction network if all points in $\mathcal{P}$ are triple junction points and $\mathcal{Q}\equiv\emptyset $. This means $\mathcal{N}$ is closed.

The classification of all equiangular nets on the sphere was given by Heppes \cite{Heppes64}. There are precisely 10
such $\frac{2\pi}{3}$ nets. Based on our definition, we know all the stationary triple junction networks are contained in that list.

The first one net is trivial. It is the great circle of $\mathbb{S}^2$.
It is not a triple junction network based on our definition.

Apart from the trivial one, the left nine nets can also be viewed as stationary triple junction networks here.

Following the notation of polyhedra, we define the following notation.
\begin{itemize}
	\item $V=\sharp(\mathcal{P})$, the number of triple junction points (the number of vertices).
\item $E=\sharp(\mathcal{N})$, the number of geodesics in the network (the number of edges).
\item $F$, the number of regions after dividing $\mathbb{S}^2$ by network $\mathcal{N}$ (the number of faces).
\end{itemize}

By Euler's formula, we have $V-E+F=2$.

Our main theorem can be stated as follows.
\begin{theorem}
	[Index and nullity of embedded stationary triple junction networks]
	Let $\mathcal{N}$ be an embedded stationary triple junction network on sphere $\mathbb{S}^2$.
	Let $V,E,F$ be defined as above. Then the (Morse) index of $\mathcal{N}$ is $F-1$, and the nullity of $\mathcal{N}$ is 3 (with respect to the stability operators).

	More precisely, for the operator $L(u)=\Delta u+u$, $L$ has only one negative eigenvalue $\lambda=-1$ (without multiplicity) and its eigenspace consisting of the locally constant functions on $\mathcal{N}$.

	Besides, all of the eigenfunctions with eigenvalue 0 are generated by the rotation on spheres.
	\label{thm_index_and_nullity_of_triple_junction_networks}
\end{theorem}

Here, a locally constant function means a function $u \in C^\infty_{1}(\mathcal{N} )$ that is constant on each geodesic $\gamma^i$.

The proof of Theorem \ref{thm_index_theorem_for_network} is followed by the following two propositions.

\begin{proposition}
	For an embedded stationary triple junction network $\mathcal{N}$ in $\mathbb{S}^2$, we have 
	\[
		\mathrm{Ind}(\mathcal{N})\ge F-1,\quad \mathrm{Nul}(\mathcal{N})\ge 3.
	\]
	\label{prop:lower_bound}
\end{proposition}
\begin{proof}
	Still, we use $L=\Delta+1$ to denote the operator we are interested in. So we only need to show $\mathrm{dim}(\mathcal{J}_{-1}(L) )\ge F-1$ . Hence we will have $\mathrm{Ind}(\mathcal{N})\ge \mathrm{dim}(\mathcal{J}_{-1})=F-1$.

	Note that the functions in $\mathcal{J}_{-1}(L)$ are the locally constant functions. So indeed, we only need to find some locally constant function $u=(c^1,\cdots ,c^E) \in C^\infty_{1} (\mathcal{N})$ that span a linear subspace of dimension at least $F-1$.

	Indeed, we have the following result.
	\begin{lemma}
		\label{lem_locallyConst}
		For each face $F_j$ of $\mathcal{N}$ (a region in $\mathbb{S}^2$ with boundary consisting of the curves in $\mathcal{N}$), there exists a locally constant function $u_{F_j}$ on $\mathcal{N}$ such that $u_{F_j}$ is not identically zero on $F_j$.
	\end{lemma}
	\begin{proof}
	For any $\gamma^i$ which is a part of boundary $F_j$, we choose $c^i=1$ if $\nu^i$ pointing outside of $F_j$ and $c^i=-1$  if $\nu^i$ pointing inside of $F_j$.
For any $\gamma^i$ which is not a part of boundary $F_j$, we just choose $c^i=0$. 

\begin{figure}[ht]
	\centering 
	\begingroup
	\fontsize{7pt}{12pt}
	\def\svgwidth{0.6\columnwidth}
	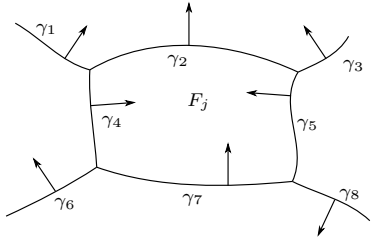
	\endgroup

    \caption{A part of a triple junction network}
    \label{fig:choiceofu}
\end{figure}

	Here is an example of how to choose $u$ on the network.
	In Figure \ref{fig:choiceofu}, we choose the face $F_j$ (quadrilateral) as shown in this figure, then the choice of $u$ should be in the form
	\[
		c^2=1, c^4=c^5=c^7=-1, c^i=0
		\text{ otherwise. }
	\]
	
	Based on this choice of $u$, we can check $u|_{\mathcal{P}} \in V_1(\mathcal{P})$ (we only need to check it on the vertexes of $F_j$ indeed) and hence $u \in C^\infty_{1}(\mathcal{N})$.
	\end{proof}
	So at least, we find $F$ different elements $\{ u_1,\cdots ,u_{F} \}$ in the space $\mathcal{J}_{-1}(L)$, where $u_j$ comes from the previous construction for the face $F_j$.
	We do not know if they are linearly independent (actually they are not since $\sum_{i =1}^{F}u_i\equiv 0$).
	Now we define
	\[V:=\{ (a_1,\cdots ,a_F)
	\in \mathbb{R} ^F: a_1u_1+\cdots a_Fu_F=0\}.
\] 

	Here, we will view $u_j=(c_j^1,\cdots ,c_j^E)$ as an element in $\mathbb{R} ^F$ for each $j$.

If $\mathrm{dim}(V)\le1$, then we know $\mathrm{dim}\mathcal{J}_{-1}(L)\ge F-1$.
On the other hand, if $\mathrm{dim}(V)\ge 2$, we know there is at least one element $\alpha=(\alpha_1,\cdots ,\alpha_F)\neq 0$ in $V$ such that at least one component of $\alpha$, is $0$, i.e. $\alpha_j=0$ for some $j$. Let $M$ be the largest number such that there is at least one non-zero element in $V$ with exact $M$ components taking 0. 

Suppose this element written as $\alpha=(0,\cdots , 0, \alpha_{M+1},\cdots , \alpha_F)$ after relabeling of faces.
Now let us choose a curve $\gamma^i$ in the boundary of $F_1 \cup \cdots \cup F_M$ and $F_{M+1}\cup \cdots \cup F_F$. This is possible since $F_1 \cup \cdots \cup F_M$ is not the whole sphere since $M<F$.
Note that there are exactly two $j_1,j_2$ such that $c^i_{j_1},c^i_{j_2} \neq 0$ for such fixed $i$ (This is because each edge is incident with exactly 2 faces). By the choice of $\gamma^i$, we can suppose $j_1\le M, j_2>M$.
Now from $\alpha_{M+1}u_{M+1}+\cdots +\alpha_Fu_M=0$, we have $\alpha_{j_2}c^{i}_{j_2}=0$, which implies $\alpha_{j_2}=0$. This contradicts the choice of $\alpha$.
Hence, we should have $\mathrm{dim}(V)\le 1$.

So we have shown $\mathrm{Ind}(\mathcal{N})\ge F-1$.

Now let us show $\mathrm{Nul}(\mathcal{N})\ge 3$.
Since the rotation on a sphere is an isometry, if we let $X$ be the vector field from this rotation, then the function defined by $u=X\cdot \nu$ should satisfy $Lu=0$ and $u \in C^\infty_{1}(\mathcal{N})$.
Hence $u$ is a function with eigenvalue 0.

We will construct at least three different linearly independent functions in $\mathcal{J}_0(L)$.
Let us choose a triple junction point $P \in \mathcal{P}$ first. We still use $\gamma^1_P,\gamma^2_P, \gamma^3_P$ to denote the curves adjacent to $P$.

First, we can consider the function $u_0$ generated by the rotation that fixes $P$. This $u$ will satisfy $u|_{P}=0$.
Second, we can consider the function $u_i$ generated by the rotation that fixes the great circle that $\gamma^i_P$ is contained. This $u_i$ will take the value zero along $\gamma^i$ and $u_i|_{P}$ is not identically zero. Based on these results, we know $u_0,u_1,u_2$ are linearly independent. 

So $\mathrm{Nul}(\mathcal{N})\ge 3$.
\end{proof}

\begin{remark}
	From the proof of Proposition \ref{prop:lower_bound}, we can find the geometric meaning of the index and nullity of a network.
	Usually, the index means we have a method to decrease the total length of the network in the sense of the second derivative.
	Then, Theorem \ref{thm_index_and_nullity_of_triple_junction_networks} tells us, the way to decrease the total length is generated by the deformation which shrinks or expands the faces of the network.
	For the nullity, we know the null space is generated by the rotations on $\mathbb{S}^2$.
\end{remark}

Now we come to the main part proof of the theorem.
Note that we will only give the key idea of how to use the main theorem to compute the index and nullity and put some detailed calculations in the appendix.

%In the following calculation, when the network is small, to be more precise, the number of geodesic arcs in this network is not greater than five, we will assume we have known the Index, Nullity, and Dirichlet-to-Neumann map on the boundary. This is because the small networks considered in the proof usually have some symmetric, and the calculation for specific eigenvalues and eigenfunctions is relatively easy although it is a bit tedious. 
%We will use the results from a small network to give an estimate for the index and nullity of a large network.

\begin{proposition}
	For any embedded stationary triple junction network in $\mathbb{S}^2$, we have
	\[
		\mathrm{Ind}(\mathcal{N})+ \mathrm{Nul}(\mathcal{N})
		\le F+2.
	\]
	\label{prop:upper_bound}
\end{proposition}
\begin{proof}
	We will use the index theorem and the nullity theorem to give the estimate of the index and nullity of $\mathcal{N}$.

	More precisely, Corollary \ref{cor:index_estimate} is enough in this proof.
	We divide our nine networks into five classes. We will use the classification results listed in \cite{Heppes64} (see also \cite{taylor1976structure}).
	In other words, we will assume the lengths of arcs contained in these nine networks are all known to us.
	\\

	\noindent\textbf{First class.} 
	We consider the network of three half-great circles having a pair of antipodal points as their endpoints.

	Since it has only three geodesic arcs, we can get its index and nullity by directly finding its eigenvalues and eigenfunctions.
	Indeed, all of the eigenfunctions on this network come from the eigenfunctions on a great circle. \\

	\noindent\textbf{Second class.}
	We consider the two following networks coming from the regular spherical polyhedron.

	\begin{itemize}
		\item One-skeleton of a regular spherical tetrahedron.
		\item One-skeleton of a cube.
		%\item One-skeleton of a regular spherical dodecahedron.
	\end{itemize}

	\begin{figure}[ht]
		\centering
\begin{subfigure}[b]{.33\linewidth}
			\centering %
	\begingroup
	\fontsize{7pt}{12pt}
	\def\svgwidth{1\columnwidth}
	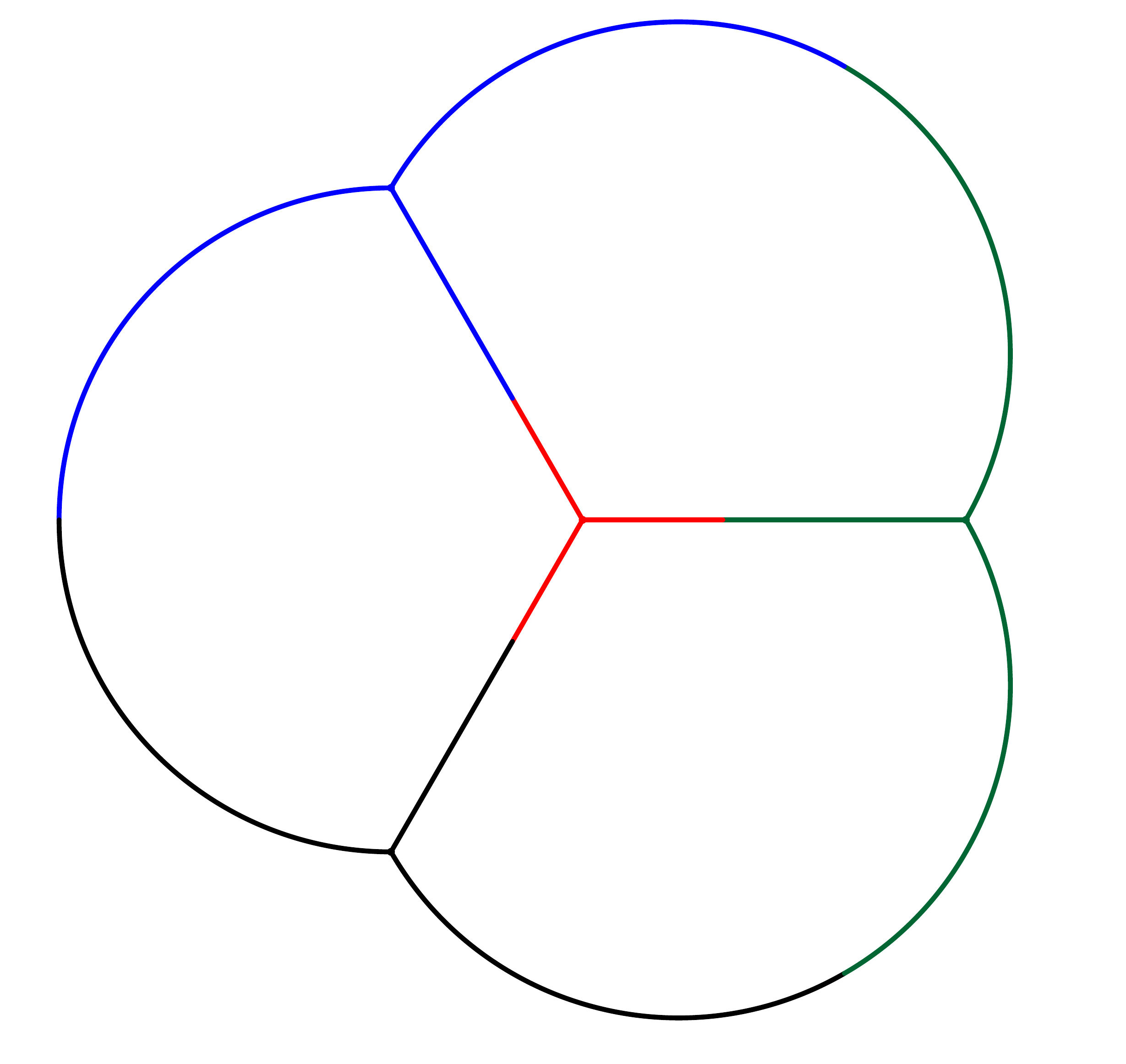
	\endgroup

			\caption{Partition of tetrahedron}
			\label{fig:sub_tetrahedron}
		\end{subfigure}
		\begin{subfigure}[b]{.33\linewidth}
			\centering %
	\begingroup
	\fontsize{7pt}{12pt}
	\def\svgwidth{1\columnwidth}
	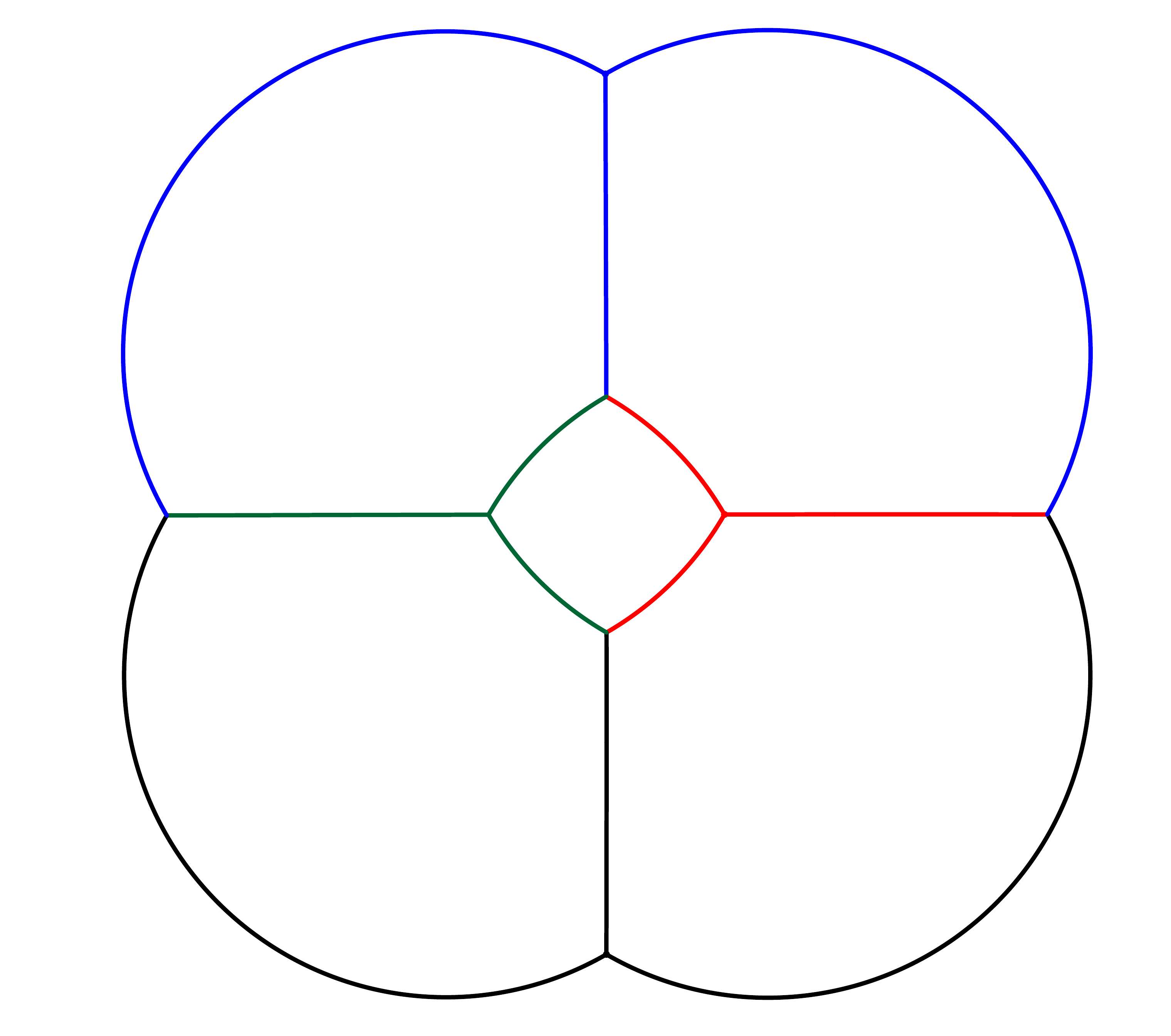
	\endgroup

			\caption{Partition of cube}
			\label{fig:sub_cube}
		\end{subfigure}
		%\begin{subfigure}[b]{.30\linewidth}
		%	\centering \incfig[1]{12-faces}
		%	\caption{Partition of dodecahedron}
		%	\label{fig:sub_doda}
		%\end{subfigure}
		\caption{Partition of one-skeleton of the regular spherical polyhedron}
		\label{fig:partition_regular}
	\end{figure}
	
	For $\mathcal{N}$ to be one of the above networks, we will divide it into four isometric networks $\mathcal{N}_1,\cdots ,\mathcal{N}_4$ shown in Figure \ref{fig:partition_regular}. (Different colors mean different networks.)
	If we can show $\mathrm{Ind}(\mathcal{N}_i) =\mathrm{Nul}(\mathcal{N}_i)=0$ for each $i$ for all these three networks, we can get 
	\[
		\mathrm{Ind}(\mathcal{N})+ \mathrm{Nul}(\mathcal{N})= \mathrm{Ind}(\overline{T})+ \mathrm{Nul}(\overline{T})
	\]
	by Corollary \ref{cor:index_estimate}.
	Note that we have $\mathrm{Ind}(\overline{T})+ \mathrm{Nul}(\overline{T})\le \mathrm{dim}(V_1(\overline{\mathcal{P}}))$, we only need to find out the dimension of $V_1(\overline{\mathcal{P}})$ to give the estimation.

	Indeed, we have the following results.
	\begin{itemize}
		\item For the tetrahedron type, we can see all the points in $\overline{\mathcal{P}}$ are double points, so $\mathrm{dim}V_1(P)=1$ for all $P \in \overline{\mathcal{P}}$.
			Hence $\mathrm{dim}V_1 (\overline{\mathcal{P}})= 6=F+2$.
		\item For the cube type, all $P \in \overline{\mathcal{P}}$ are triple points. So $\mathrm{dim}V_1(P)=2$.
			Hence $\mathrm{dim}V_1 (\overline{\mathcal{P}})= 2\times 4=8=F+2$.
		%\item For the dodecahedron type, there are four triple-points and six double-points in the set $\overline{\mathcal{P}}$.
			%So $\mathrm{dim}V_1 (\overline{\mathcal{P}})= 2\times 4+6=14=F+2$
	\end{itemize}

	The only thing we have left is to check if each network satisfies
\begin{align}\mathrm{Ind}(\mathcal{N}_i)= \mathrm{Nul}(\mathcal{N}_i)=0
\label{eq:pf_ind_nul}.\end{align}
	For the tetrahedron type and cube type, 
	based on the computation we have done in Example \ref{ex:3nets}, we know $\mathrm{Ind}(\mathcal{N}_i)=\mathrm{Nul}(\mathcal{N}_i)=0$ by Corollary \ref{cor_3nets}.\\

	%For the case of dodecahedron type, each $\mathcal{N}_i$ has nine arcs.
	%It is not easy to directly compute its index and nullity.
	%But we note we can divide it into three small isometric networks further. For example, for the blue network in Figure \ref{fig:sub_doda}, we cut it into three pieces $\mathcal{N}_1=\cup _{j=1}^3 \mathcal{N}_{1j}$ along the point $P$, each small piece has both zero index and nullity according to the result in Example \ref{ex:3nets}.
	%So the index and nullity of network $\mathcal{N}_1$ will come from the D-N map on $V_1({P})$.
	%This can be done using the explicit form of Dirichlet-to-Neumann map in Example \ref{ex:3nets}.
	%We only need to choose $l_0=\arccos \frac{\sqrt{5}}{3}$ and set $l_1=l_2=l_0,l_3=\frac{l_0}{2}$ in Example \ref{ex:3nets}.
	%Note that the D-N map on $P$ for each small piece is positive(direct computational), so the D-N map on $V_2(P)$ is positively defined.
	%This shows equation (\ref{eq:pf_ind_nul}) holds for $\mathcal{N}_i$.

	%So for the networks in this class we have $\mathrm{Ind}(\mathcal{N}) +\mathrm{Nul}(\mathcal{N})\le F+2$.\\

	\noindent\textbf{Third class.}
	We consider the networks of prism type. There are two types of them (Actually three if we count the cube as a prism, too),
	\begin{itemize}
		\item One-skeleton of the prism over a regular triangle.
		\item One-skeleton of the prism over a regular pentagon.
	\end{itemize}

	We will divide this network into two isometric parts by a great circle with the same axis as this prism.
	We write it as $\mathcal{N}= \mathcal{N}_1 \cup  \mathcal{N}_2$.

	We only focus on $3$-prism (the prism over a regular triangle) since the method here can be applied to $5$-prism, too.

	\begin{figure}[ht]
		\centering
\begin{subfigure}[b]{.29\linewidth}
			\centering %
	\begingroup
	\fontsize{7pt}{12pt}
	\def\svgwidth{0.72\columnwidth}
	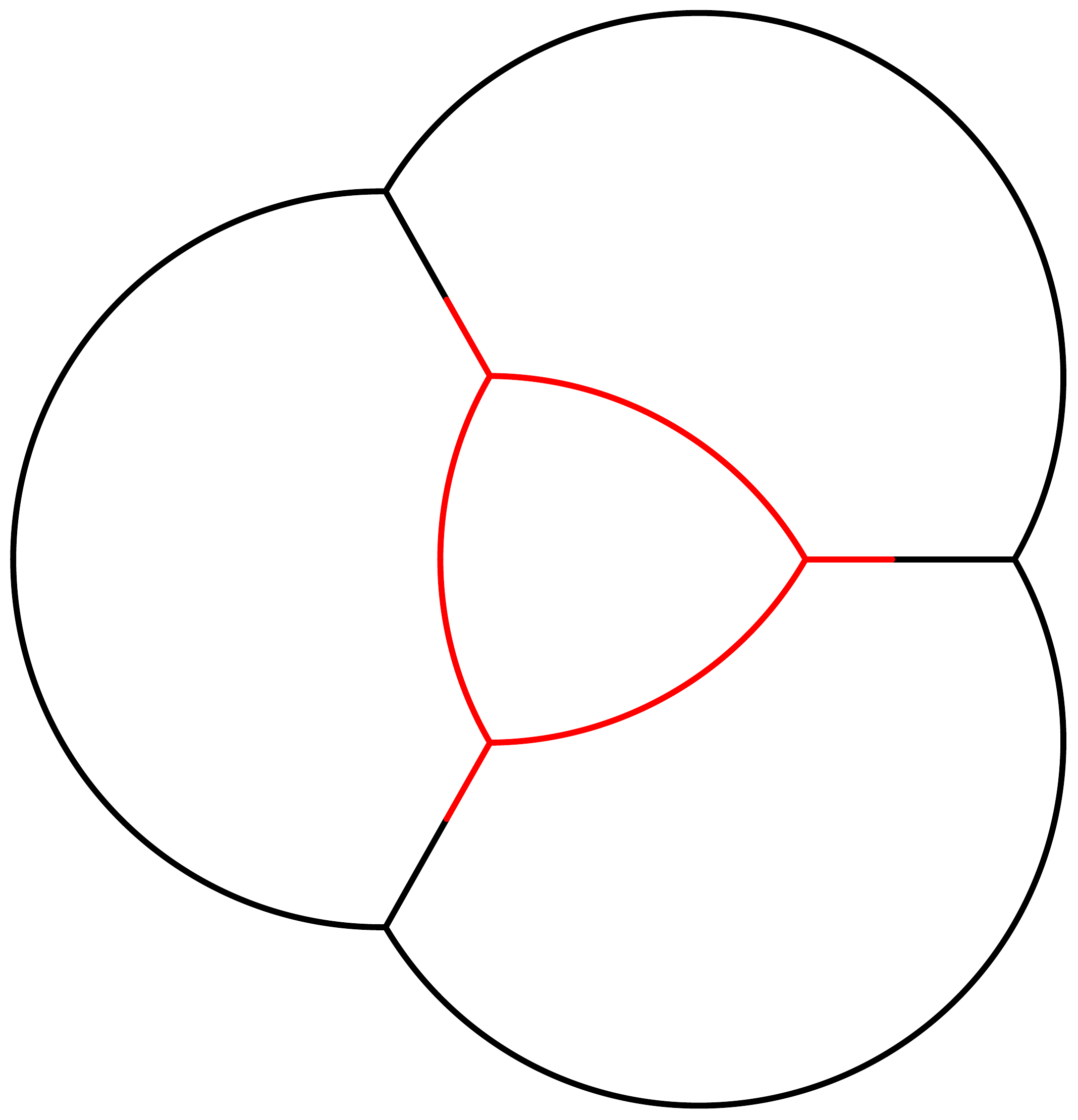
	\endgroup

			\caption{Partition of 3 prism type}
			\label{fig:sub_3_prism}
		\end{subfigure}
		\begin{subfigure}[b]{.65\linewidth}
			\centering %
	\begingroup
	\fontsize{7pt}{12pt}
	\def\svgwidth{0.72\columnwidth}
	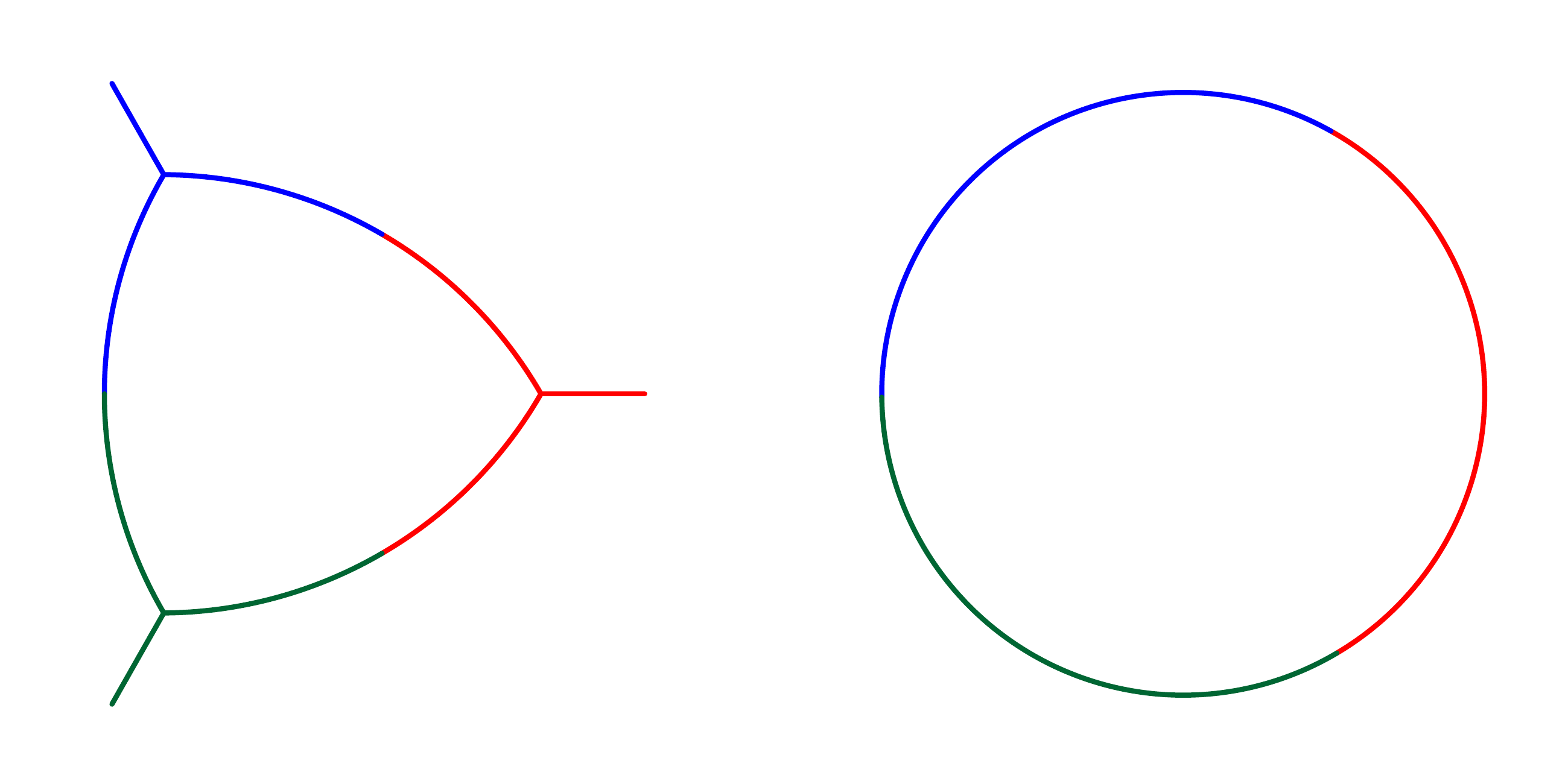
	\endgroup

			\caption{Comparison of two Dirichlet-to-Neumann map at interior}
			\label{fig:3_prism_cut}
		\end{subfigure}
		\caption{Partition of prism-type networks}
		\label{fig:partition_3_prism}
	\end{figure}

	For the subnetwork $\mathcal{N}_1$, we want to compute its index and nullity with zero Dirichlet boundary condition and want to know how the Dirichlet-to-Neumann map is defined on the boundary.

	In order to do that, we will still consider dividing it into three isometric parts as $\mathcal{N}_1= \mathcal{N}_{11}\cup \mathcal{N}_{12} \cup \mathcal{N}_{13}$ and try to compute the Dirichlet-to-Neumann map on set $\overline{\mathcal{P}} =\{ P_1,P_2,P_3 \}$. 
	See Figure \ref{fig:3_prism_cut} for the partition and the choice of orientation for each curve. Since $\mathcal{N}_{1i}$ is stable and has zero null space by Corollary \ref{cor_3nets}, we only need to compute the index and nullity of the Dirichlet-to-Neumann map on $\overline{\mathcal{P}}$. 

	Let us focus on network $\mathcal{N}_{11}$ first. % it will define a Dirichlet-to-Neumann map $T_{11}$ on set $\{ P_1,P_2 \}$. Moreover, we can write $T_{11}:\mathbb{R}^2 \rightarrow \mathbb{R}^2 $ as
	%\[
	%	T(
	%\begin{bmatrix}
	%		v_1\\
	%		v_2\\
	%	\end{bmatrix}
	%	)=
	%\begin{bmatrix}
	%		a & b\\
	%		b & a \\
	%	\end{bmatrix}
	%	\begin{bmatrix}
	%		v_1 \\
	%		v_2 \\
	%	\end{bmatrix}.
	%\]
	%by the symmetric properties of $\mathcal{N}_{11}$.
	%($P_1,P_2$ symmetric to each other by a reflection.) 
	By taking $l_1=l_2=\arccos\frac{1}{\sqrt{3}},l_3=\arccos \frac{2\sqrt{2}}{3}$ in \eqref{eq:matrixM3arc}, we can find the matrix form $M$ of the Dirichlet-to-Neumann map $\mathcal{N}_{11}$ is given by
	\[
		M=
		\begin{bmatrix}
			-\frac{\sqrt{2}}{3} & \frac{2\sqrt{2}}{3} & * \\
			\frac{2\sqrt{2}}{3} & -\frac{\sqrt{2}}{3} & * \\
			* & * & * \\
		\end{bmatrix}.
	\]
	The $*$ means the values that we do not need to compute.
	Then we know the Dirichlet-to-Neumann map on $\left\{ P_1,P_2 \right\}$ is given by
	\[
		\tilde{M}=
		\begin{bmatrix}
			-\frac{\sqrt{2}}{3} & \frac{2\sqrt{2}}{3} \\
			\frac{2\sqrt{2}}{3} & -\frac{\sqrt{2}}{3} \\
		\end{bmatrix}.
	\]
	%The exact values of $a,b$ can be obtained from matrix $M$ in \eqref{eq:matrixM}.
	Now if we take another network $\mathcal{N}'_{11}$, which contains only one geodesic arc with length $\frac{2\pi}{3}$ and compute its matrix form $M'$ of Dirichlet-to-Neumann map on boundary. 
	By Example \ref{ex:6-3subnet}, we know $M'$ is given by
	\[
		M'=
		\begin{bmatrix}
			-\frac{1}{\sqrt{3}} & \frac{2}{\sqrt{3}}\\
			\frac{2}{\sqrt{3}} & -\frac{1}{\sqrt{3}}\\
		\end{bmatrix}.
	\]
	
	Note that $\tilde{M}=\sqrt{\frac{2}{3}}M'$.
	%$T'$ still has a matrix form, which we will write as $
	%\begin{bmatrix}
	%	a' & b'\\
	%	b' & a' \\
	%\end{bmatrix}$.
	%%(Actually one can get $a'= \cot(l),b'= -\csc(l)$ for length $l$ arc.) 
	%The interesting part is, we can check $\frac{a}{a'}= \frac{b}{b'}>0$ so $T$ actually equals to $\alpha T'$ for some $\alpha > 0$.
	So we can do a comparison of partition of $\mathcal{N}_1=\mathcal{N}_{11} \cup \mathcal{N}_{12} \cup \mathcal{N}_{13}$  and partition of a unit circle $\mathcal{N}_1'= \mathcal{N}_{11}' \cup \mathcal{N}_{12}' \cup \mathcal{N}_{13}'$ as shown in Figure \ref{fig:3_prism_cut} where $\mathcal{N}_{1i}'$ will be geodesic arc with length $\frac{2\pi}{3}$.
	Since they have the same way to form a big network, they will have the same index and nullity. In particular, $\mathcal{N}_1$ will have index one and nullity two.

	%Another way to see $\mathcal{N}_1$ has nullity two is, we can consider the rotation on $\mathbb{S}^2$ who fixes the point $Q_1$ or $Q_2$.
	%Such rotation will generate a function $u$ defined on $\mathcal{N}$ which equals zero on $Q_i$ for $i=1,2,3$.

	Now since $\mathcal{N}_1$ has nullity two, the Dirichlet-to-Neumann map $T_1$ of $\mathcal{N}_1$ is defined on $(D_\tau\mathcal{J}_0^0( \mathcal{Q} ))^\bot $, which has dimension 1. 
	Note that since $Q_1,Q_2,Q_3$ are double points, and $\mathcal{N}$ is symmetric with respect to the reflection on the plane containing $Q_1,Q_2,Q_3$, we know the space $V_1(\mathcal{P})$ is isomorphic to $V(\left\{ Q_1,Q_2,Q_3 \right\})$.
	Moreover, $T_1$ is same with $\overline{T}$ up to the isomorphism mentioned above.
	Hence, $\overline{T}$ is defined on a space of dimension 1.
	%Same things hold for $\mathcal{N}_2$ and hence, the interior Dirichlet-to-Neumann map $\overline{T}$ for partition $\mathcal{N}= \mathcal{N}_1 \cup \mathcal{N}_2$ is defined on a space of dimension 1.
	%(Indeed, by the symmetric property, we know $\overline{T}$ should be same with $T$ when we restrict it on the boundary of $\mathcal{N}$.) 
	So we know
\begin{align*}
		\mathrm{Ind}(\mathcal{N})+ \mathrm{Nul}(\mathcal{N})
		={}&(\mathrm{Ind}(\overline{T})+ \mathrm{Nul}(\overline{T}))+ \sum_{i =1}^{2}
		\mathrm{Ind}(\mathcal{N}_i)
		+\mathrm{Nul}(\mathcal{N}_i)\\
		\le{}&
		1+2\times (1+2)=7=F+2.
	\end{align*}

\begin{remark}
	\label{rmk_dimVbar}
	Indeed, since $\mathcal{N}_2$ equals to $\mathcal{N}_1$ after a reflection on the plane containing $Q_1,Q_2,Q_3$, we can find the dimension of the space $\overline{V}_1(\overline{\mathcal{P}})$ with respect to the partition $\mathcal{N}= \mathcal{N}_1 \cup \mathcal{N}_2$ is given by
	\[
		\mathrm{dim}\overline{V}_1(\overline{\mathcal{P}})=\mathrm{dim}V_1(\overline{\mathcal{P}})-\mathrm{dim}(D_\tau\mathcal{J}_0^0(\mathcal{Q}_i))=1,
	\]
	in view of \eqref{eq:V1decomp} and definition of $F_1$, where $i=1$ or $2$.
\end{remark}

	A similar method can be applied to 5-prism type networks. This time, we will know each $\mathcal{N}_i$ will still have index 1 and nullity 2 coming from the index and nullity of a unit circle. So $\overline{T}$, in this case, will be defined in the space of dimension 3.
	This will give
	\[
		\mathrm{Ind}(\mathcal{N})+ \mathrm{Nul}(\mathcal{N})\le 3+2\times 3=9=F+2
	\]
	for 5-prism type case.\\
	
	\noindent\textbf{Fourth class.}
	For the networks in this class, we do not have a uniform way to cut them into several isometric parts. So we will analyze them one by one.
	%We put all the remaining networks in this class. We need to analyze them one by one since they are quite different from each other and all complicated.
	There are two of them.
	\begin{itemize}
		\item One skeleton of a spherical polyhedron with four equal pentagons and four equal quadrilaterals. We call it a 4-4 type.
		\item One skeleton of a spherical polyhedron with six equal pentagons and three regular quadrilaterals.
			We call it a 6-3 type.
		%\item One skeleton of a spherical polyhedron with eight equal pentagons and two regular quadrilaterals. We call it the 8-2 type.
	\end{itemize}
	
	1. For $\mathcal{N}$ to be 4-4 type, we will still divide it into two isometric parts $\mathcal{N}_1$ and $\mathcal{N}_2$ as shown in Figure \ref{fig:8faces} by a great circle.

	\begin{figure}[ht]
		\centering
		\begin{subfigure}[b]{.61\linewidth}
			\centering %
	\begingroup
	\fontsize{7pt}{12pt}
	\def\svgwidth{0.94\columnwidth}
	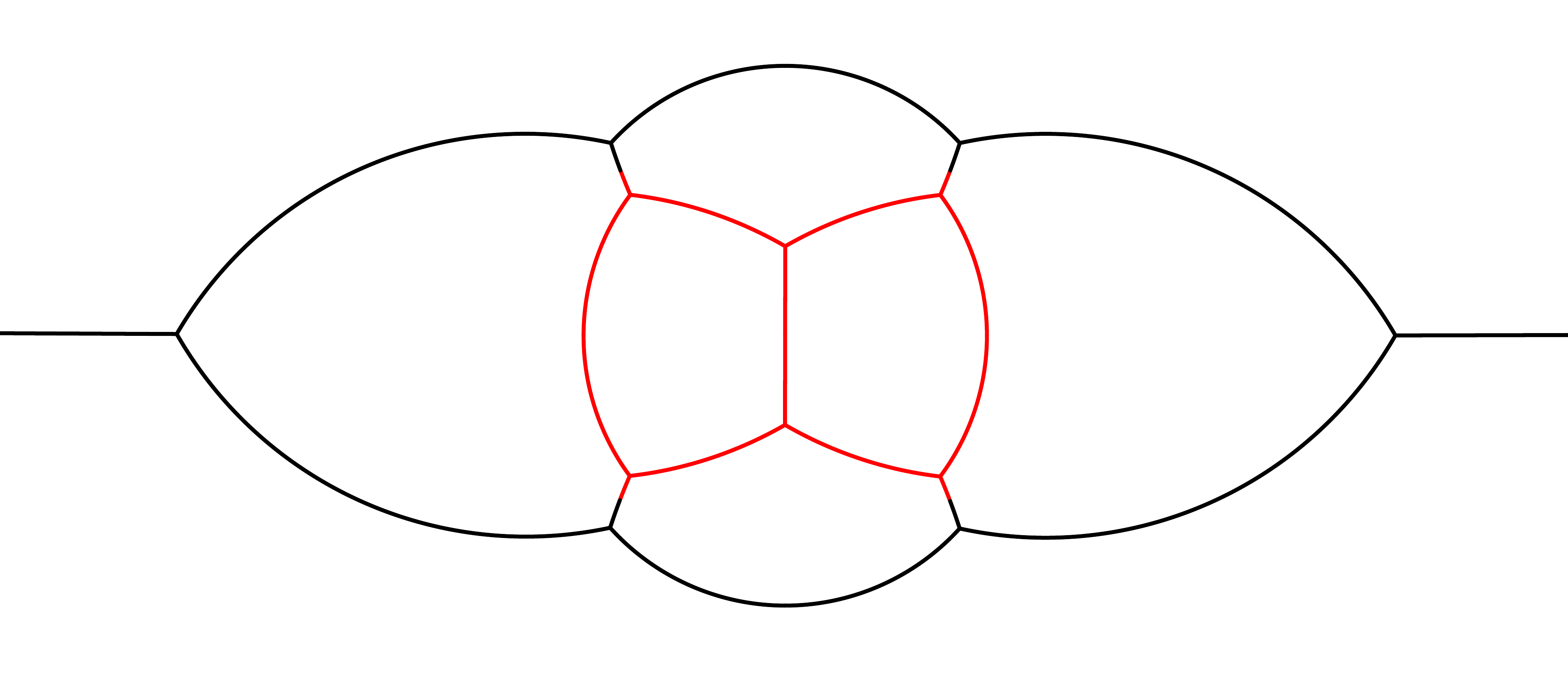
	\endgroup

			\caption{$\mathcal{N}_1$ and $\mathcal{N}_2$}
			\label{fig:8faces}
		\end{subfigure}
		\begin{subfigure}[b]{.37\linewidth}
			\centering %
	\begingroup
	\fontsize{7pt}{12pt}
	\def\svgwidth{0.72\columnwidth}
	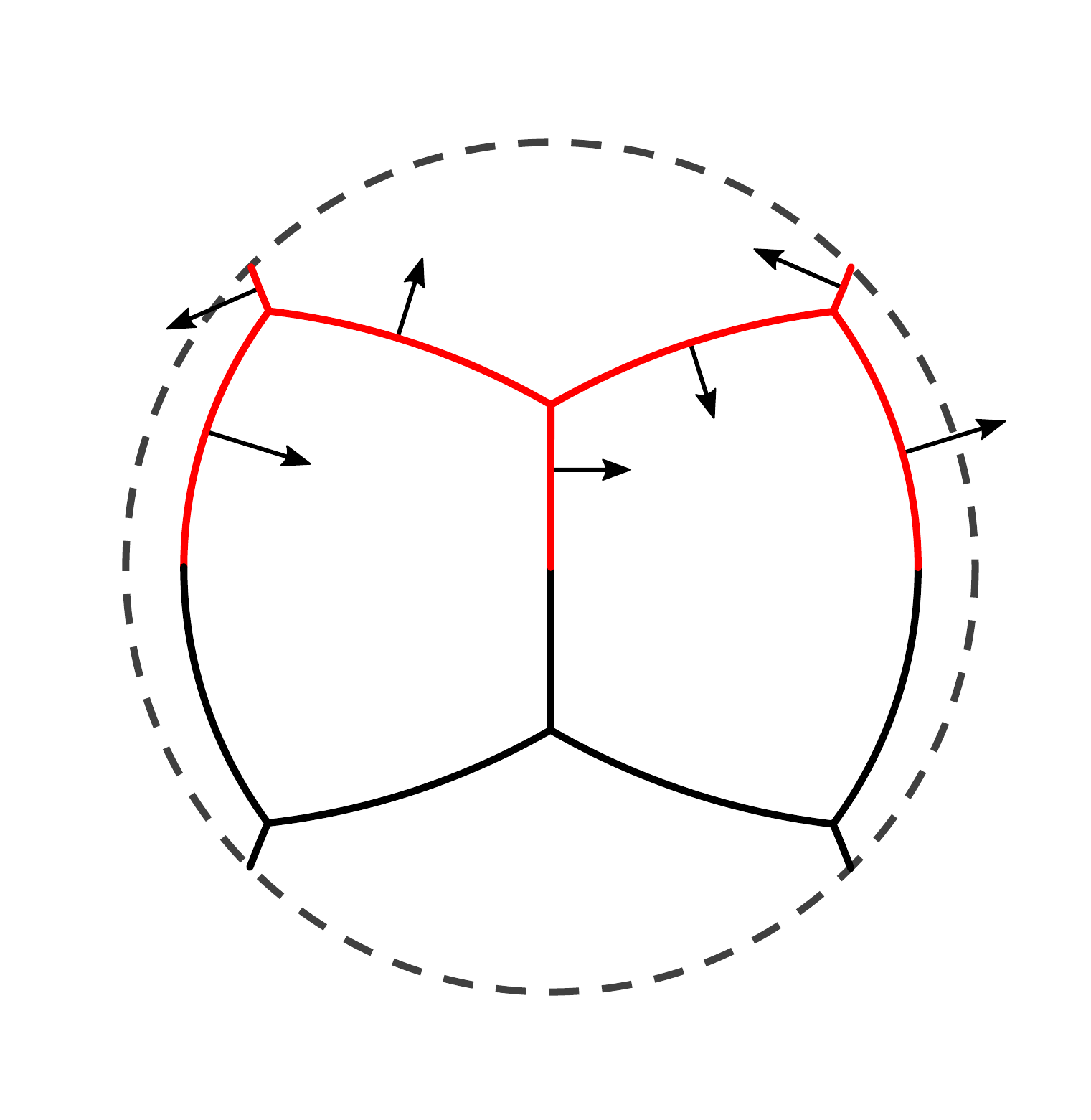
	\endgroup

			\caption{Divide $\mathcal{N}_1$ to $\mathcal{N}_{11}$ and $\mathcal{N}_{12}$}
			\label{fig:8_prism_cut}
		\end{subfigure}
		\caption{Partition of 4-4 type network}
		\label{fig:partition_8_faces}
	\end{figure}

	We will show each $\mathcal{N}_i$ will have index 3 and nullity 0.
	Again, we divide $\mathcal{N}_1$ into two subnetworks $\mathcal{N}_{11}$ and $\mathcal{N}_{12}$ as shown in Figure \ref{fig:partition_8_faces}.

	To compute the index and nullity of $\mathcal{N}_{11}$, we divide it into three subnetworks $\mathcal{N}_{111}$, $\mathcal{N}_{112}$ and $\mathcal{N}_{113}$ along point $P'$ as shown in Figure \ref{fig:8_prism_cut}.
	We assume $\mathcal{N}_{113}$ is the network associated with arc $P'P_3$.
	At first, we know $\mathcal{N}_{111},\mathcal{N}_{112},\mathcal{N}_{113}$ have index 0 and nullity 0 by Example \ref{ex:arc} and Corollary \ref{cor_3nets}.
	We directly compute the matrix form of the Dirichlet-to-Neumann map of $\mathcal{N}_{111}$ by \eqref{eq:matrixM3arc} with order of the boundary points as $\left\{ Q_1,P_1,P' \right\}$ and denote it $M$.
	We know the Dirichlet-to-Neumann map of $\mathcal{N}_{111}$ at $P'$ is given by
	\[
		a:=M_{33}=\frac{1-\tan l_3(\tan l_1+\tan l_2)}{\tan l_1+\tan l_2+\tan l_3}.
	\]
	Here,
	\begin{align*}
		l_1={} & \text{ the length of arc }Q_1\mathring P,\\
		l_2={} & \text{ the length of arc }P_1\mathring P,\\
		l_3={}& \text{ the length of arc }P'\mathring P.
	\end{align*}

	The Dirichlet-to-Neumann map of $\mathcal{N}_{113}$ at $P'$ is given by
	\[
		b:=\cot l_2,
	\]
	by Example \ref{ex:arc}.
	Now, in view of the computation in Example \ref{ex:3nets}, we can check $\sigma_2=2ab+a^2$ defined in \eqref{eq:sigma} is negative.
	Here, we have used the approximate values of $l_1,l_2,l_3$ to find the sign of $\sigma_2$ (cf., \cite[p. 502]{taylor1976structure}).
	This will imply the index of $\mathcal{N}_{11}$ is 1 and nullity is 0.

	%A short calculation can show $\mathcal{N}_{11}$ (and $\mathcal{N}_{12}$) have index 1 and nullity 0 if we further divide $\mathcal{N}_{11}$ into three small networks at $P'$. Now we need to find the index and nullity of $\overline{T}$ on set $\overline{\mathcal{P}}_1= \{ P_1,P_2,P_3 \}$.

	Now, let us consider the Dirichlet-to-Neumann map $\overline{T}$ with respect to the partition $\mathcal{N}_1=\mathcal{N}_{11}\cup \mathcal{N}_{12}$.
	By the symmetric properties of $\mathcal{N}_1$, we know $\overline{T}$ will equal to the Dirichlet-to-Neumann map $T$ of $\mathcal{N}_{11}$ restricted on $\left\{ P_1,P_3,P_2 \right\}$.
	To compute the matrix form of $T$, we can ignore the actions at point $Q_1$ and $Q_2$ and view $\mathcal{N}_{11i}$ as a network with only two boundary points $P_i,P'$.
	Then, we know the matrix form $\overline{M}$ of $\overline{T}$ can be given by \eqref{eq:matrixM} with $
	\begin{bmatrix}
		M_{33} & M_{32} \\
		M_{23} & M_{22}
	\end{bmatrix}$ in place of $M^1$ and $M^2$,
	$
	\begin{bmatrix}
		\cot l_2 & -\csc l_2\\
		-\csc l_2 & \cot l_2
	\end{bmatrix}$ in place of $M^3$.
	Then, we need to verify $\overline{M}$ has two positive eigenvalues and one negative eigenvalue.
	This can be done by following method.
	One can check the vector $(1,-1,0)$ is indeed an eigenvector of $\overline{M}$ with positive eigenvalue.
	Then, we can compute the determinate of $\overline{M}$ and find it is a negative number.
	Thus, $\overline{M}$ has two positive eigenvalues and one negative eigenvalue.

	\begin{remark}
		It is not hard to see the vector $(1,-1,0)$ is an eigenvector by the symmetric property of network $\mathcal{N}_{11}$ but it might not easy to see the corresponding eigenvalue is positive.
		To do that, we need to calculate the numerical value of $\overline{M}$ to determine how many positive eigenvalues it has.
		This has been done by a Python code I wrote.
		See \ref{app1:D-N} for details.
	\end{remark}

	%Restricting to $\mathcal{N}_{11}$, we want to show the D-N map $T$ of $\mathcal{N}_{11}$ has a positive eigenfunction. Indeed, for $g=(1,0,-1) \in V (\overline{\mathcal{P}}_1)$, we still let $u_g$ be the $L$-extension of $g$. By the symmetric properties of $\mathcal{N}_{11}$, we know $u_g$ will equal to 0 identically on arc $P'P_2$ and $T(g)=\lambda g$ for some $\lambda \in \mathbb{R}$.
	%This implies
	%\begin{equation}
	%	\frac{\partial u_g}{\partial \tau}=0
	%	\label{eq:pfDvalueP}
	%\end{equation}
	%at $P'$ by the compatibility condition.

	%If we write the value of $u_g$ at $P'$ on network $\mathcal{N}_{111}$ as $c$, from \eqref{eq:pfDvalueP}, we have
	%\[
	%	\begin{bmatrix}
	%		M_{22} & M_{23} \\
	%		M_{32} & M_{33}
	%	\end{bmatrix}
	%	\begin{bmatrix}
	%		1 \\ c
	%	\end{bmatrix}=
	%	\begin{bmatrix}
	%		\lambda \\ 0
	%	\end{bmatrix}.
	%\]
	%Solving $c$, and then $\lambda=M_{22}-\frac{M_{23}^2}{M_{33}}$.
	%We can check $\lambda$ is positive by the approximate values of $l_1,l_2,l_3$.
	
	%So we can remove this arc and view $\mathcal{N}_{11}$ as a network with 5 geodesic arcs in the calculation of $T_1(g)$ to get $T_1(g)=\lambda g$ for some positive $\lambda$. Then by the symmetry of $\mathcal{N}_{11}$ and $\mathcal{N}_{12}$, the $\overline{T}$ has at least one positive eigenvalue (actually the same with $\lambda$).

	Now, by our Theorem \ref{thm_index_theorem_for_network} and Theorem \ref{thm_nullity_theorem_for_network}, we know $\mathcal{N}$ has index 3 and nullity 0.
	%Hence the sum of index and nullity of $\overline{T}_1$ is at most 2.
	%This will imply the sum of index and nullity of $\mathcal{N}_1$ is at most 4. 
	%Recall that by Lemma \ref{lem_locallyConst}, we know $\mathcal{N}_1$ has index at least 2 since each face of $\mathcal{N}_1$ will contribute at least one to the index and those functions are linearly independent.

	%But note that $\mathcal{N}_1 $ has index at least 2 since it contains two quadrilaterals. 
	%Besides, $\mathcal{N}_1$ has nullity at least 2 since for any rotation that fixes two antipodal points on the great circle shown in Figure \ref{fig:8_prism_cut}, it will generate functions taking 0 on the boundary of $\mathcal{N}_1$.
	%There are two linearly independent such functions on $\mathcal{N}_1$, which implies $\mathcal{N}_1$ has nullity at least 2.

	%Hence $\mathcal{N}_1$ has index 2 and nullity 2.

	Hence, we know the Dirichlet-to-Neumann map $\overline{T}$ with respect to the partition $\mathcal{N}=\mathcal{N}_1 \cup  \mathcal{N}_2$ defined on the 4-dimensional space.

	So 
	\[
		\mathrm{Ind}(\mathcal{N})+ \mathrm{Nul}(\mathcal{N})\le 4+2\times 3=10=F+2.
	\]

	2. For $\mathcal{N}$ to be 6-3 type, we will divide $\mathcal{N}$ into three isometric parts $\mathcal{N}_1, \mathcal{N}_2, \mathcal{N}_3$ as shown in Figure \ref{fig:partition_9_faces}.

	\begin{figure}[ht]
		\centering 
		\begin{subfigure}[b]{.35\linewidth}
 	\centering
	\begingroup
	\fontsize{7pt}{12pt}
	\def\svgwidth{1\columnwidth}
	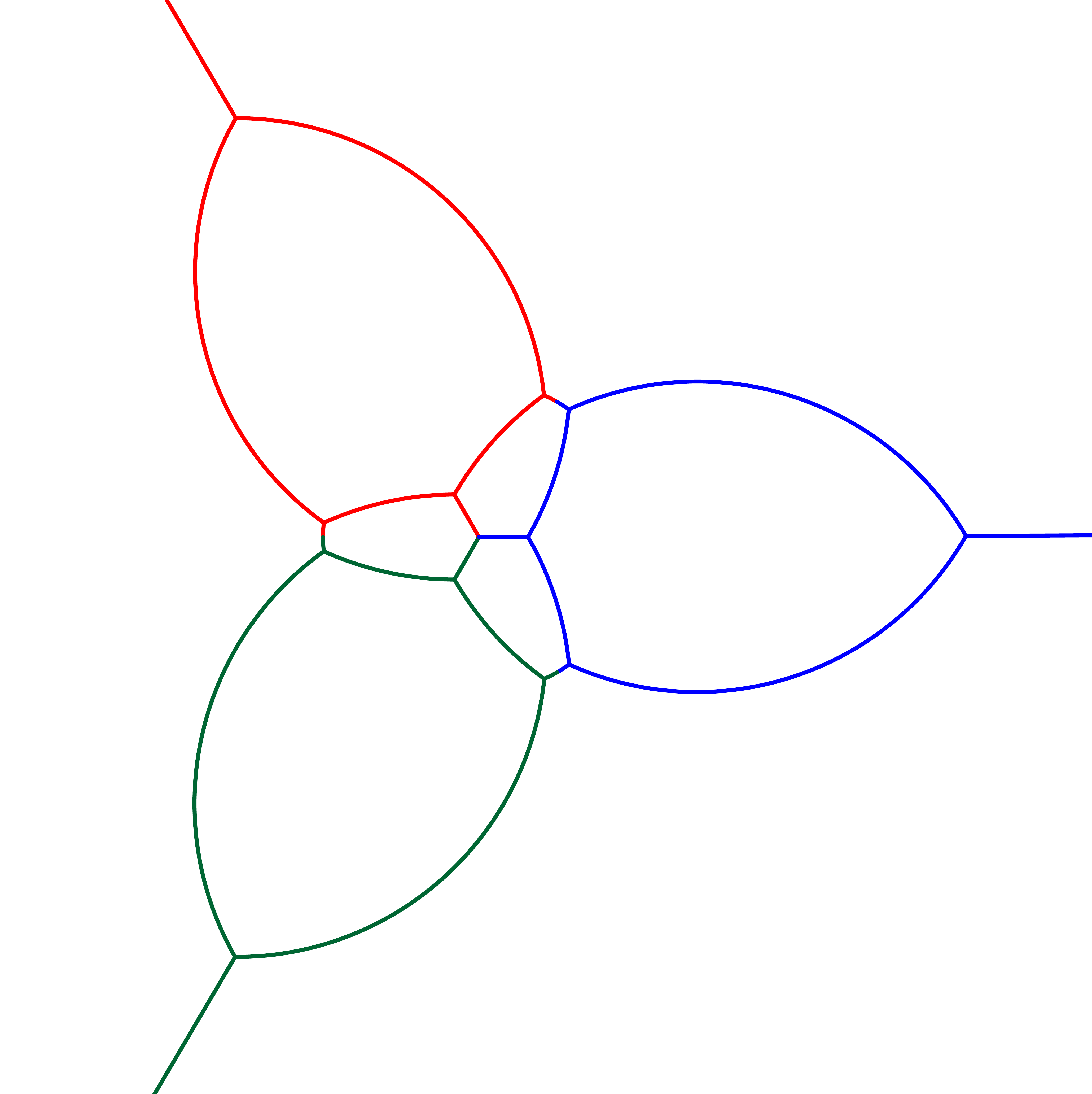
	\endgroup

 	\caption{$\mathcal{N}_1$, $\mathcal{N}_2$, and $\mathcal{N}_3$}
 	\label{fig:9faces}
 \end{subfigure}
\begin{subfigure}[b]{.58\linewidth}
 	\centering
	\begingroup
	\fontsize{7pt}{12pt}
	\def\svgwidth{1\columnwidth}
	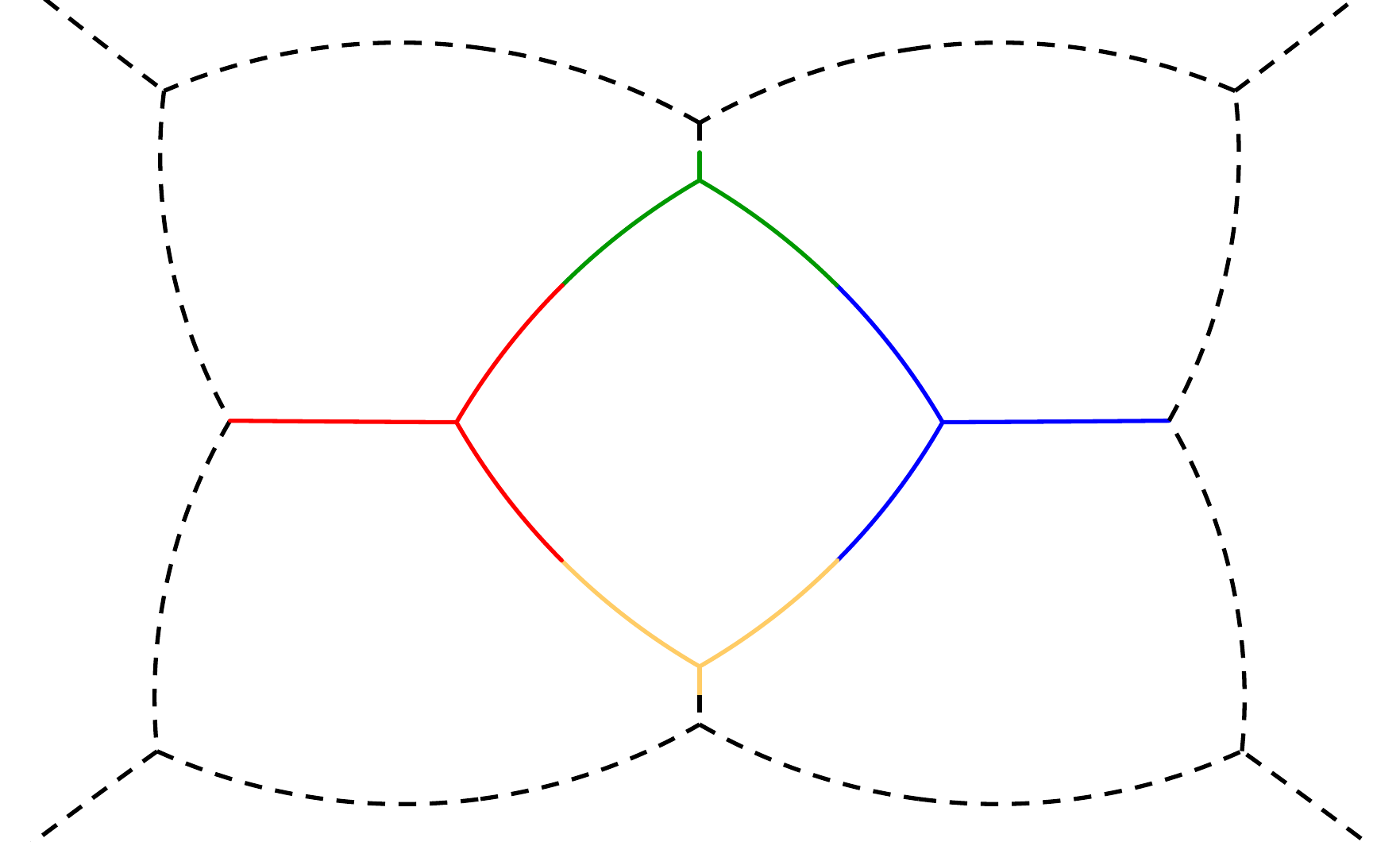
	\endgroup

 	\caption{Divide $\mathcal{N}_1$ into four parts}
 	\label{fig:9faces_sub}
 \end{subfigure}
		\caption{Partition of 6-3 type network}
		\label{fig:partition_9_faces}
	\end{figure}

	Again, we can show that each $\mathcal{N}_i $ has index 1 and nullity 1 by dividing it into four parts. We will put the detailed calculation in Example \ref{ex:6-3subnet}.

	Moreover, the null space of $\mathcal{N}_i$ is still generated by the rotation that fixed $0$ and $\infty$ as shown in Figure \ref{fig:9faces}.
	We can see the Dirichlet-to-Neumann map on $\mathcal{N}_i$ is well-defined if and only if for $g$ defined on the boundary of $\mathcal{N}_i$, $g$ takes opposite values on $0$ and $\infty$ with suitable orientation.
	So we can find the Dirichlet-to-Neumann map $\overline{T}$ for $\mathcal{N}= \mathcal{N}_1 \cup  \mathcal{N}_2 \cup  \mathcal{N}_3$ defined on a space of dimension at most 5. Hence
	\[
		\mathrm{Ind}(\mathcal{N})+ \mathrm{Nul}(\mathcal{N})\le 5+3\times 2=11=F+2.
	\]

	\noindent\textbf{Fifth class.} For the last two networks, we do not have a good partition to reduce the burden of calculation.
	The possible way is to calculate the Dirichlet-to-Neumann map by writing an algorithm to compute the Dirichlet-to-Neumann map directly.
	Indeed, we have written a Python code to help us determine the index and nullity of a given network. See \ref{app1:D-N} for the link to the code.
	%The detailed method has been illustrated in \ref{app1:D-N}.
	Here, we list the key points of the calculation.

	For the case of regular spherical dodecahedron, we divide it into six isometric parts as shown in Figure \ref{fig:12-faces2}.
	\begin{figure}[ht]
		\centering 
		\begin{subfigure}[b]{.35\linewidth}
 	\centering
	\begingroup
	\fontsize{7pt}{12pt}
	\def\svgwidth{1\columnwidth}
	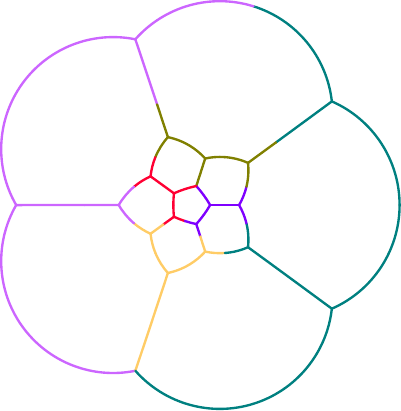
	\endgroup

 	\caption{Partition of the dodecahedron type network}
 	\label{fig:12-faces2}
 \end{subfigure}
 \begin{subfigure}[b]{0.36\linewidth}
		\centering 
	\begingroup
	\fontsize{7pt}{12pt}
	\def\svgwidth{1\columnwidth}
	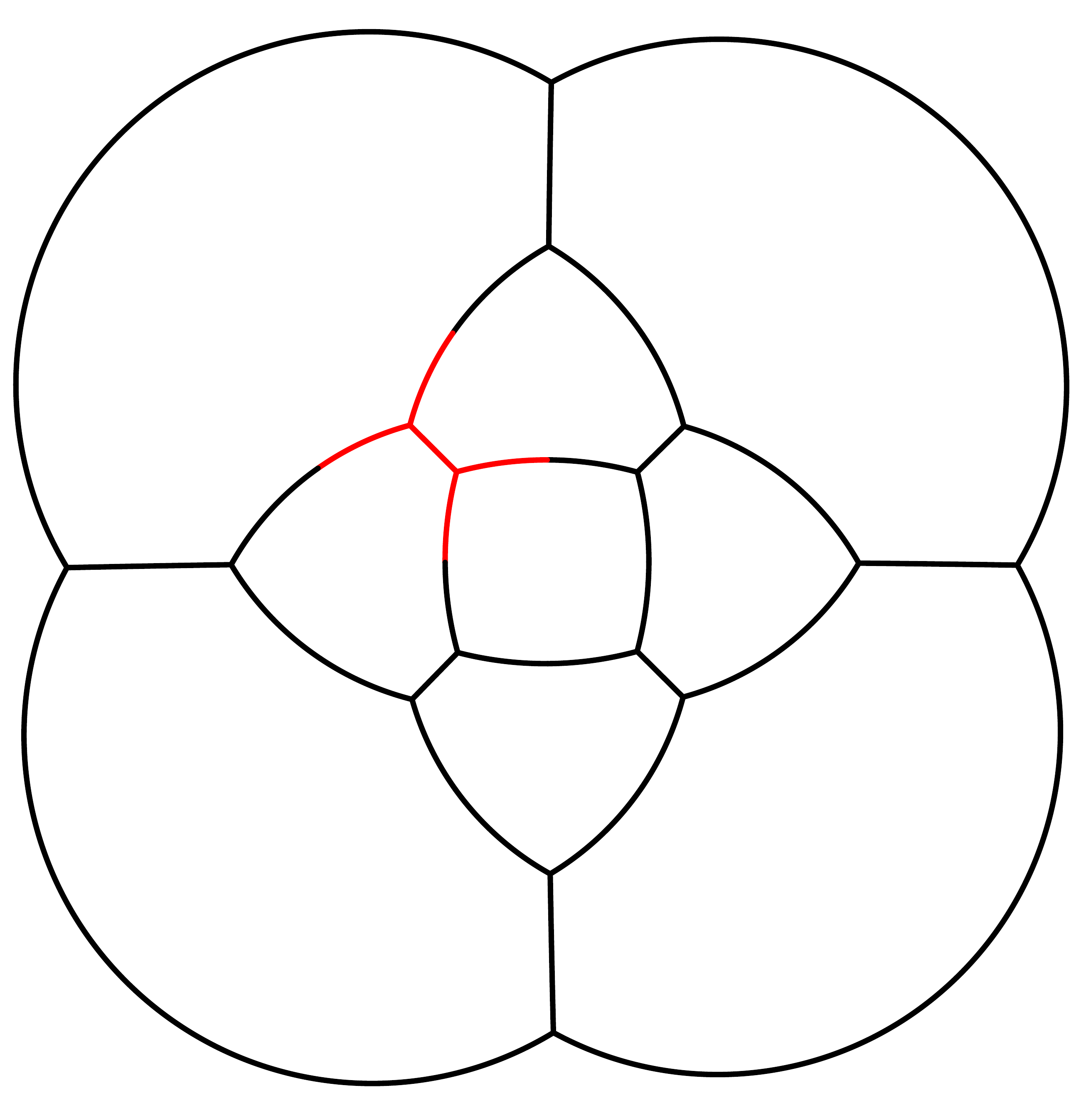
	\endgroup

		\caption{Partition of the 8-2 type network}
		\label{fig:partition_10_faces}
	\end{subfigure}
		\caption{Partition of the dodecahedron type network and the 8-2 type network}
		\label{fig:sub_doda}
	\end{figure}

	First, we can compute the index and nullity of each subnetwork using Proposition \ref{prop:matrix_form_of_DN_map}.
	It will tell us these subnetworks have index zero and nullity zero.
	Then, we can compute the Dirichlet-to-Neumann map on boundary of each subnetwork and the Dirichlet-to-Neumann map $\overline{T}$ with respect to the partition shown in Figure \ref{fig:12-faces2} using Proposition \ref{prop:matrix_form_of_DN_map} again.
	Numerical results show that $\overline{T}$ has at least two positive eigenvalues.
	Note that $\overline{T}$ is defined on a space of dimension 16.
	Hence $\mathrm{Ind}(\mathcal{N})+\mathrm{Nul}(\mathcal{N})\le 16-2=14=F+2$.

	%Finally, we can compute the D-N map on the whole network using Proposition \ref{prop:DN_map_on_whole_network}.

	%So we will calculate the D-N map directly.

	The last network $\mathcal{N}$ is the one skeleton of a spherical polyhedron with eight equal pentagons and two regular quadrilaterals.
	We call it an 8-2 type.

	%For $\mathcal{N}$ to be the network, the calculation is very long, and it seems there is no obvious partition like before to reduce the burden of detailed calculation.
	We divide $\mathcal{N}$ into 8 isometric parts as shown in Figure \ref{fig:partition_10_faces} (only one subnetwork is marked).
	We can show that each subnetwork has index 0 and nullity 0 by the direct computation based on Proposition \ref{prop:matrix_form_of_DN_map}.
	
	%The idea is that we can calculate the D-N map for this subnetwork first, and it will have the form
	%\[
	%	\begin{bmatrix}
	%		a & b & c & c\\
	%		b & a & c & c \\
	%		c & c & d & d \\
	%		c & c & d & d
	%	\end{bmatrix}.
	%\]
	%under suitable orientation. So the whole D-N map for this partition can be represented by a 16 by 16 symmetric matrix with its elements given by the linear combination of $a,b,c,d$. We can find at least four linearly independent vectors with positive eigenvalues after a long calculation.

	Then, we use Proposition \ref{prop:matrix_form_of_DN_map} to the Dirichlet-to-Neumann map $\overline{T}$ with respect to this partition and find at least four positive eigenvalues.

	%Of course, we'll need to know the approximate values of $a,b,c,d$ in the calculation.
	
	This will show 
	\[
		\mathrm{Ind}(\mathcal{N})
		+\mathrm{Nul}(\mathcal{N})\le 12=F+2.
	\]
	
	Hence, we finish the proof for all cases.
\end{proof}

\begin{proof}
	[Proof of Theorem \ref{thm_index_and_nullity_of_triple_junction_networks}]
	From the Proposition \ref{prop:lower_bound} and Proposition \ref{prop:upper_bound}, we know
	\[
		\mathrm{Ind}(\mathcal{N})=F-1,\quad 
		\mathrm{Nul}(\mathcal{N})=3.
	\]
	
	Moreover, by the proof of Proposition \ref{prop:lower_bound}, we know $\mathrm{dim}(\mathcal{J}_{-1}(L))=F-1$ and $\mathrm{dim}(\mathcal{J}_0(L))=3$. Hence we should have $\mathcal{J}_0^{-}(\mathcal{N})=\mathcal{J}_{-1}(L)$ and $\mathcal{J}_0^0(\mathcal{N})=\mathcal{J}_0(L)$. So the conclusion of this theorem is directly followed by the results in Proposition \ref{prop:lower_bound}.
\end{proof}

\subsection*{Acknowledgements}
	I would like to thank my advisor Prof. Martin Li for his helpful discussions and encouragement. I would also like to thank Prof. Frank Morgan for his comments on several typos and for providing several references.
	
	The author is partially supported by a research grant from the Research Grants Council of the Hong Kong Special Administrative Region, China [Project No.: CUHK 14301319] and CUHK Direct Grant [Project Code: 4053338].

%%The author is partially supported by a research grant from the Research Grants Council of the Hong Kong Special Administrative Region, China [Project No.: CUHK 14301319] and CUHK Direct Grant [Project Code: 4053338]. 

\appendix

\section{Computation of Dirichlet-to-Neumann maps}
\label{app1:D-N}

\renewcommand{\thesection}{\Alph{section}}

We will give a direct method of computing the index, nullity, and Dirichlet-to-Neumann maps of a given network.
Recall that there are indeed two types of Dirichlet-to-Neumann maps.
The first one is the Dirichlet-to-Neumann map defined on the boundary of a given network $\mathcal{N}$ (defined in Definition \ref{def_DNmap}), and another one is the Dirichlet-to-Neumann map with respect to a partition of the network (defined in Definition \ref{def_DNpartition}).
Before that, we fix some notation used in this section.

Given a network $\mathcal{N}\subset \mathbb{S}^2$ with $\mathcal{Q}=\left\{ q_1,\cdots ,q_k \right\}$ where $k=\# \mathcal{Q}$, if $\mathrm{Nul}(\mathcal{N})=0$, then we know the Dirichlet-to-Neumann map $T$ of $\mathcal{N}$ is defined on space $V(\mathcal{Q})$.
We also fix an orientation on $\mathcal{N}$ such that the unit normal vector field $\nu$ satisfies $\nu(q_i)=R \tau(q_i)$ where $R$ is a counterclockwise rotation on $\mathbb{S}^2$.

\begin{remark}
	Sometimes, it seems impossible to choose such an orientation.
	For example, $\mathcal{N}$ could contain only one curve.
	But we can always add several interior points if needed to make it possible. Note that the Dirichlet-to-Neumann map only depends on the image of $\mathcal{N}$ in view of the proof of Proposition \ref{prop_congruent}.
\end{remark}

For the network $\mathcal{N}$ described above, we write $M^{\mathcal{N}}=(M_{ij}^{\mathcal{N}})_{i,j=1}^k$ as the matrix form of the Dirichlet-to-Neumann map $T$ defined on $V(\mathcal{Q})$.
In other words, for any fixed $1\le i,j\le k$, if $u \in C_1^\infty(\mathcal{N})$ with $Lu=0$ on $\mathcal{N}$, $u(q_l)=0$ for any $l\neq i$, $u(q_i)=1$, and $u$ satisfies the compatibility condition,
then $M_{ij}^{\mathcal{N}}=\frac{\partial u}{\partial \tau_{q_j}}(q_j)$.
It is easy to see that $M^{\mathcal{N}}$ is symmetric.
Note that $M^{\mathcal{N}}$ depends on the order of the points in set $\mathcal{Q}$.

\begin{example}
	\label{ex:arc}
	If the image of $\mathcal{N}$ is a simple curve of length $l$ with $\mathcal{Q}=\left\{ q_1,q_2 \right\}$, $l \neq k \pi$ for $k \in \mathbb{Z}$, then the matrix form $M^{\mathcal{N}}$ of the Dirichlet-to-Neumann map of $\mathcal{N}$ is given by
	\begin{equation}
		M^{\mathcal{N}}=
		\begin{bmatrix}
			\cot l & -\csc l\\
			-\csc l& \cot l \\
		\end{bmatrix}.
		\label{eq:fundBlock}
	\end{equation}

	To see this, let us choose $s$ to be the arc length parameterization of curve $\gamma$.
	We choose $q_1=\gamma(0)$, $q_2=\gamma(l)$.
	To find the value $M_{22}^{\mathcal{N}}$, we need to consider the function $u$ on $\gamma$ such that
	\[
		\frac{d^{2}}{ds^{2}}u+u=0,\quad u(0)=0,\quad  u(l)=1.
	\]
	Hence, we find $u(s)=\frac{\sin s}{\sin l}$.
	Then, $M_{22}^{\mathcal{N}}=\frac{d}{ds}u(l)=\cot l$.
	For the value of $M_{21}^{\mathcal{N}}$, note that $\tau(\gamma(0))=-\frac{\partial }{\partial s}$, then we find
	\[
		M_{21}^{\mathcal{N}}=-\frac{d}{ds}u(0)=-\csc l.
	\]

	Similarly for the values of $M_{11}^{\mathcal{N}}$ and $M_{12}^{\mathcal{N}}$.
	But we should keep in mind the choice of orientation.
	This means we should choose another arc length parameterization of $\gamma$ such that $\gamma(0)=q_2$, $\gamma(l)=q_1$ to find the correct values of $M_{11}^{\mathcal{N}}$ and $M_{12}^{\mathcal{N}}$.
\end{example}

Later on, we will find that we can compute any Dirichlet-to-Neumann map of a given network $\mathcal{N}$ only based on fundamental matrix \eqref{eq:fundBlock} together with some operations described below.

Suppose $\mathcal{N}$ is the network with partition $\left\{ \mathcal{N}_i \right\}_{i=1}^l$ we have defined in Section \ref{sec:index_theorem_for_networks}.
We will give an algorithm to compute the Dirichlet-to-Neumann map of $\mathcal{N}$ using the Dirichlet-to-Neumann map of each $\mathcal{N}_i$.
We need to assume $\mathrm{Nul}(\mathcal{N}_i)=0$ for each $i$ so that the Dirichlet-to-Neumann map $T_i$ of $\mathcal{N}_i$ is well-defined on $V(\mathcal{Q}_i)$.

Recall that, we have written $\mathcal{Q}_i=\tilde{\mathcal{Q}}_i\cup \overline{\mathcal{Q}}_i$.
We arrange the points in $\mathcal{Q}_i$ such that
\[
	\mathcal{Q}_i=\{\tilde{q}_1,\cdots ,\tilde{q}_{\tilde{k}_i},\overline{q}_1,\cdots ,\overline{q}_{\overline{k}_i}\},
\]
where $\tilde{q}_i \in \tilde{\mathcal{Q}}_i, \overline{q}_i \in \overline{\mathcal{Q}}_i$, $\tilde{k}_i=\# \mathcal{\tilde{Q}}_i$, $\overline{k}_i=\# \overline{\mathcal{Q}}_i$.
For simplicity, we write $\tilde{k}=\tilde{k}_1+\cdots +\tilde{k}_l$, $\overline{k}=\overline{k}_1+\cdots +\overline{k}_l$.

Then, we can write the matrix form of Dirichlet-to-Neumann map $T_i$ as
\[
	M^i=
	\begin{bmatrix}
		A_{11}^i & A_{12}^i\\
		A_{21}^i& A_{22}^i \\
	\end{bmatrix},
\]
where $A_{11}^i$ is a $\tilde{k}_i\times \tilde{k}_i$ matrix, $A_{22}^i$ is a $\overline{k}_i\times \overline{k}_i$ matrix.
We define a block matrix $P$ by
\[
	P=
	\begin{bmatrix}
		A_{11}^1 & 0 & \cdots & 0 & A_{12}^1 & 0 & \cdots & 0\\
		0 & A_{11}^2 & \cdots & 0 & 0 & A_{12}^2 & \cdots & 0\\
		\vdots & \vdots & \ddots & \vdots & \vdots & \vdots & \ddots & \vdots\\
		0 & 0 & \cdots & A_{11}^l & 0 & 0 & \cdots & A_{12}^l\\
		A_{21}^1 & 0 & \cdots & 0 & A_{22}^1 & 0 & \cdots & 0\\
		0 & A_{21}^2 & \cdots & 0 & 0 & A_{22}^2 & \cdots & 0\\
		\vdots & \vdots & \ddots & \vdots & \vdots & \vdots & \ddots & \vdots\\
		0 & 0 & \cdots & A_{21}^l & 0 & 0 & \cdots & A_{22}^l
	\end{bmatrix}=:
		\begin{bmatrix}
			P_{11} & P_{12}\\
			P_{21} & P_{22}\\
		\end{bmatrix},
\]
where $P_{11}$ is a $\tilde{k}\times \tilde{k}$ matrix, $P_{22}$ is a $\overline{k}\times \overline{k}$ matrix.
We can view $P$ as the matrix form of Dirichlet-to-Neumann map of $\cup _{i=1}^l\mathcal{N}_i$ without identifying the points in $\overline{\mathcal{P}}$ and it acts on space $V(\cup _{i=1}^l \mathcal{Q}_i)$ where the order of the points in $V(\cup _{i=1}^l \mathcal{Q}_i)$ is given by
\begin{equation}
	V(\cup _{i=1}^l \mathcal{Q}_i)=\left\{ \tilde{q}_1^1,\cdots ,\tilde{q}_{\tilde{k}_1}^1,\cdots ,\tilde{q}_1^l,\cdots ,\tilde{q}_{\tilde{k}_l}^l,\overline{q}_1^1,\cdots ,\overline{q}_{\overline{k}_1}^1,\cdots ,\overline{q}_1^l,\cdots ,\overline{q}_{\overline{k}_l}^l \right\}.
	\label{eq:orderOfptsAllQ}
\end{equation}

%Similarly, $\mathring P$ can act on space $V(\cup _{i=1}^l \tilde{\mathcal{Q}}_i)$.

%Let $Q$ be an matrix such that rows of $Q$ form the orthonormal basis of $V_2(\overline{\mathcal{P}})$ where we view $V_2(\overline{\mathcal{P}})$ as a subspace of $V(\cup _{i=1}^l \mathcal{Q}_i)$.
We write $\tilde{Q}$ such that rows of $\tilde{Q}$ form the orthonormal basis of $V_1(\cup _{i=1}^l \tilde{\mathcal{Q}}_i)$ where we view $V_1(\cup _{i=1}^l \tilde{\mathcal{Q}}_i)$ as a subspace of $V(\cup _{i=1}^l \mathcal{\tilde{Q}}_i)$, and $\tilde{Q}^\bot$ such that rows of $\tilde{Q}^\bot$ form the orthonormal basis of the orthogonal complement of $\tilde{Q}$ in $V(\cup _{i=1}^l \mathcal{\tilde{Q}}_i)$.
Note that $\tilde{Q}$ is a $\mathrm{dim}(V_1(\overline{\mathcal{P}}))$ by $\tilde{k}$ matrix. 

%row vectors form the orthonormal basis of $V_2(\cup _{i=1}^l \tilde{\mathcal{Q}}_i)$ as a subspace of $V(\cup _{i=1}^l \tilde{\mathcal{Q}}_i)$.
%Indeed, we can write $Q=
%\begin{bmatrix}
%	\tilde{Q} &
%\end{bmatrix}$.
%
Later on, we will use an example to illustrate those matrices.

We also arrange the points in $\mathcal{Q}$ such that
\[
	\mathcal{Q}=\left\{ \overline{q}_1^1,\cdots ,\overline{q}_{\overline{k}_1}^1,\cdots ,\overline{q}_1^l,\cdots ,\overline{q}_{\overline{k}_l}^l \right\}.
\]

\begin{proposition}
	\label{prop:matrix_form_of_DN_map}
	The matrix form $\overline{M}$ of the Dirichlet-to-Neumann map $\overline{T}$ with respect to the network $\mathcal{N}$ with its partition $\left\{ \mathcal{N}_i \right\}_{i=1}^l$ is given by
	\[
		\overline{M}=
		\tilde{Q} P_{11} \tilde{Q}^T.
	\]
	If we also assume $\mathrm{Nul}(\mathcal{N})=0$, then the matrix form $M$ of the Dirichlet-to-Neumann map $T$ with respect to the network $\mathcal{N}$ on the boundary $\mathcal{Q}$ is given by	
	\begin{equation}
		M=P_{22}-
		P_{21}\begin{bmatrix}
			\tilde{Q} P_{11} \\ \tilde{Q}^\bot
		\end{bmatrix}^{-1}
		\begin{bmatrix}
			\tilde{Q}P_{12} \\ 0 \\
		\end{bmatrix}.
		\label{eq:propMmatrix}
	\end{equation}

	Here, the number $0$ in the right hand side of \eqref{eq:propMmatrix} is a $\tilde{k}-\mathrm{dim}(V_1(\overline{\mathcal{P}}))$ by $\overline{k}$ zero matrix.
\end{proposition}
\begin{proof}
	We first prove the first statement.
	Recall that $P_{11}$ can be viewed as the matrix form of the Dirichlet-to-Neumann map $\hat{T}$ described in \eqref{eq:That}.
	On the other hand, we can view $\tilde{Q}$ as a projection from $V(\tilde{Q})$ to $V_1(\tilde{Q})$.
	Based on the definition of $\overline{T}$ in \eqref{eq:Tbar}, we know
	\[
		\overline{M}=\tilde{Q} P_{11} \tilde{Q}^T.
	\]

	Next, we prove the second statement.
	Since $\mathrm{Nul}(\mathcal{N})=0$, we know the matrix $\overline{M}$ is invertible by Theorem \ref{thm_nullity_theorem_for_network}.

	Given any $f\in V(\mathcal{Q})$, we know there exists $u \in C_1^\infty(\mathcal{N})$ which solves the following problem,
	\[
		\begin{cases}
			Lu=0 & \text{ on } \mathcal{N},\\
			u=f & \text{ on } \mathcal{Q},\\
			u \in V_1(\mathcal{P}),\\
			\frac{\partial u}{\partial \tau} \in V_2(\mathcal{P}).
		\end{cases}
	\]
	Then $Mf$ is given by the value of $\frac{\partial u}{\partial \tau}$ restricted on $\mathcal{Q}$.

	Let $\tilde{f}$ be the value of $u$ on $\mathcal{\mathcal{P}}$.
	Then, $Mf$ can be written as
	\begin{equation}
		Mf=P_{21}\tilde{f}+P_{22}f.
		\label{eq:pfMf}
	\end{equation}
	
	Note that we know
	\begin{equation}
		\tilde{Q}^\bot \tilde{f}=0,
		\label{eq:pfftilbot}
	\end{equation}
	since $u \in V_1(\mathcal{P})$.

	On the other hand, we write the image of $(\tilde{f},f)$ under $P$ as $(\tilde{g},g)$.
	That is, we have
	\begin{equation}
		\tilde{g}=P_{11}\tilde{f}+P_{12}f, \quad g=P_{21}\tilde{f}+P_{22}f.
		\label{eq:pfgtil}
	\end{equation}

	Then, from $\frac{\partial u}{\partial \tau} \in V_2(\mathcal{P})$, we have
	\[
		\tilde{Q} \tilde{g}=0.
	\]

	Using \eqref{eq:pfgtil}, we have
	\[
		\tilde{Q}P_{11}\tilde{f}=-\tilde{Q}P_{12}f.
	\]

	Together with \eqref{eq:pfftilbot}, we can solve $\tilde{f}$ and find
	\[
		\tilde{f}=
		\begin{bmatrix}
			\tilde{Q} P_{11} \\ \tilde{Q}^\bot
		\end{bmatrix}^{-1}
		\begin{bmatrix}
			-\tilde{Q}P_{12}\\ 0
		\end{bmatrix}f.
		%=
		%-\begin{bmatrix}
			%\tilde{Q} P_{11} \\ \tilde{Q}^\bot
		%\end{bmatrix}^{-1}
		%\begin{bmatrix}
			%I_{\tilde{k}} \\ 0_{\overline{k}}
		%\end{bmatrix}\tilde{Q}P_{12}f.
	\]

	Putting this into \eqref{eq:pfMf}, we have
	\[
		Mf=P_{21}\tilde{f}+P_{22}f=
	-	
		P_{21}\begin{bmatrix}
			\tilde{Q} P_{11} \\ \tilde{Q}^\bot
		\end{bmatrix}^{-1}
		\begin{bmatrix}
			\tilde{Q}P_{12} \\ 0
		\end{bmatrix}f+P_{22}f.
	\]

	Therefore, we know $M$ is given by \eqref{eq:propMmatrix}.

	One thing left, we need to show the matrix $
	\begin{bmatrix}
		\tilde{Q} P_{11} \\ \tilde{Q}^\bot
	\end{bmatrix}$ is invertible.
	Since $\tilde{Q}P_{11}\tilde{Q}^\bot $ is invertible, we know
	\[
		\begin{bmatrix}
			\tilde{Q} P_{11} \tilde{Q}^T  & * \\
			0& I_{\overline{k}} \\
		\end{bmatrix}
	\]
	is invertible.
	Here, $I_{\overline{k}}$ is a $\overline{k}$ by $\overline{k}$ identity matrix.
	Note that
	\[
		\begin{bmatrix}
			\tilde{Q} P_{11} \\ \tilde{Q}^\bot
		\end{bmatrix}
		\begin{bmatrix}
			\tilde{Q}^T & (\tilde{Q}^\bot )^T
		\end{bmatrix}=
		\begin{bmatrix}
			\tilde{Q} P_{11} \tilde{Q}^T  & * \\
			0& I_{\overline{k}}
		\end{bmatrix}
	\]
	and $
	\begin{bmatrix}
		\tilde{Q}^T & (\tilde{Q}^\bot )^T
	\end{bmatrix}$ is invertible, we know $
	\begin{bmatrix}
		\tilde{Q} P_{11} \\ \tilde{Q}^\bot
	\end{bmatrix}$ is invertible.
\end{proof}

\begin{example}
	\label{ex:3nets}
	Suppose $\mathcal{N}_i$ is a network with end point set $\mathcal{Q}_i=\left\{ q^i_1,q^i_2 \right\}$ and assume the Dirichlet-to-Neumann map is given by
	\[
		M^i=
		\begin{bmatrix}
			M_{11}^i & M_{12}^i \\
			M_{21}^i & M_{22}^i \\
		\end{bmatrix},
	\]
	for $i=1,2,3$.
	
	%Then we know
	%\[
	%	M^i=
	%	\begin{bmatrix}
	%		\cot l_i & -\csc l_i\\
	%		-\csc l_i& \cot l_i \\
	%	\end{bmatrix}.
	%\]
	Now, let us identify the point $q_1^1,q_1^2,q_1^3$ to get a network $\mathcal{N}$.
	Then we know
	\begin{align*}
		P_{ij}={}&
		\begin{bmatrix}
			M_{ij}^1 & 0 & 0 \\
			0 & M_{ij}^2 & 0 \\
			0 & 0 & M_{ij}^3 \\
		\end{bmatrix},\\
		%P_{11}=P_{22}={}&
		%\begin{bmatrix}
		%	\cot l_1 & 0 & 0 \\
		%	0 & \cot l_2 & 0 \\
		%	0 & 0 & \cot l_3 \\
		%\end{bmatrix}, \\
		%P_{12}=P_{21}^T={}&
		%\begin{bmatrix}
		%	-\csc l_1 & 0 & 0 \\
		%	0 & -\csc l_2 & 0 \\
		%	0 & 0 & -\csc l_3 \\
		%\end{bmatrix},\\
		\tilde{Q}={}& 
		\begin{bmatrix}
			\frac{1}{\sqrt{2}} & -\frac{1}{\sqrt{2}} & 0 \\
			-\frac{1}{\sqrt{6}} & -\frac{1}{\sqrt{6}} & \frac{2}{\sqrt{6}} \\
		\end{bmatrix}.
	\end{align*}
	Hence,
	\[
		\overline{M}=
		\begin{bmatrix}
			\frac{1}{2}(M_{11}^1+M_{11}^2) & -\frac{1}{2\sqrt{3}}(M_{11}^2-M_{11}^1)\\
			-\frac{1}{2\sqrt{3}}(M_{11}^2-M_{11}^1)& \frac{1}{6}(M_{11}^1+M_{11}^2+4M_{11}^3)\\
		\end{bmatrix}.
	\]
	In particular, the eigenvalues of $\overline{T}$ is exactly the root of the polynomial $x^2-\frac{2}{3}(M_{11}^1+M_{11}^2+M_{11}^3)x+\frac{1}{3}(M_{11}^1M_{11}^2+M_{11}^2M_{11}^3+M_{11}^3M_{11}^1)$.

	%\[
	%	\overline{M}=
	%	\begin{bmatrix}
	%		\frac{1}{2}(\cot l_1+\cot l_2) & -\frac{1}{2\sqrt{3}}(\cot l_1-\cot l_2)\\
	%		-\frac{1}{2\sqrt{3}}(\cot l_1-\cot l_2)& \frac{1}{3}(\cot l_1+\cot l_2+4\cot l_3) \\
	%	\end{bmatrix}
	%\]
	%and its eigenvalues are exactly the roots of the polynomial
	%\[
	%	x^2-\frac{2}{3}(\cot l_1+\cot l_2+\cot l_3)x+\frac{1}{3}(\cot l_1\cot l_2+\cot l_2\cot l_3+\cot l_3\cot l_1)=0.
	%\]
	So we can directly determine the index and nullity of $\overline{M}$ by the sign of
	\begin{align}
		\sigma_1={} & M_{11}^1+M_{11}^2+M_{11}^3, \nonumber \\
		\sigma_2={} & M_{11}^1M_{11}^2+M_{11}^2M_{11}^3+M_{11}^3M_{11}^1.\label{eq:sigma}
	\end{align}
	
	%\begin{align}
	%	\sigma_1:={}&\cot l_1+\cot l_2+\cot l_3,\nonumber \\
	%	\sigma_2:={}&\cot l_1\cot l_2+\cot l_2\cot l_3+\cot l_3\cot l_1.
	%	\label{eq:sigma}
	%\end{align}

	Note that one special case is, $\mathcal{N}_i$ is a simple curve of length $l_i, l_i \neq k \pi$ for $k \in \mathbb{N}$.
	For this case, we know $M_{11}^i=M_{22}^i=\cot l_i$, $M_{12}^i=M_{21}^i=-\csc l_i$ by Example \ref{ex:arc}.
	One useful result is the following.
	\begin{corollary}
		\label{cor_3nets}
		If $l_i \in (0,\frac{\pi}{2})$ for $i=1,2,3$.
		Then the network $\mathcal{N}$ described in this example satisfies $\mathrm{Ind}(\mathcal{N})=0$ and $\mathrm{Nul}(\mathcal{N})=0$.
	\end{corollary}

	For the matrix form of $M$, we need to assume $\sigma_2\neq 0$. We know $\tilde{Q}^\bot = 
	\begin{bmatrix}
		\frac{1}{\sqrt{3}} & \frac{1}{\sqrt{3}} & \frac{1}{\sqrt{3}} \\
	\end{bmatrix}$.
	Put it into \eqref{eq:propMmatrix} and after simplification, we have
	\begin{align}
		M= P_{22}
			+\frac{1}{\sigma_2}\begin{bmatrix}
			-(M_{11}^2+M_{11}^3)(M_{12}^1)^2 & M_{11}^3M_{12}^1M_{12}^2 & M_{11}^2M_{12}^1M_{12}^3 \\
			M_{11}^3M_{12}^1M_{12}^2 & -(M_{11}^1+M_{11}^3)(M_{12}^2)^2 & M_{11}^1M_{12}^2M_{12}^3 \\
			M_{11}^2M_{12}^1M_{12}^3 & M_{11}^1M_{12}^2M_{12}^3 & -(M_{11}^1+M_{11}^2)(M_{12}^3)^2 \\
		\end{bmatrix}.\label{eq:matrixM}
	\end{align}
	Here, we have used $M_{12}^i=M_{21}^i$ for $i=1,2,3$ to simplify the expression.

	For our convenience, we also write down $M$ when $M_{11}^i=M_{22}^i=\cot l_i, M_{12}^i=M_{21}^i=-\csc l_i$ for $i=1,2,3$.
	\begin{align}
		&M= \frac{1}{\tan l_1+\tan l_2+\tan l_3}\left\{\vphantom{
			\begin{bmatrix}
			1 & 0 & \\
			0&	1 & 0 \\
			0&	0& 1 
\end{bmatrix}}I_3\right.\nonumber \\
			&
		-\left.\begin{bmatrix}
		\tan l_1(\tan l_2+\tan l_3)& -\frac{1}{\cos l_1
		\cos l_2} & -\frac{1}{\cos l_1 \cos l_3}\\
		-\frac{1}{\cos l_1 \cos l_2}& 
		\tan l_2(\tan l_1+\tan l_3)& -\frac{1}{\cos l_2 \cos l_3} \\
		-\frac{1}{\cos l_1 \cos l_2}& -\frac{1}{\cos l_2 \cos
		l_3} & \tan l_3(\tan l_1+\tan l_2)
	\end{bmatrix}
\label{eq:matrixM3arc}\right\}
			\end{align}
\end{example}

Now, we give another non-trivial example of how to compute the Dirichlet-to-Neumann map of a given network with its partition.
\begin{example}
	\label{ex:6-3subnet}
	We will consider the network $\mathcal{N}_{1}$ with its partition $\cup _{i=1}^4 \mathcal{N}_{1i}$ as shown in Figure \ref{fig:9faces_sub}.
At first, let us compute the matrix form $M^i$ of the Dirichlet-to-Neumann map $M$ of $\mathcal{N}_{1i}$ on its boundary.
Note that each arc in $\mathcal{N}_{11}$ has length $\arcsin(\frac{1}{\sqrt{3}})$.
By choosing $l_1=l_2=l_3=\arcsin(\frac{1}{\sqrt{3}})$ and using matrix form of $M$ in \eqref{eq:matrixM3arc}, we know $M^1$ is given by
\[
	\begin{bmatrix}
		0 & \frac{1}{\sqrt{2}} & \frac{1}{\sqrt{2}} \\
		\frac{1}{\sqrt{2}} & 0 & \frac{1}{\sqrt{2}} \\
		\frac{1}{\sqrt{2}} & \frac{1}{\sqrt{2}} & 0
	\end{bmatrix}.
\]
Similarly, we can compute the matrix form $M^2$ of the Dirichlet-to-Neumann map of $\mathcal{N}_{12}$ on its boundary.
Note that the short arc in $\mathcal{N}_{12}$ has length $\frac{\pi}{3}-\arccos(\frac{1}{\sqrt{3}})$.
By choosing $l_1=l_2=\arcsin(\frac{1}{\sqrt{3}})$, $l_3=\frac{\pi}{3}-\arccos(\frac{1}{\sqrt{3}})$, we find
\[
	M^2=
	\begin{bmatrix}
		\frac{1}{2\sqrt{3}} & \frac{1}{\sqrt{2}}+\frac{1}{2\sqrt{3}} & *\\
		\frac{1}{\sqrt{2}}+\frac{1}{2\sqrt{3}} & \frac{1}{2\sqrt{3}} & *\\
		* & * & *
	\end{bmatrix}.
\]
Here, the star $*$ means the values that we do not need to compute.

Now, we can write down the matrix $P_{11}$ in Proposition \ref{prop:matrix_form_of_DN_map} for $\mathcal{N}_{1}$ with its partition as
\[
	P_{11}=
	\begin{bmatrix}
		\tilde{M}^1 & 0 & 0 & 0\\
		0 & \tilde{M}^2 & 0 & 0\\
		0 & 0 & \tilde{M}^1 & 0\\
		0 & 0 & 0 & \tilde{M}^2
	\end{bmatrix}
\]
where
\[
	\tilde{M}^1=
	\begin{bmatrix}
		0 & \frac{1}{\sqrt{2}}\\
		\frac{1}{\sqrt{2}} & 0
	\end{bmatrix},\quad 
	\tilde{M}^2=
	\begin{bmatrix}
		\frac{1}{2\sqrt{3}} & \frac{1}{\sqrt{2}}+\frac{1}{2\sqrt{3}}\\
		\frac{1}{\sqrt{2}}+\frac{1}{2\sqrt{3}} & \frac{1}{2\sqrt{3}}
	\end{bmatrix}.
\]
Here, we ignore the points $Q_1,Q_2,Q_3,Q_4$ since we do not care about the values at $Q_i$.
By the definition of $\tilde{Q}$, we know we can choose $\tilde{Q}$ as
\[
	\tilde{Q}=
	\begin{bmatrix}
		0 & \frac{1}{\sqrt{2}} & -\frac{1}{\sqrt{2}} & 0 & 0 & 0 & 0 & 0\\
		0 & 0 & 0 & \frac{1}{\sqrt{2}} & -\frac{1}{\sqrt{2}} & 0 & 0 & 0\\
		0 & 0 & 0 & 0 & 0 & \frac{1}{\sqrt{2}} & -\frac{1}{\sqrt{2}} & 0\\
		-\frac{1}{\sqrt{2}} & 0 & 0 & 0 & 0 & 0 & 0 & \frac{1}{\sqrt{2}}
	\end{bmatrix}.
\]

Hence, the matrix form $\overline{M}$ of the Dirichlet-to-Neumann map of $\mathcal{N}_{1}$ with respect to its partition is given by
\[
	\overline{M}=\tilde{Q}P_{11}\tilde{Q}^T=
	\begin{bmatrix}
		\frac{1}{4\sqrt{3}} & -\frac{1}{2\sqrt{2}}-\frac{1}{4\sqrt{3}} & 0 & -\frac{1}{2\sqrt{2}}\\
		-\frac{1}{2\sqrt{2}}-\frac{1}{4\sqrt{3}} & \frac{1}{4\sqrt{3}} & -\frac{1}{2\sqrt{2}} & 0\\
		0 & -\frac{1}{2\sqrt{2}} & \frac{1}{4\sqrt{3}} & -\frac{1}{2\sqrt{2}}-\frac{1}{4\sqrt{3}}\\
		-\frac{1}{2\sqrt{2}} & 0 & -\frac{1}{2\sqrt{2}}-\frac{1}{4\sqrt{3}} & \frac{1}{4\sqrt{3}}
	\end{bmatrix}.
\]

We can directly check that $\overline{M}$ has two positive eigenvalues, one zero eigenvalue, and one negative eigenvalue.
So $\mathcal{N}_1$ has index 1 and nullity 1.
\end{example}

At last, we would like to mention that based on the Proposition \ref{prop:matrix_form_of_DN_map}, we can develop a Python code to compute the Dirichlet-to-Neumann map of a given network.
The code is available at \url{https://github.com/wgaom/D-Nmap}.
Indeed, we actually compute the index and nullity of every network and subnetwork numerically we have discussed in the proof of Proposition \ref{prop:upper_bound} and verify our proof.

\bibliographystyle{alpha}
\bibliography{references.bib}


%%%% \bibitem{label}
%%%% Text of bibliographic item
%%
%%\bibitem{}

%%\end{thebibliography}
\end{document}

%% file: fpage2.tex
\title{Index of Embedded Networks in the Sphere}
\author{Gaoming Wang}
\address{Yau Mathematical Sciences Center, Tsinghua University, Beijing, 100084, China, and Department of Mathematics, The Chinese University of Hong Kong, Shatin, N.T., Hong Kong.}
\email{gmwang@tsinghua.edu.cn, gmwang@math.cuhk.edu.hk}
\date{}
\maketitle
\begin{abstract}
	In this paper, we will compute the Morse index and nullity of some embedded stationary geodesic networks in the sphere. The key theorem in the computation is that the index (and nullity) for the whole network is related to the index (and nullity) of small networks and the Dirichlet-to-Neumann map defined in this paper. Finally, we will show that for all stationary triple junction networks in $\mathbb{S}^2$, there is only one eigenvalue (without multiplicity) $-1$, which is less than 0, and the corresponding eigenfunctions are locally constant. Besides, the multiplicity of eigenvalues 0 is 3 for these networks, and their eigenfunctions are generated by the rotations on the sphere.
\end{abstract}

%% file: figures/4faces.pdf_tex
%% Creator: Inkscape 1.0 (4035a4f, 2020-05-01), www.inkscape.org
%% Accompanies image file '4faces.pdf' (pdf, eps, ps)
\begingroup%
  \makeatletter%
  \providecommand\color[2][]{%
    \errmessage{(Inkscape) Color is used for the text in Inkscape, but the package 'color.sty' is not loaded}%
    \renewcommand\color[2][]{}%
  }%
  \providecommand\transparent[1]{%
    \errmessage{(Inkscape) Transparency is used (non-zero) for the text in Inkscape, but the package 'transparent.sty' is not loaded}%
    \renewcommand\transparent[1]{}%
  }%
  \providecommand\rotatebox[2]{#2}%
  \newcommand*\fsize{\dimexpr\f@size pt\relax}%
  \newcommand*\lineheight[1]{\fontsize{\fsize}{#1\fsize}\selectfont}%
  \ifx\svgwidth\undefined%
    \setlength{\unitlength}{939.04504395bp}%
    \ifx\svgscale\undefined%
      \relax%
    \else%
      \setlength{\unitlength}{\unitlength * \real{\svgscale}}%
    \fi%
  \else%
    \setlength{\unitlength}{\svgwidth}%
  \fi%
  \global\let\svgwidth\undefined%
  \global\let\svgscale\undefined%
  \makeatother%
  \begin{picture}(1,0.819777)%
    \lineheight{1}%
    \setlength\tabcolsep{0pt}%
    \put(0,0){\includegraphics[width=\unitlength,page=1]{4faces.pdf}}%
    \put(0.54023191,0.41633844){\makebox(0,0)[lt]{\lineheight{1.25}\smash{\begin{tabular}[t]{l}$A$\end{tabular}}}}%
    \put(0.59367863,0.4947481){\makebox(0,0)[lt]{\lineheight{1.25}\smash{\begin{tabular}[t]{l}$\nu^1$\end{tabular}}}}%
    \put(0.65437963,0.43124132){\makebox(0,0)[lt]{\lineheight{1.25}\smash{\begin{tabular}[t]{l}$\gamma^1$\end{tabular}}}}%
    \put(0.46280344,0.55827381){\makebox(0,0)[lt]{\lineheight{1.25}\smash{\begin{tabular}[t]{l}$\gamma^2$\end{tabular}}}}%
    \put(0.67647737,0.08121083){\makebox(0,0)[lt]{\lineheight{1.25}\smash{\begin{tabular}[t]{l}$\gamma^5$\end{tabular}}}}%
    \put(0.36782836,0.24765244){\makebox(0,0)[lt]{\lineheight{1.25}\smash{\begin{tabular}[t]{l}$\gamma^3$\end{tabular}}}}%
    \put(0.14589108,0.43988276){\makebox(0,0)[lt]{\lineheight{1.25}\smash{\begin{tabular}[t]{l}$\gamma^4$\end{tabular}}}}%
    \put(0.66419264,0.68007515){\makebox(0,0)[lt]{\lineheight{1.25}\smash{\begin{tabular}[t]{l}$\gamma^6$\end{tabular}}}}%
    \put(0.85866999,0.4103908){\makebox(0,0)[lt]{\lineheight{1.25}\smash{\begin{tabular}[t]{l}$B$\end{tabular}}}}%
    \put(0.33760418,0.09921842){\makebox(0,0)[lt]{\lineheight{1.25}\smash{\begin{tabular}[t]{l}$C$\end{tabular}}}}%
    \put(0.32709661,0.68596552){\makebox(0,0)[lt]{\lineheight{1.25}\smash{\begin{tabular}[t]{l}$D$\end{tabular}}}}%
    \put(0.38125344,0.37231029){\makebox(0,0)[lt]{\lineheight{1.25}\smash{\begin{tabular}[t]{l}$\nu^2$\end{tabular}}}}%
    \put(0.53871841,0.28368426){\makebox(0,0)[lt]{\lineheight{1.25}\smash{\begin{tabular}[t]{l}$\nu^3$\end{tabular}}}}%
    \put(0.85887885,0.03157998){\makebox(0,0)[lt]{\lineheight{1.25}\smash{\begin{tabular}[t]{l}$\nu^5$\end{tabular}}}}%
    \put(0,0.38276928){\makebox(0,0)[lt]{\lineheight{1.25}\smash{\begin{tabular}[t]{l}$\nu^4$\end{tabular}}}}%
    \put(0.85887909,0.759097){\makebox(0,0)[lt]{\lineheight{1.25}\smash{\begin{tabular}[t]{l}$\nu^6$\end{tabular}}}}%
  \end{picture}%
\endgroup%

%% file: figures/choiceofu.pdf_tex
%% Creator: Inkscape 1.2 (dc2aeda, 2022-05-15), www.inkscape.org
%% Accompanies image file 'choiceofu.pdf' (pdf, eps, ps)
\begingroup%
  \makeatletter%
  \providecommand\color[2][]{%
    \errmessage{(Inkscape) Color is used for the text in Inkscape, but the package 'color.sty' is not loaded}%
    \renewcommand\color[2][]{}%
  }%
  \providecommand\transparent[1]{%
    \errmessage{(Inkscape) Transparency is used (non-zero) for the text in Inkscape, but the package 'transparent.sty' is not loaded}%
    \renewcommand\transparent[1]{}%
  }%
  \providecommand\rotatebox[2]{#2}%
  \newcommand*\fsize{\dimexpr\f@size pt\relax}%
  \newcommand*\lineheight[1]{\fontsize{\fsize}{#1\fsize}\selectfont}%
  \ifx\svgwidth\undefined%
    \setlength{\unitlength}{680.31496063bp}%
    \ifx\svgscale\undefined%
      \relax%
    \else%
      \setlength{\unitlength}{\unitlength * \real{\svgscale}}%
    \fi%
  \else%
    \setlength{\unitlength}{\svgwidth}%
  \fi%
  \global\let\svgwidth\undefined%
  \global\let\svgscale\undefined%
  \makeatother%
  \begin{picture}(1,0.5)%
    \lineheight{1}%
    \setlength\tabcolsep{0pt}%
    \put(0,0){\includegraphics[width=\unitlength,page=1]{choiceofu.pdf}}%
    \put(0.23372814,0.40334816){\makebox(0,0)[lt]{\lineheight{1.25}\smash{\begin{tabular}[t]{l}$\gamma_1$\end{tabular}}}}%
    \put(0.47162031,0.34616197){\makebox(0,0)[lt]{\lineheight{1.25}\smash{\begin{tabular}[t]{l}$\gamma_2$\end{tabular}}}}%
    \put(0.78940258,0.33717874){\makebox(0,0)[lt]{\lineheight{1.25}\smash{\begin{tabular}[t]{l}$\gamma_3$\end{tabular}}}}%
    \put(0.35259234,0.23499432){\makebox(0,0)[lt]{\lineheight{1.25}\smash{\begin{tabular}[t]{l}$\gamma_4$\end{tabular}}}}%
    \put(0.7063076,0.22937978){\makebox(0,0)[lt]{\lineheight{1.25}\smash{\begin{tabular}[t]{l}$\gamma_5$\end{tabular}}}}%
    \put(0.26725156,0.08564784){\makebox(0,0)[lt]{\lineheight{1.25}\smash{\begin{tabular}[t]{l}$\gamma_6$\end{tabular}}}}%
    \put(0.49857009,0.09350815){\makebox(0,0)[lt]{\lineheight{1.25}\smash{\begin{tabular}[t]{l}$\gamma_7$\end{tabular}}}}%
    \put(0.78266519,0.10249138){\makebox(0,0)[lt]{\lineheight{1.25}\smash{\begin{tabular}[t]{l}$\gamma_8$\end{tabular}}}}%
    \put(0,0){\includegraphics[width=\unitlength,page=2]{choiceofu.pdf}}%
    \put(0.50755326,0.26868146){\makebox(0,0)[lt]{\lineheight{1.25}\smash{\begin{tabular}[t]{l}$F_j$\end{tabular}}}}%
  \end{picture}%
\endgroup%

%% file: figures/Triangle_Prism2.pdf_tex
%% Creator: Inkscape 1.0 (4035a4f, 2020-05-01), www.inkscape.org
%% Accompanies image file 'Triangle_Prism2.pdf' (pdf, eps, ps)
\begingroup%
  \makeatletter%
  \providecommand\color[2][]{%
    \errmessage{(Inkscape) Color is used for the text in Inkscape, but the package 'color.sty' is not loaded}%
    \renewcommand\color[2][]{}%
  }%
  \providecommand\transparent[1]{%
    \errmessage{(Inkscape) Transparency is used (non-zero) for the text in Inkscape, but the package 'transparent.sty' is not loaded}%
    \renewcommand\transparent[1]{}%
  }%
  \providecommand\rotatebox[2]{#2}%
  \newcommand*\fsize{\dimexpr\f@size pt\relax}%
  \newcommand*\lineheight[1]{\fontsize{\fsize}{#1\fsize}\selectfont}%
  \ifx\svgwidth\undefined%
    \setlength{\unitlength}{1179.65060257bp}%
    \ifx\svgscale\undefined%
      \relax%
    \else%
      \setlength{\unitlength}{\unitlength * \real{\svgscale}}%
    \fi%
  \else%
    \setlength{\unitlength}{\svgwidth}%
  \fi%
  \global\let\svgwidth\undefined%
  \global\let\svgscale\undefined%
  \makeatother%
  \begin{picture}(1,0.94463666)%
    \lineheight{1}%
    \setlength\tabcolsep{0pt}%
    \put(0,0){\includegraphics[width=\unitlength,page=1]{Triangle_Prism2.pdf}}%
  \end{picture}%
\endgroup%

%% file: figures/cube.pdf_tex
%% Creator: Inkscape 1.0 (4035a4f, 2020-05-01), www.inkscape.org
%% Accompanies image file 'cube.pdf' (pdf, eps, ps)
\begingroup%
  \makeatletter%
  \providecommand\color[2][]{%
    \errmessage{(Inkscape) Color is used for the text in Inkscape, but the package 'color.sty' is not loaded}%
    \renewcommand\color[2][]{}%
  }%
  \providecommand\transparent[1]{%
    \errmessage{(Inkscape) Transparency is used (non-zero) for the text in Inkscape, but the package 'transparent.sty' is not loaded}%
    \renewcommand\transparent[1]{}%
  }%
  \providecommand\rotatebox[2]{#2}%
  \newcommand*\fsize{\dimexpr\f@size pt\relax}%
  \newcommand*\lineheight[1]{\fontsize{\fsize}{#1\fsize}\selectfont}%
  \ifx\svgwidth\undefined%
    \setlength{\unitlength}{1458.42348835bp}%
    \ifx\svgscale\undefined%
      \relax%
    \else%
      \setlength{\unitlength}{\unitlength * \real{\svgscale}}%
    \fi%
  \else%
    \setlength{\unitlength}{\svgwidth}%
  \fi%
  \global\let\svgwidth\undefined%
  \global\let\svgscale\undefined%
  \makeatother%
  \begin{picture}(1,0.88478578)%
    \lineheight{1}%
    \setlength\tabcolsep{0pt}%
    \put(0,0){\includegraphics[width=\unitlength,page=1]{cube.pdf}}%
  \end{picture}%
\endgroup%

%% file: figures/3_prism.pdf_tex
%% Creator: Inkscape 1.2 (dc2aeda, 2022-05-15), www.inkscape.org
%% Accompanies image file '3_prism.pdf' (pdf, eps, ps)
\begingroup%
  \makeatletter%
  \providecommand\color[2][]{%
    \errmessage{(Inkscape) Color is used for the text in Inkscape, but the package 'color.sty' is not loaded}%
    \renewcommand\color[2][]{}%
  }%
  \providecommand\transparent[1]{%
    \errmessage{(Inkscape) Transparency is used (non-zero) for the text in Inkscape, but the package 'transparent.sty' is not loaded}%
    \renewcommand\transparent[1]{}%
  }%
  \providecommand\rotatebox[2]{#2}%
  \newcommand*\fsize{\dimexpr\f@size pt\relax}%
  \newcommand*\lineheight[1]{\fontsize{\fsize}{#1\fsize}\selectfont}%
  \ifx\svgwidth\undefined%
    \setlength{\unitlength}{1025.2026175bp}%
    \ifx\svgscale\undefined%
      \relax%
    \else%
      \setlength{\unitlength}{\unitlength * \real{\svgscale}}%
    \fi%
  \else%
    \setlength{\unitlength}{\svgwidth}%
  \fi%
  \global\let\svgwidth\undefined%
  \global\let\svgscale\undefined%
  \makeatother%
  \begin{picture}(1,1.03906253)%
    \lineheight{1}%
    \setlength\tabcolsep{0pt}%
    \put(0,0){\includegraphics[width=\unitlength,page=1]{3_prism.pdf}}%
  \end{picture}%
\endgroup%

%% file: figures/3_prism_cut.pdf_tex
%% Creator: Inkscape 1.0 (4035a4f, 2020-05-01), www.inkscape.org
%% Accompanies image file '3_prism_cut.pdf' (pdf, eps, ps)
\begingroup%
  \makeatletter%
  \providecommand\color[2][]{%
    \errmessage{(Inkscape) Color is used for the text in Inkscape, but the package 'color.sty' is not loaded}%
    \renewcommand\color[2][]{}%
  }%
  \providecommand\transparent[1]{%
    \errmessage{(Inkscape) Transparency is used (non-zero) for the text in Inkscape, but the package 'transparent.sty' is not loaded}%
    \renewcommand\transparent[1]{}%
  }%
  \providecommand\rotatebox[2]{#2}%
  \newcommand*\fsize{\dimexpr\f@size pt\relax}%
  \newcommand*\lineheight[1]{\fontsize{\fsize}{#1\fsize}\selectfont}%
  \ifx\svgwidth\undefined%
    \setlength{\unitlength}{1247.24409449bp}%
    \ifx\svgscale\undefined%
      \relax%
    \else%
      \setlength{\unitlength}{\unitlength * \real{\svgscale}}%
    \fi%
  \else%
    \setlength{\unitlength}{\svgwidth}%
  \fi%
  \global\let\svgwidth\undefined%
  \global\let\svgscale\undefined%
  \makeatother%
  \begin{picture}(1,0.5)%
    \lineheight{1}%
    \setlength\tabcolsep{0pt}%
    \put(0,0){\includegraphics[width=\unitlength,page=1]{3_prism_cut.pdf}}%
    \put(0.05798364,0.45856072){\makebox(0,0)[lt]{\lineheight{1.25}\smash{\begin{tabular}[t]{l}$Q_3$\end{tabular}}}}%
    \put(0.39874294,0.20045437){\makebox(0,0)[lt]{\lineheight{1.25}\smash{\begin{tabular}[t]{l}$Q_1$\end{tabular}}}}%
    \put(0.06717346,0.01358871){\makebox(0,0)[lt]{\lineheight{1.25}\smash{\begin{tabular}[t]{l}$Q_2$\end{tabular}}}}%
    \put(0.2505038,0.36167067){\makebox(0,0)[lt]{\lineheight{1.25}\smash{\begin{tabular}[t]{l}$P_1$\end{tabular}}}}%
    \put(0.24754551,0.11169584){\makebox(0,0)[lt]{\lineheight{1.25}\smash{\begin{tabular}[t]{l}$P_2$\end{tabular}}}}%
    \put(0.07448601,0.23742282){\makebox(0,0)[lt]{\lineheight{1.25}\smash{\begin{tabular}[t]{l}$P_3$\end{tabular}}}}%
    \put(0.78743188,0.36906636){\color[rgb]{0,0,0}\makebox(0,0)[lt]{\lineheight{1.25}\smash{\begin{tabular}[t]{l}$P_1'$\end{tabular}}}}%
    \put(0.80813991,0.11484751){\color[rgb]{0,0,0}\makebox(0,0)[lt]{\lineheight{1.25}\smash{\begin{tabular}[t]{l}$P_2'$\end{tabular}}}}%
    \put(0.57591484,0.24186025){\color[rgb]{0,0,0}\makebox(0,0)[lt]{\lineheight{1.25}\smash{\begin{tabular}[t]{l}$P_3'$\end{tabular}}}}%
    \put(0.48568722,0.24186025){\color[rgb]{0,0,0}\makebox(0,0)[lt]{\lineheight{1.25}\smash{\begin{tabular}[t]{l}$\simeq$\end{tabular}}}}%
    \put(0.22792993,0.24481854){\makebox(0,0)[lt]{\lineheight{1.25}\smash{\begin{tabular}[t]{l}$\mathcal{N}_{11}$\end{tabular}}}}%
    \put(0.8035559,0.2574029){\makebox(0,0)[lt]{\lineheight{1.25}\smash{\begin{tabular}[t]{l}$\mathcal{N}'_{11}$\end{tabular}}}}%
    \put(0,0){\includegraphics[width=\unitlength,page=2]{3_prism_cut.pdf}}%
  \end{picture}%
\endgroup%

%% file: figures/8faces.pdf_tex
%% Creator: Inkscape 1.0 (4035a4f, 2020-05-01), www.inkscape.org
%% Accompanies image file '8faces.pdf' (pdf, eps, ps)
\begingroup%
  \makeatletter%
  \providecommand\color[2][]{%
    \errmessage{(Inkscape) Color is used for the text in Inkscape, but the package 'color.sty' is not loaded}%
    \renewcommand\color[2][]{}%
  }%
  \providecommand\transparent[1]{%
    \errmessage{(Inkscape) Transparency is used (non-zero) for the text in Inkscape, but the package 'transparent.sty' is not loaded}%
    \renewcommand\transparent[1]{}%
  }%
  \providecommand\rotatebox[2]{#2}%
  \newcommand*\fsize{\dimexpr\f@size pt\relax}%
  \newcommand*\lineheight[1]{\fontsize{\fsize}{#1\fsize}\selectfont}%
  \ifx\svgwidth\undefined%
    \setlength{\unitlength}{1906.65951226bp}%
    \ifx\svgscale\undefined%
      \relax%
    \else%
      \setlength{\unitlength}{\unitlength * \real{\svgscale}}%
    \fi%
  \else%
    \setlength{\unitlength}{\svgwidth}%
  \fi%
  \global\let\svgwidth\undefined%
  \global\let\svgscale\undefined%
  \makeatother%
  \begin{picture}(1,0.44310174)%
    \lineheight{1}%
    \setlength\tabcolsep{0pt}%
    \put(0,0){\includegraphics[width=\unitlength,page=1]{8faces.pdf}}%
  \end{picture}%
\endgroup%

%% file: figures/8faces_cut.pdf_tex
%% Creator: Inkscape 1.2 (dc2aeda, 2022-05-15), www.inkscape.org
%% Accompanies image file '8faces_cut.pdf' (pdf, eps, ps)
\begingroup%
  \makeatletter%
  \providecommand\color[2][]{%
    \errmessage{(Inkscape) Color is used for the text in Inkscape, but the package 'color.sty' is not loaded}%
    \renewcommand\color[2][]{}%
  }%
  \providecommand\transparent[1]{%
    \errmessage{(Inkscape) Transparency is used (non-zero) for the text in Inkscape, but the package 'transparent.sty' is not loaded}%
    \renewcommand\transparent[1]{}%
  }%
  \providecommand\rotatebox[2]{#2}%
  \newcommand*\fsize{\dimexpr\f@size pt\relax}%
  \newcommand*\lineheight[1]{\fontsize{\fsize}{#1\fsize}\selectfont}%
  \ifx\svgwidth\undefined%
    \setlength{\unitlength}{729.63811319bp}%
    \ifx\svgscale\undefined%
      \relax%
    \else%
      \setlength{\unitlength}{\unitlength * \real{\svgscale}}%
    \fi%
  \else%
    \setlength{\unitlength}{\svgwidth}%
  \fi%
  \global\let\svgwidth\undefined%
  \global\let\svgscale\undefined%
  \makeatother%
  \begin{picture}(1,1.02631577)%
    \lineheight{1}%
    \setlength\tabcolsep{0pt}%
    \put(0.2005487,0.47870771){\color[rgb]{0,0,0}\makebox(0,0)[lt]{\lineheight{1.25}\smash{\begin{tabular}[t]{l}$P_1$\end{tabular}}}}%
    \put(0.53564647,0.47870771){\color[rgb]{0,0,0}\makebox(0,0)[lt]{\lineheight{1.25}\smash{\begin{tabular}[t]{l}$P_3$\end{tabular}}}}%
    \put(0.72070237,0.47870771){\color[rgb]{0,0,0}\makebox(0,0)[lt]{\lineheight{1.25}\smash{\begin{tabular}[t]{l}$P_2$\end{tabular}}}}%
    \put(0.1512091,0.8146059){\color[rgb]{0,0,0}\makebox(0,0)[lt]{\lineheight{1.25}\smash{\begin{tabular}[t]{l}$Q_1$\end{tabular}}}}%
    \put(0.78851154,0.79404776){\color[rgb]{0,0,0}\makebox(0,0)[lt]{\lineheight{1.25}\smash{\begin{tabular}[t]{l}$Q_2$\end{tabular}}}}%
    \put(0.47808359,0.68714546){\color[rgb]{0,0,0}\makebox(0,0)[lt]{\lineheight{1.25}\smash{\begin{tabular}[t]{l}$P'$\end{tabular}}}}%
    \put(0.45958124,0.03134095){\color[rgb]{0,0,0}\makebox(0,0)[lt]{\lineheight{1.25}\smash{\begin{tabular}[t]{l}Great circle\end{tabular}}}}%
    \put(0,0){\includegraphics[width=\unitlength,page=1]{8faces_cut.pdf}}%
    \put(0.45958124,0.79404776){\color[rgb]{0,0,0}\makebox(0,0)[lt]{\lineheight{1.25}\smash{\begin{tabular}[t]{l}$\mathcal{N}_{11}$\end{tabular}}}}%
    \put(0.23979062,0.65336503){\color[rgb]{0,0,0}\makebox(0,0)[lt]{\lineheight{1.25}\smash{\begin{tabular}[t]{l}$\mathring P$\end{tabular}}}}%
  \end{picture}%
\endgroup%

%% file: figures/9faces.pdf_tex
%% Creator: Inkscape 1.0 (4035a4f, 2020-05-01), www.inkscape.org
%% Accompanies image file '9faces.pdf' (pdf, eps, ps)
\begingroup%
  \makeatletter%
  \providecommand\color[2][]{%
    \errmessage{(Inkscape) Color is used for the text in Inkscape, but the package 'color.sty' is not loaded}%
    \renewcommand\color[2][]{}%
  }%
  \providecommand\transparent[1]{%
    \errmessage{(Inkscape) Transparency is used (non-zero) for the text in Inkscape, but the package 'transparent.sty' is not loaded}%
    \renewcommand\transparent[1]{}%
  }%
  \providecommand\rotatebox[2]{#2}%
  \newcommand*\fsize{\dimexpr\f@size pt\relax}%
  \newcommand*\lineheight[1]{\fontsize{\fsize}{#1\fsize}\selectfont}%
  \ifx\svgwidth\undefined%
    \setlength{\unitlength}{1976.94692537bp}%
    \ifx\svgscale\undefined%
      \relax%
    \else%
      \setlength{\unitlength}{\unitlength * \real{\svgscale}}%
    \fi%
  \else%
    \setlength{\unitlength}{\svgwidth}%
  \fi%
  \global\let\svgwidth\undefined%
  \global\let\svgscale\undefined%
  \makeatother%
  \begin{picture}(1,1.00146413)%
    \lineheight{1}%
    \setlength\tabcolsep{0pt}%
    \put(0,0){\includegraphics[width=\unitlength,page=1]{9faces.pdf}}%
  \end{picture}%
\endgroup%

%% file: figures/9faces_sub.pdf_tex
%% Creator: Inkscape 1.2 (dc2aeda, 2022-05-15), www.inkscape.org
%% Accompanies image file '9faces_sub.pdf' (pdf, eps, ps)
\begingroup%
  \makeatletter%
  \providecommand\color[2][]{%
    \errmessage{(Inkscape) Color is used for the text in Inkscape, but the package 'color.sty' is not loaded}%
    \renewcommand\color[2][]{}%
  }%
  \providecommand\transparent[1]{%
    \errmessage{(Inkscape) Transparency is used (non-zero) for the text in Inkscape, but the package 'transparent.sty' is not loaded}%
    \renewcommand\transparent[1]{}%
  }%
  \providecommand\rotatebox[2]{#2}%
  \newcommand*\fsize{\dimexpr\f@size pt\relax}%
  \newcommand*\lineheight[1]{\fontsize{\fsize}{#1\fsize}\selectfont}%
  \ifx\svgwidth\undefined%
    \setlength{\unitlength}{845.96516166bp}%
    \ifx\svgscale\undefined%
      \relax%
    \else%
      \setlength{\unitlength}{\unitlength * \real{\svgscale}}%
    \fi%
  \else%
    \setlength{\unitlength}{\svgwidth}%
  \fi%
  \global\let\svgwidth\undefined%
  \global\let\svgscale\undefined%
  \makeatother%
  \begin{picture}(1,0.60144058)%
    \lineheight{1}%
    \setlength\tabcolsep{0pt}%
    \put(0,0){\includegraphics[width=\unitlength,page=1]{9faces_sub.pdf}}%
    \put(0.21368822,0.31455466){\color[rgb]{0,0,0}\makebox(0,0)[lt]{\lineheight{1.25}\smash{\begin{tabular}[t]{l}$\mathcal{N}_{11}$\end{tabular}}}}%
    \put(0.56046149,0.44487496){\color[rgb]{0,0,0}\makebox(0,0)[lt]{\lineheight{1.25}\smash{\begin{tabular}[t]{l}$\mathcal{N}_{12}$\end{tabular}}}}%
    \put(0.70347841,0.26107855){\color[rgb]{0,0,0}\makebox(0,0)[lt]{\lineheight{1.25}\smash{\begin{tabular}[t]{l}$\mathcal{N}_{13}$\end{tabular}}}}%
    \put(0.53978041,0.12076588){\color[rgb]{0,0,0}\makebox(0,0)[lt]{\lineheight{1.25}\smash{\begin{tabular}[t]{l}$\mathcal{N}_{14}$\end{tabular}}}}%
    \put(0.16452497,0.26116042){\color[rgb]{0,0,0}\makebox(0,0)[lt]{\lineheight{1.25}\smash{\begin{tabular}[t]{l}$Q_1$\end{tabular}}}}%
    \put(0.45156579,0.10241411){\color[rgb]{0,0,0}\makebox(0,0)[lt]{\lineheight{1.25}\smash{\begin{tabular}[t]{l}$Q_4$\end{tabular}}}}%
    \put(0.80385153,0.31600237){\color[rgb]{0,0,0}\makebox(0,0)[lt]{\lineheight{1.25}\smash{\begin{tabular}[t]{l}$Q_3$\end{tabular}}}}%
    \put(0.50827988,0.48500322){\color[rgb]{0,0,0}\makebox(0,0)[lt]{\lineheight{1.25}\smash{\begin{tabular}[t]{l}$Q_2$\end{tabular}}}}%
    \put(0.34431309,0.29409299){\color[rgb]{0,0,0}\makebox(0,0)[lt]{\lineheight{1.25}\smash{\begin{tabular}[t]{l}$P_1$\end{tabular}}}}%
    \put(0.48857738,0.43159969){\color[rgb]{0,0,0}\makebox(0,0)[lt]{\lineheight{1.25}\smash{\begin{tabular}[t]{l}$P_2$\end{tabular}}}}%
    \put(0.625886,0.29375818){\color[rgb]{0,0,0}\makebox(0,0)[lt]{\lineheight{1.25}\smash{\begin{tabular}[t]{l}$P_3$\end{tabular}}}}%
    \put(0.48670528,0.149159){\color[rgb]{0,0,0}\makebox(0,0)[lt]{\lineheight{1.25}\smash{\begin{tabular}[t]{l}$P_4$\end{tabular}}}}%
  \end{picture}%
\endgroup%

%% file: figures/12-faces2.pdf_tex
%% Creator: Inkscape 1.2 (dc2aeda, 2022-05-15), www.inkscape.org
%% Accompanies image file '12-faces2.pdf' (pdf, eps, ps)
\begingroup%
  \makeatletter%
  \providecommand\color[2][]{%
    \errmessage{(Inkscape) Color is used for the text in Inkscape, but the package 'color.sty' is not loaded}%
    \renewcommand\color[2][]{}%
  }%
  \providecommand\transparent[1]{%
    \errmessage{(Inkscape) Transparency is used (non-zero) for the text in Inkscape, but the package 'transparent.sty' is not loaded}%
    \renewcommand\transparent[1]{}%
  }%
  \providecommand\rotatebox[2]{#2}%
  \newcommand*\fsize{\dimexpr\f@size pt\relax}%
  \newcommand*\lineheight[1]{\fontsize{\fsize}{#1\fsize}\selectfont}%
  \ifx\svgwidth\undefined%
    \setlength{\unitlength}{192.37048382bp}%
    \ifx\svgscale\undefined%
      \relax%
    \else%
      \setlength{\unitlength}{\unitlength * \real{\svgscale}}%
    \fi%
  \else%
    \setlength{\unitlength}{\svgwidth}%
  \fi%
  \global\let\svgwidth\undefined%
  \global\let\svgscale\undefined%
  \makeatother%
  \begin{picture}(1,1.02322838)%
    \lineheight{1}%
    \setlength\tabcolsep{0pt}%
    \put(0,0){\includegraphics[width=\unitlength,page=1]{12-faces2.pdf}}%
  \end{picture}%
\endgroup%

%% file: figures/10faces.pdf_tex
%% Creator: Inkscape 1.0 (4035a4f, 2020-05-01), www.inkscape.org
%% Accompanies image file '10faces.pdf' (pdf, eps, ps)
\begingroup%
  \makeatletter%
  \providecommand\color[2][]{%
    \errmessage{(Inkscape) Color is used for the text in Inkscape, but the package 'color.sty' is not loaded}%
    \renewcommand\color[2][]{}%
  }%
  \providecommand\transparent[1]{%
    \errmessage{(Inkscape) Transparency is used (non-zero) for the text in Inkscape, but the package 'transparent.sty' is not loaded}%
    \renewcommand\transparent[1]{}%
  }%
  \providecommand\rotatebox[2]{#2}%
  \newcommand*\fsize{\dimexpr\f@size pt\relax}%
  \newcommand*\lineheight[1]{\fontsize{\fsize}{#1\fsize}\selectfont}%
  \ifx\svgwidth\undefined%
    \setlength{\unitlength}{1248.17695918bp}%
    \ifx\svgscale\undefined%
      \relax%
    \else%
      \setlength{\unitlength}{\unitlength * \real{\svgscale}}%
    \fi%
  \else%
    \setlength{\unitlength}{\svgwidth}%
  \fi%
  \global\let\svgwidth\undefined%
  \global\let\svgscale\undefined%
  \makeatother%
  \begin{picture}(1,1.0362173)%
    \lineheight{1}%
    \setlength\tabcolsep{0pt}%
    \put(0,0){\includegraphics[width=\unitlength,page=1]{10faces.pdf}}%
  \end{picture}%
\endgroup%